\documentclass[a4paper,12pt]{article}
\usepackage{ucs}
\usepackage{amsfonts, amssymb, amsmath, amsthm, amscd}
\usepackage{graphicx}
\usepackage{cite}
\ifx\undefined \pdfgentounicode \else
\input{glyphtounicode} \pdfgentounicode=1
\fi

\author{A.A. Vasil'eva\footnote{Faculty of mechanics and mathematics, Lomonosov Moscow State University; Moscow Center for Fundamental and Applied Mathematics}}
\title{Kolmogorov widths of an intersection of a family of balls in a mixed norm\footnote{This research was conducted at Lomonosov Moscow State University with the support of the Russian Science Foundation (grant no.\ 22-21-00204).}}
\date{}
\begin{document}

\maketitle

\newenvironment{Biblio}{%
                  \renewcommand{\refname}{\footnotesize REFERENCES}%
                  \begin{thebibliography}{99}%
                  \itemsep -0.2em}%
                  {\end{thebibliography}}

\def\inff{\mathop{\smash\inf\vphantom\sup}}
\renewcommand{\le}{\leqslant}
\renewcommand{\ge}{\geqslant}
\newcommand{\sgn}{\mathrm {sgn}\,}
\newcommand{\inter}{\mathrm {int}\,}
\newcommand{\dist}{\mathrm {dist}}
\newcommand{\supp}{\mathrm {supp}\,}
\newcommand{\R}{\mathbb{R}}
\newcommand{\Z}{\mathbb{Z}}
\newcommand{\N}{\mathbb{N}}
\newcommand{\Q}{\mathbb{Q}}
\theoremstyle{plain}
\newtheorem{Trm}{Theorem}
\newtheorem{trma}{Theorem}

\newtheorem{Def}{Definition}
\newtheorem{Cor}{Corollary}
\newtheorem{Lem}{Lemma}
\newtheorem{Rem}{Remark}
\newtheorem{Sta}{Proposition}

\newtheorem{Exa}{Example}
\renewcommand{\proofname}{\bf Proof}
\renewcommand{\thetrma}{\Alph{trma}}

\begin{abstract}
In this paper, order estimates for the Kolmogorov $n$-widths of an intersection of a family of balls in a mixed norm in the space $l^{m,k}_{q,\sigma}$ with $2\le q, \, \sigma <\infty$, $n\le mk/2$ are obtained.
\end{abstract}

\section{Introduction}

In this paper a problem of the Kolmogorov widths of an intersection of an arbitrary family of finite-dimensional balls  in a mixed norm is studied. The obtained estimates can be applyed in solving the problem of the widths of an intersection of weighted Besov classes with a strong singularity at a point, or an intersection of Besov classes  with dominating mixed smoothness. For a single Besov class, the problem was considered in \cite{galeev4, dir_ull, vyb_06, vas_besov}.

First we give some definitions.

Let $N\in \N$, $1\le s\le \infty$, $(x_i)_{i=1}^N\in \R^N$. We set $\|(x_i)_{i=1}^N\|_{l_s^N} = \left(\sum \limits _{i=1}^N |x_i|^s\right)^{1/s}$ for $s<\infty$, and $\|(x_i)_{i=1}^N\|_{l_s^N} = \max _{1\le i\le N}|x_i|$ for $s=\infty$. The space $\R^N$ equipped with this norm we denote by $l_s^N$; by $B_s^N$, we denote the unit ball in $l_s^N$.

Let now $m, \, k\in \N$, $1\le p\le \infty$, $1\le \theta\le \infty$. By $l_{p,\theta}^{m,k}$ we denote the space $\R^{mk}= \{(x_{i,j})_{1\le i\le m, \, 1\le j\le k}:\; x_{i,j}\in \R\}$ with the norm
$$
\|(x_{i,j})_{1\le i\le m, \, 1\le j\le k}\|_{l_{p,\theta}^{m,k}} = \left\|\bigl(\|(x_{i,j})_{i=1}^m\|_{l_p^m}\bigr)_{j=1}^k\right\|_{l_\theta^k}.
$$
By $B_{p,\theta}^{m,k}$ we denote the unit ball of $l_{p,\theta}^{m,k}$.

Let $X$ be a normed space, let $M\subset X$, $n\in \Z_+$, and let ${\cal L}_n(X)$ be the family of all subspaces in $X$ of dimension at most $n$. The Kolmogorov $n$-width of the set $M$ in the space $X$ is defined by
$$
d_n(M, \, X) = \inf _{L\in {\cal L}_n(X)} \sup _{x\in M} \inf
_{y\in L} \|x-y\|.
$$

The problem of estimating the widths $d_n(B_p^m, \, l_q^m)$ was studied in \cite{k_p_s, stech_poper, pietsch1, stesin, gluskin1, bib_gluskin, kashin_oct, bib_kashin, garn_glus}. For $p\ge q$ and for $p=1$, $q=2$ the exact values were found; for $p<q<\infty$ and for $q=\infty$, $p\ge 2$ order estimates were obtained. For $q=\infty$, $1\le p<2$ the estimates are known up to a factor, which is a degree of $\log \left(\frac{em}{n}\right)$. For details, see \cite{itogi_nt, kniga_pinkusa, nvtp}.

The problem of estimating the widths $d_n(B^{m,k}_{p,\theta}, \, l^{m,k}_{q,\sigma})$ was investigated in \cite{galeev2, galeev5, izaak1, izaak2, mal_rjut, vas_besov, dir_ull}. The order estimates for $n \le \frac{mk}{2}$ are known in the following cases:
\begin{enumerate}
\item $p=1$, $\theta=\infty$, $q=2$, $1<\sigma <\infty$ \cite{galeev2};
\item $p=1$ or $p=\infty$; $\theta=\infty$; here one of the following conditions holds: a) $q=2$, $1<\sigma \le \infty$, or b) $1<q\le \min \{2, \, \sigma\}$ \cite{galeev5};
\item $p=\theta$, $q=2$, $\sigma=1$; here $p=1$ or $2\le p\le \infty$ \cite{izaak2};
\item $2\le q<\infty$, $2\le \sigma <\infty$, $1\le p\le q$, $1\le \theta \le \sigma$ (see \cite{vas_besov} for $n\le a(q, \, \sigma)mk$, where $a(q, \, \sigma)>0$, and \cite{vas_mix2}, for $a(q, \, \sigma)mk \le n\le \frac{mk}{2}$);
\item a) $p=1$, $\theta=\infty$, $q=2$, $\sigma =1$; b) $p\le q\le 2$, $\theta \ge \sigma$ \cite{mal_rjut};
\item a) $p=q=2$, $\theta\ge 2$, $\sigma=\infty$; b) $p=\theta=\sigma\ge 2$, $q=\infty$ \cite{dir_ull}; the result is formulated in terms of the dual quantities (the Gelfand widths).
\end{enumerate}
In addition, in \cite{galeev6} the lower estimate of $d_n(B^{m,k}_{p,\infty}, \, l^{m,k}_{q, \, q})$ was obtained for $2\le q<\infty$, $n\le c(q)mk$, where $c(q)$ is a positive number.

The problem of estimating the widths $d_n(\cap _{\alpha \in A} \nu_\alpha B^m_{p_\alpha}, \, l_q^m)$ was studied in \cite{galeev1, vas_ball_inters} (here $p_\alpha\in [1, \, \infty]$, $\nu_\alpha >0$, $\alpha\in A$, $A\ne \varnothing$). In \cite{galeev1} the order estimates were obtained for $n= \frac{m}{2}$, and in \cite{vas_ball_inters}, for $n \le \frac{m}{2}$. These results were applyed to estimating the widths of an intersection of Sobolev classes (see, e.g., \cite{galeev2, vas_int_sob}).

Let $A$ be a nonempty set, and let $p_\alpha \in [1, \infty]$, $\theta_\alpha \in [1, \, \infty]$, $\nu_\alpha>0$ be given for each $\alpha \in A$; we suppose that $(p_\alpha, \, \theta_\alpha) \ne (p_\beta, \, \theta_\beta)$ for $\alpha \ne \beta$. Let
\begin{align}
\label{m_def} M = \cap _{\alpha\in A} \nu_\alpha B^{m,k}_{p_\alpha, \, \theta_\alpha}.
\end{align}

In this paper order estimates of $d_n(M, \, l_{q,\sigma}^{m,k})$ for $2\le q<\infty$, $2\le \sigma <\infty$, $n \le \frac{mk}{2}$ are obtained. We need some more notation.

Let $X$, $Y$ be sets, and let $f_1$, $f_2:\ X\times Y\rightarrow \mathbb{R}_+$.
We write $f_1(x, \, y)\underset{y}{\lesssim} f_2(x, \, y)$ (or
$f_2(x, \, y)\underset{y}{\gtrsim} f_1(x, \, y)$) if, for each 
$y\in Y$, there exists $c(y)>0$ such that $f_1(x, \, y)\le
c(y)f_2(x, \, y)$ for all $x\in X$; $f_1(x, \,
y)\underset{y}{\asymp} f_2(x, \, y)$ if $f_1(x, \, y)
\underset{y}{\lesssim} f_2(x, \, y)$ and $f_2(x, \,
y)\underset{y}{\lesssim} f_1(x, \, y)$.

Given $1\le p\le q$, we set 
\begin{align}
\label{om_pq} \omega_{p,q} =\begin{cases} \min \left\{\frac{1/p-1/q}{1/2-1/q}, \, 1 \right\} & \text{if } q>2 \\ 1, & \text{if } q=2.\end{cases}
\end{align}

We define the values $\Phi(p, \, \theta)=\Phi(p, \, \theta; \, q, \, \sigma, \, m, \, k, \, n)$ as follows:
\begin{enumerate}
\item for $p\ge q$, $\theta \ge \sigma$,
\begin{align}
\label{phi1} \Phi(p, \, \theta) = m^{1/q-1/p}k^{1/\sigma-1/\theta};
\end{align}
\item for $p\ge q$, $\theta\le \sigma$,
\begin{align}
\label{phi2} \Phi(p, \, \theta) = \min \left\{m^{1/q-1/p}, \, m^{1/q-1/p}(n^{-\frac 12}m^{\frac 12} k^{\frac{1}{\sigma}}) ^{\omega_{\theta,\sigma}}\right\};
\end{align}

\item for $\theta\ge \sigma$, $p\le q$,
\begin{align}
\label{phi3} \Phi(p, \, \theta) = \min \left\{k^{1/\sigma-1/\theta}, \, k^{1/\sigma-1/\theta}(n^{-\frac 12}m^{\frac 1q} k^{\frac{1}{2}}) ^{\omega_{p,q}}\right\};
\end{align}

\item for $2\le p\le q$, $1\le \theta \le \sigma$, $\omega_{p,q}\le \omega_{\theta,\sigma}$,
\begin{align}
\label{phi4}
\Phi(p, \, \theta) = \min \left\{1, \, (n^{-\frac 12}m^{\frac 1q} k^{\frac{1}{\sigma}})^{\omega_{p,q}}, \, m^{1/q-1/p}(n^{-\frac 12}m^{\frac 12} k^{\frac{1}{\sigma}})^{\omega_{\theta,\sigma}}\right\};
\end{align}

\item for $2\le \theta\le \sigma$, $1\le p \le q$, $\omega_{\theta,\sigma}\le \omega_{p,q}$,
\begin{align}
\label{phi5}
\Phi(p, \, \theta) = \min \left\{1, \, (n^{-\frac 12}m^{\frac 1q} k^{\frac{1}{\sigma}})^{\omega_{\theta,\sigma}}, \, k^{1/\sigma-1/\theta}(n^{-\frac 12}m^{\frac 1q} k^{\frac{1}{2}})^{\omega_{p,q}}\right\};
\end{align}
\item for $1\le p\le 2$, $1\le \theta\le 2$,
\begin{align}
\label{phi6}
\Phi(p, \, \theta) = \min \{1, \, n^{-\frac 12}m^{\frac 1q} k^{\frac{1}{\sigma}}\}.
\end{align}
\end{enumerate}

We claim that if $2\le q<\infty$, $2\le \sigma <\infty$, $n\le mk/2$, then
\begin{align}
\label{d_n_bmkpt_kolm}
d_n(B^{m,k}_{p,\theta}, \, l^{m,k}_{q,\sigma}) \underset{q,\sigma}{\asymp} \Phi(p, \, \theta).
\end{align}
For $1\le p\le q$, $1\le \theta\le \sigma$, it was proved in \cite{vas_besov} for $n\le a(q, \, \sigma)mk$, and in \cite{vas_mix2}, for $a(q, \, \sigma)mk \le n \le mk/2$; for the other cases it will be proved in \S 2 (see Propositions \ref{pgq} and \ref{est_dn}).

Given $\alpha$, $\beta$, $\gamma \in A$, by $\Delta_{\alpha,\beta,\gamma}$ we denote the triangle with the vertices $(1/p_\alpha, \, 1/\theta_\alpha)$, $(1/p_\beta, \, 1/\theta_\beta)$, $(1/p_\gamma, \, 1/\theta_\gamma)$. We write $\Delta_{\alpha,\beta,\gamma} \in {\cal R}$ if the vertices do not lie on the same line.

We set $P=\{(p_\alpha, \, \theta_\alpha)\}_{\alpha\in A}$. The function $\nu:A \rightarrow (0, \, \infty)$ is defined as $\alpha \stackrel{\nu}{\mapsto}\nu_\alpha$. The sets ${\cal N}_j={\cal N}_j(P)={\cal N}_j(P; \, q, \, \sigma)$ ($1\le j\le 7$) and the values $\Psi_j=\Psi_j(P)= \Psi_j(P; \, q, \, \sigma; \, m, \, k, \, n; \, \nu)$ ($0\le j\le 7$) are defined as follows:
\begin{align}
\label{n1} \begin{array}{c} {\cal N}_1= \left\{ (\alpha, \, \beta)\in A\times A:\;p_\alpha \ne q, \; \exists \hat \lambda _{\alpha,\beta} \in (0, \, 1):\; \frac 1q = \frac{1-\hat \lambda _{\alpha,\beta}}{p_\alpha} + \frac{\hat \lambda _{\alpha,\beta}}{p_\beta}\right\}, \\ \frac{1}{\hat \theta_{\alpha,\beta}} := \frac{1-\hat \lambda _{\alpha,\beta}}{\theta_\alpha} + \frac{\hat \lambda _{\alpha,\beta}}{\theta_\beta}, \quad (\alpha, \, \beta) \in {\cal N}_1, \end{array}
\end{align}
\begin{align}
\label{n2} \begin{array}{c}{\cal N}_2 = \left\{ (\alpha, \, \beta)\in A\times A:\;\theta_\alpha \ne \sigma, \; \exists \hat \mu _{\alpha,\beta} \in (0, \, 1):\; \frac{1}{\sigma} = \frac{1-\hat \mu _{\alpha,\beta}}{\theta_\alpha} + \frac{\hat \mu _{\alpha,\beta}}{\theta_\beta}\right\}, \\ \frac{1}{\hat p_{\alpha,\beta}} := \frac{1-\hat \mu _{\alpha,\beta}}{p_\alpha} + \frac{\hat \mu _{\alpha,\beta}}{p_\beta}, \quad (\alpha, \, \beta) \in {\cal N}_2, \end{array}
\end{align}
\begin{align}
\label{n3} \begin{array}{c}{\cal N}_3 = \left\{ (\alpha, \, \beta)\in A\times A:\; p_\alpha \ne 2, \; \exists \tilde \lambda _{\alpha,\beta} \in (0, \, 1):\; \frac 12 = \frac{1-\tilde \lambda _{\alpha,\beta}}{p_\alpha} + \frac{\tilde \lambda _{\alpha,\beta}}{p_\beta}\right\}, \\ \frac{1}{\tilde \theta_{\alpha,\beta}} := \frac{1-\tilde \lambda _{\alpha,\beta}}{\theta_\alpha} + \frac{\tilde \lambda _{\alpha,\beta}}{\theta_\beta}, \quad (\alpha, \, \beta) \in {\cal N}_3, \end{array}
\end{align}
\begin{align}
\label{n4} \begin{array}{c}{\cal N}_4 = \left\{ (\alpha, \, \beta)\in A\times A:\;\theta_\alpha \ne 2, \; \exists \tilde \mu _{\alpha,\beta} \in (0, \, 1):\; \frac{1}{2} = \frac{1-\tilde \mu _{\alpha,\beta}}{\theta_\alpha} + \frac{\tilde \mu _{\alpha,\beta}}{\theta_\beta}\right\}, \\ \frac{1}{\tilde p_{\alpha,\beta}} := \frac{1-\tilde \mu _{\alpha,\beta}}{p_\alpha} + \frac{\tilde \mu _{\alpha,\beta}}{p_\beta},\quad (\alpha, \, \beta) \in {\cal N}_4,  \end{array}
\end{align}
\begin{align}
\label{n5} \begin{array}{c}{\cal N}_5 = \left\{ (\alpha, \, \beta)\in A\times A:\; \exists \lambda _{\alpha,\beta} \in (0, \, 1), \; p_{\alpha,\beta}\in (2, \, q), \; \theta_{\alpha,\beta} \in (2, \, \sigma):\right. \\ \left. \frac{1}{p_{\alpha,\beta}} = \frac{1-\lambda _{\alpha,\beta}}{p_\alpha} + \frac{\lambda _{\alpha,\beta}}{p_\beta}, \; \frac{1}{\theta_{\alpha,\beta}} = \frac{1-\lambda _{\alpha,\beta}}{\theta_\alpha} + \frac{\lambda _{\alpha,\beta}}{\theta_\beta}, \; \frac{1/p_{\alpha,\beta}-1/q}{1/2-1/q}=\frac{1/\theta_{\alpha,\beta}-1/\sigma}{1/2-1/
\sigma}, \right. \\ \left. \text{and}\; \frac{1/p_{\alpha}-1/q}{1/2-1/q}\ne \frac{1/\theta_{\alpha}-1/\sigma}{1/2-1/
\sigma}\right\}, \end{array}
\end{align}
\begin{align}
\label{n6} \begin{array}{c}{\cal N}_6 = \left\{ (\alpha, \, \beta, \, \gamma)\in A\times A \times A:\; \exists \tau_\alpha, \, \tau_\beta, \, \tau_\gamma > 0:\; \tau_\alpha+\tau_\beta+\tau_\gamma =1, \right. \\ \left.\frac{1}{q} = \frac{\tau_\alpha}{p_\alpha} + \frac{\tau_\beta}{p_\beta}+\frac{\tau_\gamma}{p_\gamma}, \; \frac{1}{\sigma} = \frac{\tau_\alpha}{\theta_\alpha} + \frac{\tau_\beta}{\theta_\beta}+\frac{\tau_\gamma}{\theta_\gamma}, \; \Delta _{\alpha,\beta,\gamma}\in {\cal R}\right\},\end{array}
\end{align}
\begin{align}
\label{n7} \begin{array}{c}{\cal N}_7 = \left\{ (\alpha, \, \beta, \, \gamma)\in A\times A \times A:\; \exists \overline{\tau}_\alpha, \, \overline{\tau}_\beta, \, \overline{\tau}_\gamma > 0:\; \overline{\tau}_\alpha+\overline{\tau}_\beta+\overline{\tau}_\gamma =1, \right. \\ \left.\frac{1}{2} = \frac{\overline{\tau}_\alpha}{p_\alpha} + \frac{\overline{\tau}_\beta}{p_\beta}+\frac{\overline{\tau}_\gamma}{p_\gamma}, \; \frac{1}{2} = \frac{\overline{\tau}_\alpha}{\theta_\alpha} + \frac{\overline{\tau}_\beta}{\theta_\beta}+\frac{\overline{\tau}_\gamma}{\theta_\gamma}, \; \Delta _{\alpha,\beta,\gamma}\in {\cal R}\right\},\end{array}
\end{align}
\begin{align}
\label{psi0} \Psi_0 =\Psi_0(P) = \inf _{\alpha \in A} \nu_\alpha \Phi(p_\alpha, \, \theta _\alpha),
\end{align}
\begin{align}
\label{psi1} \Psi_1=\Psi_1(P) = \inf _{(\alpha,\, \beta) \in {\cal N}_1} \nu_\alpha^{1-\hat \lambda_{\alpha,\beta}} \nu_\beta ^{\hat \lambda_{\alpha,\beta}} \Phi(q, \, \hat\theta _{\alpha,\beta}),
\end{align}
\begin{align}
\label{psi2} \Psi_2=\Psi_2(P) = \inf _{(\alpha,\, \beta) \in {\cal N}_2} \nu_\alpha^{1-\hat \mu_{\alpha,\beta}} \nu_\beta ^{\hat \mu_{\alpha,\beta}} \Phi(\hat p_{\alpha,\beta}, \, \sigma),
\end{align}
\begin{align}
\label{psi3} \Psi_3 =\Psi_3(P)= \inf _{(\alpha,\, \beta) \in {\cal N}_3} \nu_\alpha^{1-\tilde \lambda_{\alpha,\beta}} \nu_\beta ^{\tilde \lambda_{\alpha,\beta}} \Phi(2, \, \tilde\theta _{\alpha,\beta}),
\end{align}
\begin{align}
\label{psi4} \Psi_4 =\Psi_4(P)= \inf _{(\alpha,\, \beta) \in {\cal N}_4} \nu_\alpha^{1-\tilde \mu_{\alpha,\beta}} \nu_\beta ^{\tilde \mu_{\alpha,\beta}} \Phi(\tilde p_{\alpha,\beta}, \, 2),
\end{align}
\begin{align}
\label{psi5} \Psi_5 =\Psi_5(P)= \inf _{(\alpha,\, \beta) \in {\cal N}_5} \nu_\alpha^{1-\lambda_{\alpha,\beta}} \nu_\beta ^{\lambda_{\alpha,\beta}} \Phi(p_{\alpha,\beta}, \, \theta_{\alpha,\beta}),
\end{align}
\begin{align}
\label{psi6} \Psi_6 =\Psi_6(P)= \inf _{(\alpha,\, \beta,\, \gamma) \in {\cal N}_6} \nu_\alpha^{\tau_\alpha} \nu_\beta ^{\tau_\beta} \nu_\gamma^{\tau_\gamma} \Phi(q, \, \sigma),
\end{align}
\begin{align}
\label{psi7} \Psi_7 =\Psi_7(P)= \inf _{(\alpha,\, \beta,\, \gamma) \in {\cal N}_7} \nu_\alpha^{\overline{\tau}_\alpha} \nu_\beta ^{\overline{\tau}_\beta} \nu_\gamma^{\overline{\tau}_\gamma} \Phi(2, \, 2)
\end{align}
(the infimum of the empty set is $+\infty$); here the values $\hat \theta_{\alpha,\beta}$, $\hat \lambda_{\alpha,\beta}$, etc., are defined in (\ref{n1})--(\ref{n7}).

\begin{Trm}
\label{main} Let $2\le q<\infty$, $2\le \sigma<\infty$, $m$, $k\in \N$, $n\in \Z_+$, $n \le \frac{mk}{2}$, and let the set $M$ be defined by \eqref{m_def}. Then
\begin{align}
\label{dn_main}
d_n(M, \, l^{m,k}_{q,\sigma}) \underset{q,\sigma}{\asymp} \min _{0\le j\le 7} \Psi_j(P; \, q, \, \sigma; \, m, \, k, \, n; \, \nu).
\end{align}
\end{Trm}
Hence (see \eqref{d_n_bmkpt_kolm}), the estimating of the Kolmogorov widths of the intersection of a family of the balls can be reduced to estimating the infimum of a family of values $d_n(\nu B^{m,k}_{p,\theta}, \, l^{m,k}_{q,\sigma})$.

The particular case for an intersection of two balls and $1\le p_\alpha\le q$, $1\le \theta_\alpha\le \sigma$, was considered in \cite{vas_mix2}.

The paper is organized as follows. In \S 2, the well-known results are formulated, auxiliary assertions are proved an the upper estimate for $d_n(M, \, l^{m,k}_{q,\sigma})$ is obtained. In \S 3, the lower estimate is obtained in the case of a finite set $A$, when the points $\{(1/p_\alpha, \, 1/\theta_\alpha)\}_{\alpha \in A}$ are in a general position. In \S 4, the lower estimate is proved for a finite set $A$ and arbitrary $\{(1/p_\alpha, \, 1/\theta_\alpha)\}_{\alpha \in A}$. In \S 5, the estimate is proved for an arbitrary nonempty set $A$.

\section{Auxiliary assertions and proof of the upper estimate}

Let $m$, $k$, $r$, $l\in \N$, $1\le r\le m$, $1\le l\le k$. We set $$G=\{(\tau_1, \, \tau_2, \, \varepsilon_1, \, \varepsilon_2):\; \tau_1\in S_m, \, \tau_2\in S_k, \, \varepsilon_1\in \{1, \, -1\}^m, \, \varepsilon_2\in \{1, \, -1\}^k\},$$
where $S_m$ and $S_k$ are the groups of permutations of $m$ and $k$ elements, respectively. Given $x=(x_{i,j})_{1\le i\le m, \, 1\le j\le k}\in \R^{mk}$, $\gamma = (\tau_1, \, \tau_2, \, \varepsilon_1, \, \varepsilon_2)\in G$, $\varepsilon_1=(\varepsilon_{1,i})_{1\le i\le m}$, $\varepsilon_2=(\varepsilon_{2,j})_{1\le j\le k}$, we set 
\begin{align}
\label{gamma_x_def}
\gamma(x) = (\varepsilon_{1,i}\varepsilon_{2,j}x_{\tau_1(i)\tau_2(j)})_{1\le i\le m, \, 1\le j\le k}.
\end{align}

We denote $e=(e_{i,j}^{m,k,r,l})_{1\le i\le m, \, 1\le j\le k}$, where
\begin{align}
\label{eij_kl}
e_{i,j}^{m,k,r,l} = \begin{cases} 1 & \text{if } 1\le i\le r, \; 1\le j\le l, \\ 0 & \text{otherwise},\end{cases}
\end{align}
\begin{align}
\label{vrl_km}
V_{r,l}^{m,k} = {\rm conv}\{\gamma(e):\; \gamma\in G\}.
\end{align}

The following assertion was proved in \cite{vas_besov} for $n\le a(q, \, \sigma) m^{\frac 2q}k^{\frac{2}{\sigma}}r^{1-\frac 2q}l^{1-\frac{2}{\sigma}}$, and in \cite{vas_mix2}, for $a(q, \, \sigma) m^{\frac 2q}k^{\frac{2}{\sigma}}r^{1-\frac 2q}l^{1-\frac{2}{\sigma}}\le n \le \frac{mk}{2}$ (here $a(q, \, \sigma)>0$ is nonincreasing with $q$ and $\sigma$).
\begin{trma}
\label{v_dn} {\rm (see \cite{vas_mix2, vas_besov}).}
Let $2\le q<\infty$, $2\le \sigma <\infty$, $m$, $k$, $r$, $l\in \N$, $1\le r\le m$, $1\le l \le k$, $n\in \Z_+$, $n\le \frac{mk}{2}$. Then
\begin{align}
\label{dn_vmk1} d_n(V^{m,k}_{r,l}, \, l^{m,k}_{q,\sigma}) \underset{q,\sigma}{\gtrsim} \begin{cases} r^{\frac 1q}l^{\frac{1}{\sigma}} & \text{if}\; n\le  m^{\frac 2q}k^{\frac{2}{\sigma}}r^{1-\frac 2q} l^{1-\frac{2}{\sigma}}, \\ n^{-\frac 12}m^{\frac 1q}k^{\frac{1}{\sigma}} r^{\frac 12} l^{\frac 12} & \text{if}\; n\ge  m^{\frac 2q}k^{\frac{2}{\sigma}}r^{1-\frac 2q} l^{1-\frac{2}{\sigma}}.\end{cases}
\end{align}
\end{trma}

The sketch of proof for $a(q, \, \sigma) m^{\frac 2q}k^{\frac{2}{\sigma}}r^{1-\frac 2q}l^{1-\frac{2}{\sigma}}\le n \le \frac{mk}{2}$ is as follows. It suffices to obtain the lower estimate. If $a(q, \, \sigma) m^{\frac 2q}k^{\frac{2}{\sigma}}r^{1-\frac 2q}l^{1-\frac{2}{\sigma}}\le n \le a(q, \, \sigma)mk$, then we choose the numbers $\tilde q\in [2, \, q]$ and $\tilde \sigma \in [2, \, \sigma]$ such that $n=a(q, \, \sigma) m^{\frac{2}{\tilde q}}k^{\frac{2}{\tilde \sigma}}r^{1-\frac{2}{\tilde q}}l^{1-\frac{2}{\tilde \sigma}}$, and apply the estimate from \cite{vas_besov} for $d_n(V^{m,k}_{r,l}, \, l^{m,k}_{\tilde q,\tilde\sigma})$. For $a(q, \, \sigma)mk \le n \le \frac{mk}{2}$, it suffices to prove that $d_n(V^{m,k}_{r,l}, \, l^{m,k}_{2,2})\gtrsim r^{1/2}l^{1/2}$. To this end, we repeat the arguments from \cite[pp. 14--17]{vas_besov} for the particular case $q=\sigma=2$.

The following result was obtained in \cite{vas_besov} for $n\le a(q, \, \sigma)mk$ and in \cite{vas_mix2} for $a(q, \, \sigma)mk\le n \le mk/2$. Notice that in \cite{vas_besov} the constants in the order equality depended on $p$ and $\theta$, but the proof shows that they are independent of these parameters.
\begin{trma}
\label{1mixed}
{\rm (see \cite{vas_mix2, vas_besov}).} Let $2\le q<\infty$, $2\le \sigma < \infty$, $1\le p\le q$, $1\le \theta \le \sigma$, $n\le \frac{mk}{2}$, the value $\Phi(p, \, \theta)$ is defined by \eqref{phi4}--\eqref{phi6}. Then 
$$d_n(B^{m,k}_{p,\theta}, \, l^{m,k}_{q,\sigma}) \underset{q,\sigma}{\asymp} \Phi(p, \, \theta).$$
\end{trma}

For $a(q, \, \sigma)mk\le n \le mk/2$, the idea of the proof of the lower estimate is as follows: we apply (\ref{dn_vmk1}) and the inclusions $r^{-1/p}l^{-1/\theta}V^{m,k}_{r,l} \subset B^{m,k}_{p,\theta}$ for $r=1$ or $r=m$, $l=1$ or $l=k$. The upper estimate holds for all $n \le mk$ (see \cite{vas_besov}).

It turns out that the conditions $p\le q$, $\theta\le \sigma$ can be eliminated. Moreover, for $p\ge q$, $\theta \ge \sigma$, the conditions $2\le q<\infty$, $2\le \sigma <\infty$ can be also eliminated.

\begin{Sta}
\label{pgq} Let $1\le q\le p\le \infty$, $1\le \sigma \le \theta \le \infty$, $n\le \frac{mk}{2}$. Then
$$
d_n(B^{m,k}_{p,\theta}, \, l^{m,k}_{q,\sigma}) \asymp m^{\frac 1q -\frac 1p} k^{\frac{1}{\sigma} -\frac{1}{\theta}}.
$$
\end{Sta}
\begin{proof}
The upper estimate follows from the inclusion $B^{m,k}_{p,\theta} \subset m^{\frac 1q -\frac 1p} k^{\frac{1}{\sigma} -\frac{1}{\theta}}B^{m,k}_{q,\sigma}$. 

In order to obtain the lower estimate, we use the inclusion $B^{m,k}_{p,\theta} \supset m^{-\frac 1p} k^{-\frac{1}{\theta}} B^{m,k}_{\infty,\infty}$ and the well-known equality $d_n(B^N_\infty, \, l_s^N) = (N-n)^{1/s}$ (see \cite{pietsch1, stesin}). If $q\le \sigma$, then
$$
d_n(B^{m,k}_{\infty,\infty}, \, l^{m,k}_{q,\sigma}) \ge k^{\frac{1}{\sigma}-\frac 1q}d_n(B^{m,k}_{\infty,\infty}, \, l^{m,k}_{q,q}) \asymp k^{\frac{1}{\sigma}-\frac 1q} (mk)^{\frac 1q} = m^{\frac 1q} k^{\frac{1}{\sigma}}.
$$
The case $q\ge \sigma$ is similar.
\end{proof}

\begin{Sta}
\label{est_dn} Let $2\le q<\infty$, $2\le \sigma < \infty$, $n\le \frac{mk}{2}$. Suppose that one of the following conditions holds: 1) $p\ge q$, $\theta \le \sigma$, 2) $p\le q$, $\theta \ge \sigma$. The value $\Phi(p, \, \theta)$ is defined by \eqref{phi2}, \eqref{phi3}. Then 
$$d_n(B^{m,k}_{p,\theta}, \, l^{m,k}_{q,\sigma}) \underset{q,\sigma}{\asymp} \Phi(p, \, \theta).$$
\end{Sta}
\begin{proof}
Consider case 1 (case 2 is similar). The upper estimate follows from the inclusion $B^{m,k}_{p,\theta} \subset m^{\frac 1q -\frac 1p}B^{m,k}_{q,\theta}$, Theorem \ref{1mixed} and from (\ref{om_pq}), (\ref{phi2}), (\ref{phi4}).

Let us prove the lower estimate. For $n \le mk^{2/\sigma}$, we check the inequality $$d_n(B^{m,k}_{p,\theta}, \, l^{m,k}_{q,\sigma}) \underset{q,\sigma}{\gtrsim} m^{\frac 1q-\frac 1p}.$$ To this end, we use the inclusion $m^{-\frac 1p}V^{m,k}_{m,1} \subset B^{m,k}_{p,\theta}$ (see (\ref{gamma_x_def})--(\ref{vrl_km})), Theorem \ref{v_dn} and the inequality $n \le m^{\frac 2q}k^{\frac{2}{\sigma}} m^{1-\frac 2q}$.

Let $n \ge mk^{2/\sigma}$. Then $\sigma>2$. If $\theta \le 2$, then (see \eqref{om_pq}, \eqref{phi2}) we check the inequality $$d_n(B^{m,k}_{p,\theta}, \, l^{m,k}_{q,\sigma}) \underset{q,\sigma}{\gtrsim} m^{\frac 1q-\frac 1p}n^{-\frac 12}m^{\frac 12}k^{\frac{1}{\sigma}}.$$ As in the previous case, we use the inclusion $m^{-\frac 1p}V^{m,k}_{m,1} \subset B^{m,k}_{p,\theta}$ and Theorem \ref{v_dn} (now, for $n \ge m^{\frac 2q}k^{\frac{2}{\sigma}} m^{1-\frac 2q}$).

Let $2< \theta \le \sigma$. We prove the inequality (see \eqref{om_pq}, \eqref{phi2})
$$d_n(B^{m,k}_{p,\theta}, \, l^{m,k}_{q,\sigma}) \underset{q,\sigma}{\gtrsim} m^{\frac 1q-\frac 1p}\left(n^{-\frac 12}m^{\frac 12}k^{\frac{1}{\sigma}}\right)^{\frac{1/\theta-1/\sigma}{1/2-1/\sigma}}.$$
We set $l = \left\lceil \left(n^{\frac 12}m^{-\frac 12}k^{-\frac{1}{\sigma}}\right)^{\frac{1}{1/2-1/\sigma}}\right\rceil$. Since $mk^{2/\sigma} \le n \le mk$, we have $1\le l \le k$. In addition, $n \le m^{\frac 2q} k^{\frac{2}{\sigma}} m^{1-\frac 2q} l^{1-\frac{2}{\sigma}}$. In remains to apply the inclusion $m^{-\frac 1p} l^{-\frac{1}{\theta}} V^{m,k}_{m,l} \subset B^{m,k}_{p,\theta}$ and Theorem \ref{v_dn}.
\end{proof}

The following assertion is proved in \cite{vas_mix2}.
\begin{Lem}
\label{lem_int_emb}
{\rm (see \cite{vas_mix2}).} Let $\nu_i>0$, $1\le p_i\le \infty$, $1\le \theta_i\le \infty$, $i=1, \, 2$, $\lambda \in [0, \, 1]$. Let the numbers $p$, $\theta\in [1, \, \infty]$ be defined by
\begin{align}
\label{emb_pt}
\frac 1p = \frac{1-\lambda}{p_1} + \frac{\lambda}{p_2}, \quad \frac{1}{\theta} = \frac{1-\lambda}{\theta_1} + \frac{\lambda}{\theta_2}.
\end{align}
Then
$$
\nu_1B_{p_1,\theta_1}^{m,k} \cap \nu_2B_{p_2,\theta_2} ^{m,k} \subset \nu_1^{1-\lambda} \nu_2^\lambda B^{m,k} _{p,\theta}.
$$
\end{Lem}

The idea of the proof is as follows. Applying H\"{o}lder's inequality, we get $$\|(a_i)_{i=1}^m\|_{l_p^m}\le \|(a_i)_{i=1}^m\|_{l_{p_1}^m}^{1-\lambda}\|(a_i)_{i=1}^m\|_{l_{p_2}^m}^\lambda, \; \|(b_j)_{j=1}^k\|_{l_\theta^k}\le \|(b_j)_{j=1}^k\|_{l_{\theta_1}^k}^{1-\lambda}\|(b_j)_{j=1}^k\|_{l_{\theta_2}^k}^\lambda.$$ This yields the inequality $$\|(x_{i,j})_{1\le i\le m, \, 1\le j\le k}\|_{l_{p,\theta}^{m,k}}\le \|(x_{i,j})_{1\le i\le m, \, 1\le j\le k}\|_{l_{p_1,\theta_1}^{m,k}}^{1-\lambda}\|(x_{i,j})_{1\le i\le m, \, 1\le j\le k}\|_{l_{p_2,\theta_2}^{m,k}}^{\lambda}.$$

Applying Lemma \ref{lem_int_emb}, by induction we get
\begin{Cor}
\label{emb_s} Let $\nu_i>0$, $1\le p_i\le \infty$, $1\le \theta_i\le \infty$, $\tau_i\ge 0$, $1\le i\le s$, $\sum \limits _{i=1}^s \tau_i=1$. Let the numbers $p$, $\theta\in [1, \, \infty]$ be defined by
\begin{align}
\label{emb_pt_s}
\frac 1p = \sum \limits _{i=1}^s \frac{\tau_i}{p_i}, \quad \frac{1}{\theta} = \sum \limits _{i=1}^s \frac{\tau_i}{\theta_i}.
\end{align}
Then
$$
\cap _{i=1}^s \nu_iB_{p_i,\theta_i}^{m,k} \subset \nu_1^{\tau_1}... \nu_s^{\tau_s} B^{m,k} _{p,\theta}.
$$
\end{Cor}
For $p_i=\theta_i$ ($1\le i\le s$), this assertion was proved by Galeev \cite{galeev1}.

Applying Theorem \ref{1mixed}, Propositions \ref{pgq}, \ref{est_dn} and Corollary \ref{emb_s}, we get the upper estimate in Theorem \ref{main}:
\begin{align}
\label{main_up_est} d_n(M, \, l^{m,k}_{q,\sigma}) \underset{q,\sigma}{\lesssim} \min _{0\le j\le 7} \Psi_j(P; \, q, \, \sigma; \, m, \, k, \, n; \, \nu).
\end{align}

\section{Proof of the lower estimate: $A$ is finite, \\ $\{(1/p_\alpha, \, 1/\theta_\alpha)\}_{\alpha\in A}$ are in a general position}

Let $A$ be a finite set. We say that the points $\{(1/p_\alpha, \, 1/\theta_\alpha)\}_{\alpha\in A}$ are in a general position if the following conditions hold:
\begin{enumerate}
\item $p_\alpha\ne 2$, $p_\alpha \ne q$, $\theta_\alpha \ne 2$, $\theta_\alpha \ne \sigma$, $\alpha \in A$; in addition, $\frac{1/p_\alpha-1/q}{1/2-1/q} \ne \frac{1/\theta_\alpha-1/\sigma}{1/2-1/\sigma}$, $\alpha \in A$, for $q>2$ and $\sigma>2$;
\item $p_\alpha \ne p_\beta$, $\theta_\alpha \ne \theta _\beta$ for $\alpha \ne \beta$;
\item the points $(1/2, \, 1/2)$, $(1/2, \, 1/\sigma)$, $(1/q, \, 1/2)$, $(1/q, \, 1/\sigma)$ do not lie on segments $[(1/p_\alpha, \, 1/\theta_\alpha), \, (1/p_\beta, \, 1/\theta_\beta)]$, $\alpha,$ $\beta \in A$;
\item $\Delta_{\alpha,\beta,\gamma} \in {\cal R}$ for all different $\alpha$, $\beta$, $\gamma \in A$ (i.e., the points $(1/p_\alpha, \, 1/\theta_\alpha)$, $(1/p_\beta, \, 1/\theta_\beta)$, $(1/p_\gamma, \, 1/\theta_\gamma)$ do not lie on the same line).

Let
\begin{align}
\label{psi_def} \Psi := \Psi(P):=\Psi(P; \, q, \, \sigma; \, m, \, k, \, n; \, \nu) := \min _{0\le j\le 7} \Psi_j.
\end{align}
We prove the estimate
$$
d_n(M, \, l^{m,k}_{q,\sigma}) \underset{q,\sigma}{\gtrsim} \Psi.
$$
\end{enumerate}

\subsection{The case $\Psi = \Psi_0$}

Since the set $A$ is finite, there is $\alpha \in A$ such that
\begin{align}
\label{psi_na_phi_pa}
\Psi = \nu_\alpha \Phi(p_\alpha, \, \theta_\alpha).
\end{align}
In addition, since the points $\{(1/p_\beta, \, 1/\theta_\beta)\}_{\beta\in A}$ are in a general position, we have $p_\alpha \notin \{2, \, q\}$ and $\theta_\alpha \notin \{2, \, \sigma\}$.

{\bf Case $p_\alpha> q$, $\theta_\alpha > \sigma$.} We prove the estimate (see \eqref{phi1})
$$
d_n(M, \, l^{m,k}_{q,\sigma}) \underset{q,\sigma}{ \gtrsim} \nu_\alpha m^{1/q-1/p_\alpha} k^{1/\sigma - 1/\theta_\alpha}.
$$
To this end, we check that $\nu_\alpha m^{-1/p_\alpha} k^{-1/\theta_\alpha} V^{m,k}_{m,k} \subset M$. Then
$$
d_n(M, \, l^{m,k}_{q,\sigma}) \ge d_n(\nu_\alpha m^{-1/p_\alpha} k^{-1/\theta_\alpha} V^{m,k}_{m,k}, \, l^{m,k}_{q,\sigma}) \stackrel{(\ref{dn_vmk1})}{\underset{q,\sigma}{\gtrsim}} \nu_\alpha m^{1/q-1/p_\alpha} k^{1/\sigma-1/\theta_\alpha}.
$$
By (\ref{gamma_x_def})--(\ref{vrl_km}), it suffices to prove that
\begin{align}
\label{nu_a_b} \nu_\alpha m^{1/p_\beta-1/p_\alpha} k^{1/\theta_\beta -1/\theta_\alpha} \le \nu_\beta, \quad \beta \in A.
\end{align}

We set $$\lambda := \max \left\{ \mu \in [0, \, 1]:\; \frac{1-\mu}{p_\alpha} + \frac{\mu}{p_\beta}\le \frac 1q, \; \frac{1-\mu}{\theta_\alpha} + \frac{\mu}{\theta_\beta}\le \frac{1}{\sigma}\right\};$$
the numbers $p_*$ and $\theta_*$ are defined by 
\begin{align}
\label{p_star}
\frac{1}{p_*} = \frac{1-\lambda}{p_\alpha} + \frac{\lambda}{p_\beta}, \quad \frac{1}{\theta_*} = \frac{1-\lambda}{\theta_\alpha} + \frac{\lambda}{\theta_\beta}.
\end{align}
Notice that $\lambda \in (0, \, 1]$. In addition, $\nu_\alpha^{1-\lambda}\nu_\beta^\lambda \Phi(p_*, \, \theta_*)\ge \min \{\Psi_0, \, \Psi_1, \, \Psi_2\}$ (see (\ref{n1}), (\ref{n2}) and (\ref{psi0})--(\ref{psi2})). This together with (\ref{psi_na_phi_pa}) implies that $$\nu_\alpha \Phi(p_\alpha, \, \theta_\alpha)=\Psi \le \min \{\Psi_0, \, \Psi_1, \, \Psi_2\} \le \nu_\alpha^{1-\lambda}\nu_\beta^\lambda \Phi(p_*, \, \theta_*);$$ i.e.,
$$
\nu_\alpha m^{\frac 1q-\frac{1}{p_\alpha}} k^{\frac{1}{\sigma}-\frac{1}{\theta_\alpha}} \le \nu_\alpha^{1-\lambda}\nu_\beta^\lambda m^{\frac 1q-\frac{1-\lambda}{p_\alpha} - \frac{\lambda}{p_\beta}} k^{\frac{1}{\sigma}-\frac{1-\lambda}{\theta_\alpha} -\frac{\lambda}{\theta_\beta}}
$$
(see \eqref{phi1}); this yields (\ref{nu_a_b}).

{\bf Case $p_\alpha>q$, $\theta_\alpha < \sigma$.}

Let $n\le mk^{2/\sigma}$. By (\ref{phi2}), it suffices to prove the estimate
$$
d_n(M, \, l^{m,k}_{q,\sigma}) \underset{q,\sigma}{\gtrsim} \nu_\alpha m^{1/q-1/p_\alpha}.
$$
We check the inclusion $\nu_\alpha m^{-1/p_\alpha}  V^{m,k}_{m,1} \subset M$. Then
$$
d_n(M, \, l^{m,k}_{q,\sigma}) \ge d_n(\nu_\alpha m^{-1/p_\alpha} V^{m,k}_{m,1}, \, l^{m,k}_{q,\sigma}) \stackrel{(\ref{dn_vmk1})}{\underset{q,\sigma}{\gtrsim}} \nu_\alpha m^{1/q-1/p_\alpha}.
$$
It suffices to prove that
\begin{align}
\label{nu_a_b1} \nu_\alpha m^{1/p_\beta-1/p_\alpha} \le \nu_\beta, \quad \beta \in A.
\end{align}

Let $$\lambda := \max \left\{ \mu \in [0, \, 1]:\; \frac{1-\mu}{p_\alpha} + \frac{\mu}{p_\beta}\le \frac 1q, \; \frac{1-\mu}{\theta_\alpha} + \frac{\mu}{\theta_\beta}\ge \frac{1}{\sigma}\right\};$$
the numbers $p_*$ and $\theta_*$ are defined by (\ref{p_star}).
Then $\lambda \in (0, \, 1]$, $\nu_\alpha^{1-\lambda}\nu_\beta^\lambda \Phi(p_*, \, \theta_*)\ge \min \{\Psi_0, \, \Psi_1, \, \Psi_2\}$. This together with (\ref{psi_na_phi_pa}) yields that $\nu_\alpha \Phi(p_\alpha, \, \theta_\alpha) \le \nu_\alpha^{1-\lambda}\nu_\beta^\lambda \Phi(p_*, \, \theta_*)$; i.e.,
$$
\nu_\alpha m^{\frac 1q-\frac{1}{p_\alpha}} \le \nu_\alpha^{1-\lambda}\nu_\beta^\lambda m^{\frac 1q-\frac{1-\lambda}{p_\alpha} - \frac{\lambda}{p_\beta}}
$$
(see \eqref{phi2}); this implies (\ref{nu_a_b1}).

Let $mk^{2/\sigma}< n \le \frac{mk}{2}$. Then $\sigma>2$.

If $\theta_\alpha >2$, then, by (\ref{om_pq}), (\ref{phi2}), it suffices to prove the inequality
\begin{align}
\label{111}
d_n(M, \, l^{m,k}_{q,\sigma}) \underset{q,\sigma}{\gtrsim} \nu_\alpha m^{1/q-1/p_\alpha} \left(n^{-\frac 12} m^{\frac 12} k^{\frac{1}{\sigma}}\right) ^{\frac{1/\theta_\alpha-1/\sigma}{1/2-1/\sigma}}.
\end{align}
We set $\tilde l = \left(n^{\frac 12} m^{-\frac 12} k^{-\frac{1}{\sigma}}\right) ^{\frac{1}{1/2-1/\sigma}}$, $l =\lceil \tilde l\rceil$. Since $mk^{2/\sigma}\le n\le mk$, we have $1\le l \le k$. In addition, $n\le m^{\frac 2q} k^{\frac{2}{\sigma}} m^{1-\frac 2q} l^{1-\frac{2}{\sigma}}$. We check that $\nu_\alpha m^{-1/p_\alpha} l^{-1/\theta_\alpha} V^{m,k}_{m,l} \subset 2M$. Then
$$
d_n(M, \, l^{m,k}_{q,\sigma}) \gtrsim d_n(\nu_\alpha m^{-1/p_\alpha} l^{-1/\theta_\alpha} V^{m,k}_{m,l}, \, l^{m,k}_{q,\sigma}) \stackrel{(\ref{dn_vmk1})}{\underset{q,\sigma}{\gtrsim}} \nu_\alpha m^{1/q-1/p_\alpha} l^{1/\sigma-1/\theta_\alpha};
$$
this implies (\ref{111}).

It suffices to prove that
\begin{align}
\label{nu_a_b2} \nu_\alpha m^{1/p_\beta-1/p_\alpha} \tilde l^{1/\theta_\beta-1/\theta_\alpha} \le \nu_\beta, \quad \beta \in A.
\end{align}

We set $$\lambda := \max \left\{ \mu \in [0, \, 1]:\; \frac{1-\mu}{p_\alpha} + \frac{\mu}{p_\beta}\le \frac 1q, \; \frac{1-\mu}{\theta_\alpha} + \frac{\mu}{\theta_\beta}\in \left[\frac{1}{\sigma}, \, \frac 12\right]\right\},$$
the numbers $p_*$ and $\theta_*$ are defined by (\ref{p_star}).
Then $\lambda \in (0, \, 1]$. In addition, $$\nu_\alpha^{1-\lambda}\nu_\beta^\lambda \Phi(p_*, \, \theta_*) \ge \min \{\Psi_0, \, \Psi_1, \, \Psi_2, \, \Psi_4\}.$$ This together with (\ref{psi_na_phi_pa}) yields that $\nu_\alpha \Phi(p_\alpha, \, \theta_\alpha) \le \nu_\alpha^{1-\lambda}\nu_\beta^\lambda \Phi(p_*, \, \theta_*)$; i.e.,
$$
\nu_\alpha m^{\frac 1q-\frac{1}{p_\alpha}}\left(n^{-\frac 12} m^{\frac 12} k^{\frac{1}{\sigma}}\right) ^{\frac{1/\theta_\alpha-1/\sigma}{1/2-1/\sigma}} \le \nu_\alpha^{1-\lambda}\nu_\beta^\lambda m^{\frac 1q-\frac{1-\lambda}{p_\alpha} - \frac{\lambda}{p_\beta}} \left(n^{-\frac 12} m^{\frac 12} k^{\frac{1}{\sigma}}\right) ^{\frac{(1-\lambda)/\theta_\alpha+\lambda/\theta_\beta-1/\sigma}{1/2-1/\sigma}}
$$
(see \eqref{om_pq}, \eqref{phi2}); this implies (\ref{nu_a_b2}).

Now, let $\theta_\alpha < 2$. By (\ref{om_pq}), (\ref{phi2}), it suffices to prove the inequality
\begin{align}
\label{222}
d_n(M, \, l^{m,k}_{q,\sigma}) \underset{q,\sigma}{\gtrsim} \nu_\alpha m^{1/q-1/p_\alpha} n^{-\frac 12} m^{\frac 12} k^{\frac{1}{\sigma}}.
\end{align}
We check that $\nu_\alpha m^{-1/p_\alpha}  V^{m,k}_{m,1} \subset M$. Then
$$
d_n(M, \, l^{m,k}_{q,\sigma}) \ge d_n(\nu_\alpha m^{-1/p_\alpha} V^{m,k}_{m,1}, \, l^{m,k}_{q,\sigma}) \stackrel{(\ref{dn_vmk1})}{\underset{q,\sigma}{\gtrsim}} \nu_\alpha m^{-1/p_\alpha} n^{-1/2} m^{1/q} k^{1/\sigma} m^{1/2}
$$
(in the second inequality we used the condition $n\ge mk^{2/\sigma}$).

It suffices to prove (\ref{nu_a_b1}). Let $$\lambda := \max \left\{ \mu \in [0, \, 1]:\; \frac{1-\mu}{p_\alpha} + \frac{\mu}{p_\beta}\le \frac 1q, \; \frac{1-\mu}{\theta_\alpha} + \frac{\mu}{\theta_\beta}\ge \frac 12\right\}.$$

As in the previous cases, we get
$$
\nu_\alpha m^{1/q-1/p_\alpha} n^{-\frac 12} m^{\frac 12} k^{\frac{1}{\sigma}} \le \nu_\alpha^{1-\lambda}\nu_\beta^\lambda m^{\frac 1q-\frac{1-\lambda}{p_\alpha} - \frac{\lambda}{p_\beta}} n^{-\frac 12} m^{\frac 12} k^{\frac{1}{\sigma}};
$$
this implies (\ref{nu_a_b1}).

{\bf The case $p_\alpha<q$, $\theta_\alpha>\sigma$} is similar.

{\bf The case $p_\alpha <q$, $\theta_\alpha <\sigma$.}

Let $n\le m^{2/q}k^{2/\sigma}$. By (\ref{phi4})--(\ref{phi6}), it suffices to prove the estimate
$$
d_n(M, \, l^{m,k}_{q,\sigma}) \underset{q,\sigma}{\gtrsim} \nu_\alpha.
$$
To this end, we check the inclusion $\nu_\alpha V^{m,k}_{1,1} \subset M$, i.e.,
\begin{align}
\label{nu_a_b3} \nu_\alpha \le \nu_\beta, \quad \beta \in A;
\end{align}
then we use (\ref{dn_vmk1}).

Let $$\lambda := \max \left\{ \mu \in [0, \, 1]:\; \frac{1-\mu}{p_\alpha} + \frac{\mu}{p_\beta}\ge \frac 1q, \, \frac{1-\mu}{\theta_\alpha} + \frac{\mu}{\theta_\beta}\ge \frac{1}{\sigma}\right\};$$
the numbers $p_*$ and $\theta_*$ are defined by (\ref{p_star}). As in the previous cases, we get
$$
\nu_\alpha \le \nu_\alpha^{1-\lambda}\nu_\beta^\lambda;
$$
this implies (\ref{nu_a_b3}).

Let $n> m^{2/q}k^{2/\sigma}$. We consider the following cases.
\begin{enumerate}
\item $p_\alpha<2$, $\theta_\alpha<2$. We prove the estimate (see \eqref{phi6})
$$
d_n(M, \, l^{m,k}_{q,\sigma}) \underset{q,\sigma}{\gtrsim} \nu_\alpha n^{-\frac 12}m^{\frac 1q} k^{\frac{1}{\sigma}}.
$$
We again check the inclusion $\nu_\alpha V^{m,k}_{1,1} \subset M$, i.e., the inequality (\ref{nu_a_b3}). Let $$\lambda := \max \left\{ \mu \in [0, \, 1]:\; \frac{1-\mu}{p_\alpha} + \frac{\mu}{p_\beta}\ge \frac 12, \, \frac{1-\mu}{\theta_\alpha} + \frac{\mu}{\theta_\beta}\ge \frac{1}{2}\right\};$$
the numbers $p_*$ and $\theta_*$ are defined by (\ref{p_star}). Then $\nu_\alpha \Phi(p_\alpha, \, \theta_\alpha) \le \nu_\alpha^{1-\lambda}\nu_\beta^\lambda \Phi(p_*, \, \theta_*)$; i.e.,
$$
\nu_\alpha n^{-\frac 12}m^{\frac 1q} k^{\frac{1}{\sigma}}\le \nu_\alpha^{1-\lambda}\nu_\beta^\lambda n^{-\frac 12}m^{\frac 1q} k^{\frac{1}{\sigma}}.
$$
This implies (\ref{nu_a_b3}).

\item $p_\alpha > 2$, $\omega_{p_\alpha,q}< \omega_{\theta_\alpha,\sigma}$. Then $q>2$.

Let $m^{2/q}k^{2/\sigma}< n \le mk^{2/\sigma}$. It suffices to prove the estimate (see \eqref{phi4})
\begin{align}
\label{dn11}
d_n(M, \, l^{m,k}_{q,\sigma}) \underset{q,\sigma}{\gtrsim} \nu_\alpha \left(n^{-\frac 12}m^{\frac 1q} k^{\frac{1}{\sigma}}\right)^{\frac{1/p_\alpha-1/q}{1/2-1/q}}.
\end{align}
We set $\tilde r = \left(n^{\frac 12}m^{-\frac 1q} k^{-\frac{1}{\sigma}}\right)^{\frac{1}{1/2-1/q}}$, $r = \lceil \tilde r \rceil$. Then $1\le r \le m$ and $n \le m^{\frac 2q} k^{\frac{2}{\sigma}} r^{1-\frac 2q}$. We check the inclusion $\nu_\alpha r^{-1/p_\alpha} V^{m,k}_{r,1} \subset 2M$; then we get
$$
d_n(M, \, l^{m,k}_{q,\sigma}) \gtrsim d_n(\nu_\alpha r^{-1/p_\alpha} V^{m,k}_{r,1}, \, l^{m,k}_{q,\sigma}) \stackrel{(\ref{dn_vmk1})}{\underset{q,\sigma}{\gtrsim}} \nu_\alpha r^{1/q-1/p_\alpha};
$$
this implies (\ref{dn11}).

It suffices to check that
\begin{align}
\label{nu_a_b4} \nu_\alpha\tilde r^{1/p_\beta-1/p_\alpha}\le \nu_\beta, \quad \beta \in A.
\end{align}

For $\sigma>2$, we set 
$$
\begin{array}{l}
\lambda := \max \left\{ \mu \in [0, \, 1]:\; \frac{1-\mu}{p_\alpha} + \frac{\mu}{p_\beta}\in \left[\frac 1q, \, \frac 12\right], \right. \\ \left. \frac{(1-\mu)/p_\alpha+\mu/p_\beta-1/q}{1/2-1/q}\le \frac{(1-\mu)/\theta_\alpha+\mu/\theta_\beta-1/\sigma}{1/2-1/\sigma}\right\};
\end{array}
$$
for $\sigma=2$, we set
\begin{align}
\label{lam_max_mu}
\lambda := \max \left\{ \mu \in [0, \, 1]:\; \frac{1-\mu}{p_\alpha} + \frac{\mu}{p_\beta}\in \left[\frac 1q, \, \frac 12\right], \, \frac{1-\mu}{\theta_\alpha}+\frac{\mu}{\theta_\beta}\ge \frac 12\right\};
\end{align}
the numbers $p_*$ and $\theta_*$ are defined by (\ref{p_star}). Then $\nu_\alpha \Phi(p_\alpha, \, \theta_\alpha)\le \nu_\alpha^{1-\lambda}\nu_\beta^\lambda \Phi(p_*, \, \theta_*)$, i.e.,
$$
\nu_\alpha \left(n^{-\frac 12}m^{\frac 1q} k^{\frac{1}{\sigma}}\right)^{\frac{1/p_\alpha-1/q}{1/2-1/q}}\le \nu_\alpha^{1-\lambda}\nu_\beta^\lambda \left(n^{-\frac 12}m^{\frac 1q} k^{\frac{1}{\sigma}}\right)^{\frac{(1-\lambda)/p_\alpha+\lambda/p_\beta-1/q}{1/2-1/q}};
$$
this yields (\ref{nu_a_b4}).

Now, let $n>mk^{2/\sigma}$. Then $\sigma>2$. This case is considered similarly as the case $p_\alpha>q$, $\theta_\alpha < \sigma$ (see (\ref{phi2}), (\ref{phi4})); here $\lambda$ is defined by
$$
\begin{array}{l}
\lambda := \max \left\{ \mu \in [0, \, 1]:\; \frac{1-\mu}{p_\alpha} + \frac{\mu}{p_\beta}\in \left[\frac 1q, \, \frac 12\right], \;  \frac{1-\mu}{\theta_\alpha} + \frac{\mu}{\theta_\beta}\le \frac 12,\right. \\ \left. \frac{(1-\mu)/p_\alpha+\mu/p_\beta-1/q}{1/2-1/q}\le \frac{(1-\mu)/\theta_\alpha+\mu/\theta_\beta-1/\sigma}{1/2-1/\sigma}\right\} \quad \text{for }\theta_\alpha>2;
\end{array}
$$
for $\theta_\alpha<2$, $\lambda$ is defined by (\ref{lam_max_mu}).

\item The case $\theta_\alpha > 2$, $\omega_{p_\alpha,q}> \omega_{\theta_\alpha,\sigma}$ is similar to the previous.
\end{enumerate}
Notice that for $2<p_\alpha<q$ or $2<\theta_\alpha<\sigma$ we have $\omega_{p_\alpha,q}\ne \omega_{\theta_\alpha,\sigma}$, since the points $\{(1/p_\beta, \, 1/\theta_\beta)\}_{\beta\in A}$ are in a general position.

\subsection{The cases $\Psi = \Psi_1$, $\Psi = \Psi_2$, $\Psi = \Psi_3$ and $\Psi = \Psi_4$}

Since the set $A$ is finite, there are $\alpha$, $\beta \in A$ such that 
\begin{align}
\label{psi_eq_psi1}
p_\alpha>q>p_\beta, \quad \Psi =\Psi_1 = \nu_\alpha^{1-\hat \lambda_{\alpha,\beta}}\nu_\beta^{\hat \lambda_{\alpha,\beta}} \Phi(q, \, \hat \theta_{\alpha,\beta}),
\end{align}
\begin{align}
\label{psi_eq_psi2}
\theta_\alpha>\sigma >\theta_\beta, \quad \Psi =\Psi_2= \nu_\alpha^{1-\hat \mu_{\alpha,\beta}}\nu_\beta^{\hat \mu_{\alpha,\beta}} \Phi(\hat p_{\alpha,\beta}, \, \sigma),
\end{align}
\begin{align}
\label{psi3_psi} p_\alpha>2>p_\beta, \quad \Psi =\Psi_3= \nu_\alpha^{1-\tilde \lambda_{\alpha,\beta}} \nu_\beta^{\tilde \lambda_{\alpha,\beta}} \Phi(2, \, \tilde \theta_{\alpha,\beta}),
\end{align}
or
\begin{align}
\label{psi4_psi} \theta_\alpha>2>\theta_\beta, \quad \Psi =\Psi_4= \nu_\alpha^{1-\tilde \mu_{\alpha,\beta}} \nu_\beta^{\tilde \mu_{\alpha,\beta}} \Phi(\tilde p_{\alpha,\beta}, \, 2).
\end{align}
If $\Psi=\Psi_3$, we suppose that $q>2$, and if $\Psi=\Psi_4$, we suppose that $\sigma>2$ (otherwise, we have $\Psi=\Psi_1$ or $\Psi=\Psi_2$, respectively).

Since the points $\{(1/p_\gamma, \, 1/\theta_\gamma)\}_{\gamma\in A}$ are in a general position, we have $\hat \theta_{\alpha,\beta}$, $\tilde \theta_{\alpha,\beta}\notin \{2, \, \sigma\}$, $\hat p_{\alpha,\beta}, \tilde p_{\alpha,\beta} \notin \{2, \, q\}$.

We consider the cases in the following order:
\begin{enumerate}
\item a) (\ref{psi_eq_psi1}) holds with $\hat \theta_{\alpha,\beta}> \sigma$, b) (\ref{psi_eq_psi2}) holds with $\hat p_{\alpha,\beta}> q$;
\item a) (\ref{psi_eq_psi1}) holds with $\hat \theta_{\alpha,\beta}<\sigma$, $n\le mk^{2/\sigma}$, b) (\ref{psi_eq_psi2}) holds with $\hat p_{\alpha,\beta}<q$, $n\le m^{2/q}k$;
\item a) (\ref{psi3_psi}) holds with $\tilde \theta_{\alpha,\beta} > \sigma$, b) (\ref{psi4_psi}) holds with $\tilde p_{\alpha,\beta}>q$;
\item a) (\ref{psi_eq_psi1}) holds with $\hat \theta_{\alpha,\beta}< \sigma$, $n> mk^{2/\sigma}$, b)  (\ref{psi_eq_psi2}) holds with $\hat p_{\alpha,\beta}< q$, $n> m^{2/q}k$;
\item a) (\ref{psi3_psi}) holds with $\tilde \theta_{\alpha,\beta} <2$, b) (\ref{psi4_psi}) holds with $\tilde p_{\alpha,\beta}<2$;
\item a) (\ref{psi3_psi}) holds with $2<\tilde \theta_{\alpha,\beta} < \sigma$, b) (\ref{psi4_psi}) holds with $2<\tilde p_{\alpha,\beta}<q$.
\end{enumerate}

{\bf Case 1a)} (case 1b) is similar). By (\ref{phi1}), it suffices to prove that
\begin{align}
\label{q_dn_est}
d_n(M, \, l^{m,k}_{q,\sigma}) \underset{q,\sigma}{\gtrsim} \nu_\alpha^{1-\hat \lambda_{\alpha,\beta}}\nu_\beta^{\hat \lambda_{\alpha,\beta}} k^{1/\sigma -1/\hat \theta_{\alpha,\beta}}.
\end{align}
We define the number $r_{\alpha,\beta}$ by
\begin{align}
\label{r_ab_def}
\frac{\nu_\alpha}{\nu_\beta} = r_{\alpha,\beta}^{1/p_\alpha-1/p_\beta}k^{1/\theta_\alpha-1/\theta_\beta}.
\end{align}
Since $p_\alpha> p_\beta$, the number $r_{\alpha,\beta}$ is well-defined.

We show that
\begin{align}
\label{r_le_m} 1\le  r_{\alpha,\beta} \le m, \quad n \le m^{\frac 2q}kr_{\alpha,\beta}^{1-\frac 2q}.
\end{align}

Let $$\lambda := \min \left\{\mu\in [0, \, \hat \lambda_{\alpha,\beta}]: \frac{1-\mu}{\theta_\alpha}+\frac{\mu}{\theta_\beta}\le \frac{1}{\sigma}\right\},$$
$$\tilde \lambda:= \max\left\{\mu\in [\hat \lambda_{\alpha,\beta}, \, 1]: \frac{1-\mu}{\theta_\alpha}+\frac{\mu}{\theta_\beta}\le \frac{1}{\sigma}, \; \frac{1-\mu}{p_\alpha}+\frac{\mu}{p_\beta}\le \frac{1}{2}\right\} \; \text{for }q>2,$$ 
$$\tilde \lambda:= \max\left\{\mu\in [\hat \lambda_{\alpha,\beta}, \, 1]: \frac{1-\mu}{\theta_\alpha}+\frac{\mu}{\theta_\beta}\le \frac{1}{\sigma}\right\} \; \text{for }q=2.$$ 
Then $\lambda < \hat \lambda_{\alpha,\beta} < \tilde \lambda$. We also define the numbers $p_*$, $\theta_*$, $p_{**}$, $\theta_{**}$ by
\begin{align}
\label{p_st1}
\frac{1}{p_*} = \frac{1-\lambda}{p_\alpha} + \frac{\lambda}{p_\beta}, \quad \frac{1}{\theta_*} = \frac{1-\lambda}{\theta_\alpha} + \frac{\lambda}{\theta_\beta},
\end{align}
\begin{align}
\label{p_st2}
\frac{1}{p_{**}} = \frac{1-\tilde\lambda}{p_\alpha} + \frac{\tilde\lambda}{p_\beta}, \quad \frac{1}{\theta_{**}} = \frac{1-\tilde\lambda}{\theta_\alpha} + \frac{\tilde\lambda}{\theta_\beta}.
\end{align}
Then $\nu_\alpha^{1-\lambda}\nu_\beta^{\lambda} \Phi(p_*, \, \theta_*)\ge \min \{\Psi_0, \, \Psi_2\}$, $\nu_\alpha^{1-\tilde\lambda}\nu_\beta^{\tilde\lambda} \Phi(p_{**}, \, \theta_{**}) \ge \min \{\Psi_0, \, \Psi_2, \, \Psi_3\}$. This together with (\ref{psi_eq_psi1}) implies that
\begin{align}
\label{nu1111}
\nu_\alpha^{1-\hat \lambda_{\alpha,\beta}}\nu_\beta^{\hat \lambda_{\alpha,\beta}} \Phi(q, \, \hat \theta_{\alpha,\beta})\le \nu_\alpha^{1-\lambda}\nu_\beta^{\lambda} \Phi(p_*, \, \theta_*), 
\end{align}
\begin{align}
\label{nu2222}
\nu_\alpha^{1-\hat \lambda_{\alpha,\beta}}\nu_\beta^{\hat \lambda_{\alpha,\beta}} \Phi(q, \, \hat \theta_{\alpha,\beta})\le \nu_\alpha^{1-\tilde\lambda}\nu_\beta^{\tilde\lambda} \Phi(p_{**}, \, \theta_{**}).
\end{align}
Inequality (\ref{nu1111}) can be written as follows (see \eqref{phi1}):
$$
\nu_\alpha^{1-\hat \lambda_{\alpha,\beta}}\nu_\beta^{\hat \lambda_{\alpha,\beta}} k^{\frac{1}{\sigma} - \frac{1-\hat \lambda_{\alpha,\beta}}{\theta_\alpha} - \frac{\hat \lambda_{\alpha,\beta}}{\theta_\beta}} \le \nu_\alpha^{1-\lambda}\nu_\beta^{\lambda} m^{\frac 1q - \frac{1-\lambda}{p_\alpha} - \frac{\lambda}{p_\beta}}k^{\frac{1}{\sigma} - \frac{1-\lambda}{\theta_\alpha} - \frac{\lambda}{\theta_\beta}};
$$
this implies $\frac{\nu_\alpha}{\nu_\beta} \ge m^{\frac{1}{p_\alpha} -\frac{1}{p_\beta}} k^{\frac{1}{\theta_\alpha} - \frac{1}{\theta_\beta}}$. Since $p_\alpha>p_\beta$, this together with (\ref{r_ab_def}) yields that $r_{\alpha,\beta}\le m$. By (\ref{phi3}), the inequality (\ref{nu2222}) can be written as follows:
$$
\nu_\alpha^{1-\hat \lambda_{\alpha,\beta}}\nu_\beta^{\hat \lambda_{\alpha,\beta}} k^{\frac{1}{\sigma} - \frac{1-\hat \lambda_{\alpha,\beta}}{\theta_\alpha} - \frac{\hat \lambda_{\alpha,\beta}}{\theta_\beta}} \le \nu_\alpha^{1-\tilde\lambda}\nu_\beta^{\tilde\lambda} k^{\frac{1}{\sigma} - \frac{1-\tilde\lambda}{\theta_\alpha} - \frac{\tilde\lambda}{\theta_\beta}},
$$
for $n \le m^{2/q}k$, and
$$
\nu_\alpha^{1-\hat \lambda_{\alpha,\beta}}\nu_\beta^{\hat \lambda_{\alpha,\beta}} k^{\frac{1}{\sigma} - \frac{1-\hat \lambda_{\alpha,\beta}}{\theta_\alpha} - \frac{\hat \lambda_{\alpha,\beta}}{\theta_\beta}} \le \nu_\alpha^{1-\tilde\lambda}\nu_\beta^{\tilde\lambda} k^{\frac{1}{\sigma} - \frac{1-\tilde\lambda}{\theta_\alpha} - \frac{\tilde\lambda}{\theta_\beta}}(n^{\frac 12}m^{-\frac 1q} k^{-\frac 12}) ^{\frac{1/q-(1-\tilde \lambda)/p_\alpha -\tilde \lambda/p_\beta}{1/2-1/q}},
$$
for $n> m^{2/q}k$;
notice that the case $n>m^{2/q}k$ is possible only for $q>2$. Hence, $\frac{\nu_\alpha}{\nu_\beta} \le k^{\frac{1}{\theta_\alpha} -\frac{1}{\theta_\beta}}$ for $n\le m^{2/q}k$, $\frac{\nu_\alpha}{\nu_\beta} \le k^{\frac{1}{\theta_\alpha} -\frac{1}{\theta_\beta}}(n^{1/2}m^{-1/q}k^{-1/2})^{\frac{1/p_\alpha-1/p_\beta}{1/2-1/q}}$ for $n> m^{2/q}k$. This together with (\ref{r_ab_def}) implies that $r_{\alpha,\beta}\ge 1$ for $n\le m^{2/q}k$, $r_{\alpha,\beta} \ge (n^{1/2}m^{-1/q}k^{-1/2})^{\frac{1}{1/2-1/q}}$ for $n> m^{2/q}k$, which completes the proof of (\ref{r_le_m}).

Let $r= \lceil r_{\alpha,\beta}\rceil$, $$W = \nu_\alpha^{1-\hat \lambda_{\alpha,\beta}}\nu_\beta^{\hat \lambda_{\alpha,\beta}} r^{-1/q} k^{-1/\hat \theta_{\alpha,\beta}} V^{m,k}_{r,k}.$$ If $W\subset 2M$, then
$$
d_n(M, \, l^{m,k}_{q,\sigma}) \gtrsim d_n(W, \, l^{m,k}_{q,\sigma}) \stackrel{(\ref{dn_vmk1}), (\ref{r_le_m})}{\underset{q,\sigma}{\gtrsim}} \nu_\alpha^{1-\hat \lambda_{\alpha,\beta}}\nu_\beta^{\hat \lambda_{\alpha,\beta}} r^{-1/q} k^{-1/\hat \theta_{\alpha,\beta}}r^{1/q}k^{1/\sigma}=$$$$= \nu_\alpha^{1-\hat \lambda_{\alpha,\beta}}\nu_\beta^{\hat \lambda_{\alpha,\beta}} k^{1/\sigma-1/\hat \theta_{\alpha,\beta}};
$$
i.e., (\ref{q_dn_est}) holds.

In order to check the inclusion $W \subset 2M$, it suffices to show that
\begin{align}
\label{incl_ineq} \nu_\alpha^{1-\hat \lambda_{\alpha,\beta}}\nu_\beta^{\hat \lambda_{\alpha,\beta}} r_{\alpha,\beta}^{1/p_\gamma-1/q} k^{1/\theta_\gamma-1/\hat \theta_{\alpha,\beta}} \le \nu_\gamma, \quad \gamma \in A.
\end{align}
For $\gamma=\alpha$ or $\gamma=\beta$, it follows from (\ref{r_ab_def}).

Let $\gamma \notin \{\alpha, \, \beta\}$. We define the numbers $r_{\alpha,\gamma}$ and $r_{\gamma,\beta}$ by
\begin{align}
\label{r_ag_def} \frac{\nu_\alpha}{\nu_\gamma} = r_{\alpha,\gamma}^{1/p_\alpha-1/p_\gamma} k^{1/\theta_\alpha-1/\theta_\gamma}, \quad \frac{\nu_\gamma}{\nu_\beta} = r_{\gamma,\beta}^{1/p_\gamma-1/p_\beta} k^{1/\theta_\gamma - 1/\theta_\beta}.
\end{align}
Since $p_\alpha\ne p_\gamma$ and $p_\beta\ne p_\gamma$ (see the definition of the general position), the numbers $r_{\alpha,\gamma}$ and $r_{\gamma,\beta}$ are well-defined.

By (\ref{r_ab_def}), (\ref{r_ag_def}), the inequality (\ref{incl_ineq}) is equivalent to each of the following inequalities:
\begin{align}
\label{incl_ineq_equiv} r_{\alpha,\beta}^{1/p_\gamma - 1/p_\alpha} \le r_{\alpha,\gamma}^{1/p_\gamma-1/p_\alpha}, \quad r_{\alpha,\beta}^{1/p_\gamma - 1/p_\beta} \le r_{\gamma,\beta}^{1/p_\gamma-1/p_\beta}.
\end{align}

\begin{enumerate}
\item Let $p_\gamma<q$. We check the first inequality of (\ref{incl_ineq_equiv}). Since $p_\alpha>q>p_\gamma$, it is equivalent to $r_{\alpha,\beta} \le r_{\alpha,\gamma}$.

There is $\hat \lambda_{\alpha,\gamma}\in (0, \, 1)$ such that $\frac{1-\hat \lambda_{\alpha,\gamma}}{p_\alpha} + \frac{\hat \lambda_{\alpha,\gamma}}{p_\gamma} = \frac 1q$. The number $\hat \theta_{\alpha,\gamma}$ is defined by the equation $\frac{1}{\hat \theta_{\alpha,\gamma}} = \frac{1-\hat \lambda_{\alpha,\gamma}}{\theta_\alpha} + \frac{\hat \lambda_{\alpha,\gamma}}{\theta_\gamma}$ (see (\ref{n1})).

If $\hat \theta_{\alpha,\gamma}>\sigma$, then from $$\nu_\alpha^{1-\hat \lambda_{\alpha,\beta}}\nu_\beta^{\hat \lambda_{\alpha,\beta}} \Phi(q, \, \hat \theta_{\alpha,\beta}) \stackrel{(\ref{psi_eq_psi1})}{=}\Psi \le \Psi_1\le \nu_\alpha^{1-\hat \lambda_{\alpha,\gamma}}\nu_\gamma^{\hat \lambda_{\alpha,\gamma}} \Phi(q, \, \hat \theta_{\alpha,\gamma})$$
and (\ref{phi1}) we get
\begin{align}
\label{nu_a_s}
\nu_\alpha\left(\frac{\nu_\beta}{\nu_\alpha}\right) ^{\hat \lambda_{\alpha,\beta}} k^{1/\sigma -1/\hat \theta_{\alpha,\beta}} \le \nu_\alpha\left(\frac{\nu_\gamma}{\nu_\alpha}\right) ^{\hat \lambda_{\alpha,\gamma}} k^{1/\sigma -1/\hat \theta_{\alpha,\gamma}}.
\end{align}
Applying (\ref{r_ab_def}), (\ref{r_ag_def}) and the equalities $\frac{1}{q}-\frac{1}{p_\alpha} = \hat \lambda_{\alpha,\beta}\left(\frac{1}{p_\beta}-\frac{1}{p_\alpha}\right) = \hat \lambda_{\alpha,\gamma}\left(\frac{1}{p_\gamma}-\frac{1}{p_\alpha}\right)$, we get $r_{\alpha,\beta} ^{1/q-1/p_\alpha} \le r_{\alpha,\gamma} ^{1/q-1/p_\alpha}$. Since $p_\alpha>q$, we have $r_{\alpha,\beta} \le r_{\alpha,\gamma}$.

If $\hat \theta_{\alpha,\gamma}<\sigma$, then there is $\mu\in (0, \, 1)$ such that $\frac{1-\mu}{\hat \theta_{\alpha,\beta}} + \frac{\mu}{\hat \theta _{\alpha,\gamma}} = \frac{1}{\sigma}$. Then from
$$
\nu_\alpha^{1-\hat \lambda_{\alpha,\beta}}\nu_\beta^{\hat \lambda_{\alpha,\beta}} \Phi(q, \, \hat\theta_{\alpha,\beta}) =\Psi\le \Psi_6\le (\nu_\alpha^{1-\hat \lambda_{\alpha,\beta}}\nu_\beta^{\hat \lambda_{\alpha,\beta}})^{1-\mu} (\nu_\alpha^{1-\hat \lambda_{\alpha,\gamma}}\nu_\gamma^{\hat \lambda_{\alpha,\gamma}})^\mu \Phi(q, \, \sigma)
$$
we get (\ref{nu_a_s}); this again implies the inequality $r_{\alpha,\beta} \le r_{\alpha,\gamma}$.

\item Let $p_\gamma>q$. We check the second inequality of (\ref{incl_ineq_equiv}). Since $p_\gamma>q>p_\beta$, it is equivalent to $r_{\alpha,\beta}\ge r_{\gamma,\beta}$. There is $\hat \lambda_{\gamma,\beta}\in (0, \, 1)$ such that $\frac 1q = \frac{1-\hat \lambda_{\gamma,\beta}}{p_\gamma}+\frac{\hat \lambda_{\gamma,\beta}}{p_\beta}$. Arguing as in the previous case, we get $r_{\alpha,\beta}^{1/q-1/p_\beta} \le r_{\gamma,\beta} ^{1/q-1/p_\beta}$. Since $p_\beta<q$, this implies $r_{\alpha,\beta}\ge r_{\gamma,\beta}$.
\end{enumerate}

{\bf Case 2a)} (case 2b) is similar). 

By (\ref{phi2}), it suffices to prove the estimate
\begin{align}
\label{dnm_ts}
d_n(M, \, l_{q,\sigma}^{m,k}) \underset{q,\sigma}{\gtrsim} \nu_\alpha^{1-\hat \lambda _{\alpha,\beta}} \nu_\beta ^{\hat \lambda _{\alpha,\beta}}.
\end{align}
We define the number $r_{\alpha,\beta}$ by the equation
\begin{align}
\label{r_ab_def_0}
\frac{\nu_\alpha}{\nu_\beta} = r_{\alpha,\beta}^{1/p_\alpha-1/p_\beta}
\end{align}
(it is well-defined, since $p_\alpha>p_\beta$). We show that
\begin{align}
\label{r_m_n_km} 1\le r_{\alpha,\beta} \le m, \quad n\le m^{\frac 2q}k^{\frac{2}{\sigma}} r_{\alpha,\beta}^{1-\frac 2q}.
\end{align}
Let
$$\lambda := \min \left\{\mu\in [0, \, \hat \lambda_{\alpha,\beta}]: \frac{1-\mu}{\theta_\alpha}+\frac{\mu}{\theta_\beta}\ge \frac{1}{\sigma}\right\},$$
$$
\tilde \lambda:= \max\left\{\mu\in [\hat \lambda_{\alpha,\beta}, \, 1]: \frac{1-\mu}{p_\alpha}+\frac{\mu}{p_\beta}\le \frac{1}{2}, \right. $$
$$\left. \frac{(1-\mu)/p_\alpha+\mu/p_\beta-1/q}{1/2-1/q}\le \frac{(1-\mu)/\theta_\alpha + \mu/\theta_\beta-1/\sigma}{1/2-1/\sigma}\right\}\; \text{for }q>2, \; \sigma >2,$$
$$
\tilde \lambda:= \max\left\{\mu\in [\hat \lambda_{\alpha,\beta}, \, 1]: \frac{1-\mu}{p_\alpha}+\frac{\mu}{p_\beta}\le \frac{1}{2},\; \frac{1-\mu}{\theta_\alpha} + \frac{\mu}{\theta_\beta} \ge \frac 12\right\}\; \text{for }q>2, \; \sigma =2,
$$
$$
\tilde \lambda:= \max\left\{\mu\in [\hat \lambda_{\alpha,\beta}, \, 1]: \frac{1-\mu}{\theta_\alpha} + \frac{\mu}{\theta_\beta} \ge \frac{1}{\sigma}\right\}\;\text{for }q=2.
$$
Then $\lambda < \hat \lambda_{\alpha,\beta} < \tilde \lambda$. The numbers $p_*$, $\theta_*$, $p_{**}$, $\theta_{**}$ are defined by (\ref{p_st1}), (\ref{p_st2}).
As in the previous case, $$\nu_\alpha^{1-\hat \lambda_{\alpha,\beta}} \nu_\beta^{\hat \lambda_{\alpha,\beta}} \Phi(q, \, \hat \theta_{\alpha,\beta}) \le \nu_\alpha ^{1-\lambda} \nu_\beta^\lambda \Phi(p_*, \, \theta_*),$$ $$\nu_\alpha^{1-\hat \lambda_{\alpha,\beta}} \nu_\beta^{\hat \lambda_{\alpha,\beta}} \Phi(q, \, \hat \theta_{\alpha,\beta}) \le \nu_\alpha ^{1-\tilde \lambda} \nu_\beta^{\tilde\lambda} \Phi(p_{**}, \, \theta_{**});$$ hence (see (\ref{phi2}), (\ref{phi4}) with $n\le mk^{2/\sigma}$ for $q>2$, and (\ref{phi2}), (\ref{phi5}), (\ref{phi6}) with $n\le m^{2/q}k^{2/\sigma}$ for $q=2$),
\begin{align}
\label{nal_l_nam}
\nu_\alpha^{1-\hat \lambda_{\alpha,\beta}} \nu_\beta^{\hat \lambda_{\alpha,\beta}} \le \nu_\alpha ^{1-\lambda} \nu_\beta^\lambda m^{\frac 1q -\frac{1-\lambda}{p_\alpha}-\frac{\lambda}{p_\beta}},
\end{align}
\begin{align}
\label{nal_l_nam1}
\nu_\alpha^{1-\hat \lambda_{\alpha,\beta}} \nu_\beta^{\hat \lambda_{\alpha,\beta}} \le \nu_\alpha ^{1-\tilde \lambda} \nu_\beta^{\tilde\lambda}\min \{1, \, n^{-1/2}m^{1/q}k^{1/\sigma}\}^{\frac{(1-\tilde\lambda)/p_\alpha + \tilde\lambda/p_\beta-1/q}{1/2-1/q}} \quad \text{for }q>2,
\end{align}
\begin{align}
\label{nal_l_nam2}
\nu_\alpha^{1-\hat \lambda_{\alpha,\beta}} \nu_\beta^{\hat \lambda_{\alpha,\beta}} \le \nu_\alpha ^{1-\tilde \lambda} \nu_\beta^{\tilde\lambda} \quad \text{for }q=2.
\end{align}
From (\ref{nal_l_nam}) we get that $\frac{\nu_\alpha}{\nu_\beta} \ge m^{1/p_\alpha-1/p_\beta}$, and from (\ref{nal_l_nam1}), (\ref{nal_l_nam2}), we obtain that $\frac{\nu_\alpha}{\nu_\beta}\le 1$ for $n\le m^{2/q}k^{2/\sigma}$, $\frac{\nu_\alpha}{\nu_\beta}\le (n^{1/2}m^{-1/q}k^{-1/\sigma})^{\frac{1/p_\alpha-1/p_\beta}{1/2-1/q}}$ for $n>m^{2/q}k^{2/\sigma}$ (since $n\le mk^{2/\sigma}$, the second case is possible only for $q>2$). This together with (\ref{r_ab_def_0}) and the condition $p_\alpha>p_\beta$ yields that $1\le r_{\alpha,\beta}\le m$ for $n\le m^{2/q}k^{2/\sigma}$, and $(n^{1/2}m^{-1/q}k^{-1/\sigma})^{\frac{1}{1/2-1/q}}\le r_{\alpha,\beta}\le m$ for $m^{2/q}k^{2/\sigma}< n \le mk^{2/\sigma}$. This completes the proof of (\ref{r_m_n_km}).

Let $r = \lceil r_{\alpha,\beta} \rceil$, $W = \nu_\alpha^{1-\hat \lambda_{\alpha,\beta}} \nu_\beta^{\hat \lambda_{\alpha,\beta}} r^{-1/q} V^{m,k}_{r,1}$. We show that $W \subset 2M$. Then
$$
d_n(M, \, l_{q,\sigma}^{m,k}) \gtrsim d_n(W, \, l_{q,\sigma}^{m,k}) \stackrel{(\ref{dn_vmk1}), (\ref{r_m_n_km})}{\underset{q,\sigma}{\gtrsim}} \nu_\alpha^{1-\hat \lambda_{\alpha,\beta}} \nu_\beta^{\hat \lambda_{\alpha,\beta}} r^{-1/q} r^{1/q} = \nu_\alpha^{1-\hat \lambda_{\alpha,\beta}} \nu_\beta^{\hat \lambda_{\alpha,\beta}};
$$
hence (\ref{dnm_ts}) holds.

In order to prove the inclusion $W \subset 2M$, it suffices to check that
\begin{align}
\label{nu_a_nu_b_g}
\nu_\alpha^{1-\hat \lambda_{\alpha,\beta}} \nu_\beta^{\hat \lambda_{\alpha,\beta}} r_{\alpha,\beta} ^{1/p_\gamma-1/q} \le \nu_\gamma, \quad \gamma \in A.
\end{align}
If $\gamma=\alpha$ or $\gamma=\beta$, it follows from (\ref{r_ab_def_0}).

Let $\gamma \notin \{\alpha, \, \beta\}$. We define the numbers $r_{\alpha,\gamma}$ and $r_{\gamma,\beta}$ by the equations
\begin{align}
\label{r_a_b_g} \frac{\nu_\alpha}{\nu_\gamma} = r_{\alpha,\gamma}^{1/p_\alpha-1/p_\gamma}, \quad \frac{\nu_\gamma}{\nu_\beta} = r_{\gamma,\beta} ^{1/p_\gamma-1/p_\beta}.
\end{align}

From (\ref{r_ab_def_0}) and (\ref{r_a_b_g}) it follows that (\ref{nu_a_nu_b_g}) is equivalent to each of the inequalities (\ref{incl_ineq_equiv}).

\begin{enumerate}
\item Let $p_\gamma<q$. We check the first inequality (\ref{incl_ineq_equiv}). Since $p_\alpha>q>p_\gamma$, it is equivalent to the inequality $r_{\alpha,\beta} \le r_{\alpha,\gamma}$.

There is $\hat \lambda_{\alpha,\gamma}\in (0, \, 1)$ such that $\frac{1-\hat \lambda_{\alpha,\gamma}}{p_\alpha} + \frac{\hat \lambda_{\alpha,\gamma}}{p_\gamma} = \frac 1q$; the number $\hat \theta_{\alpha,\gamma}$ is defined by (\ref{n1}).

If $\hat \theta_{\alpha,\gamma}<\sigma$, then from $$\nu_\alpha^{1-\hat \lambda_{\alpha,\beta}}\nu_\beta^{\hat \lambda_{\alpha,\beta}} \Phi(q, \, \hat \theta_{\alpha,\beta})=\Psi\le \Psi_1 \le \nu_\alpha^{1-\hat \lambda_{\alpha,\gamma}}\nu_\gamma^{\hat \lambda_{\alpha,\gamma}} \Phi(q, \, \hat \theta_{\alpha,\gamma})$$
we get
\begin{align}
\label{nu_a_s1}
\nu_\alpha\left(\frac{\nu_\beta}{\nu_\alpha}\right) ^{\hat \lambda_{\alpha,\beta}} \le \nu_\alpha\left(\frac{\nu_\gamma}{\nu_\alpha}\right) ^{\hat \lambda_{\alpha,\gamma}}.
\end{align}
Applying (\ref{r_ab_def_0}), (\ref{r_a_b_g}) and the equality $\frac{1}{q}-\frac{1}{p_\alpha} = \hat \lambda_{\alpha,\beta}\left(\frac{1}{p_\beta}-\frac{1}{p_\alpha}\right) = \hat \lambda_{\alpha,\gamma}\left(\frac{1}{p_\gamma}-\frac{1}{p_\alpha}\right)$, we get $r_{\alpha,\beta} ^{1/q-1/p_\alpha} \le r_{\alpha,\gamma} ^{1/q-1/p_\alpha}$. Since $p_\alpha>q$, we have $r_{\alpha,\beta} \le r_{\alpha,\gamma}$.

If $\hat \theta_{\alpha,\gamma}>\sigma$, there is $\mu\in (0, \, 1)$ such that $\frac{1-\mu}{\hat \theta_{\alpha,\beta}} + \frac{\mu}{\hat \theta _{\alpha,\gamma}} = \frac{1}{\sigma}$. Then from
$$
\nu_\alpha^{1-\hat \lambda_{\alpha,\beta}}\nu_\beta^{\hat \lambda_{\alpha,\beta}} \Phi(q, \, \hat\theta_{\alpha,\beta})=\Psi\le \Psi_6 \le (\nu_\alpha^{1-\hat \lambda_{\alpha,\beta}}\nu_\beta^{\hat \lambda_{\alpha,\beta}})^{1-\mu} (\nu_\alpha^{1-\hat \lambda_{\alpha,\gamma}}\nu_\gamma^{\hat \lambda_{\alpha,\gamma}})^\mu \Phi(q, \, \sigma)
$$
we get (\ref{nu_a_s1}), which again implies the inequality $r_{\alpha,\beta} \le r_{\alpha,\gamma}$.

\item Let $p_\gamma>q$. We check the second inequality (\ref{incl_ineq_equiv}). Since $p_\gamma>q>p_\beta$, it is equivalent to $r_{\alpha,\beta}\ge r_{\gamma,\beta}$. Arguing as in the previous case, we get $r_{\alpha,\beta}^{1/q-1/p_\beta} \le r_{\gamma,\beta} ^{1/q-1/p_\beta}$. Since $p_\beta<q$, this implies that $r_{\alpha,\beta}\ge r_{\gamma,\beta}$.
\end{enumerate}

{\bf Case 3a)} (case 3b) is similar). 

Let first $n\le km^{2/q}$. We show that $\Psi=\Psi_0$ or (\ref{psi_eq_psi2}) holds with $\hat p_{\alpha,\beta}<q$; i.e., we get one of the cases, which are already considered.

We set
\begin{align}
\label{7lam}
\lambda = \min \left\{ \mu \in [0, \, \tilde \lambda_{\alpha,\beta}]:\; \frac{1-\mu}{p_\alpha} + \frac{\mu}{p_\beta} \ge \frac 1q, \; \frac{1-\mu}{\theta_\alpha} + \frac{\mu}{\theta_\beta} \le \frac{1}{\sigma}\right\},
\end{align}
\begin{align}
\label{7til_lam}
\tilde\lambda = \max \left\{ \mu \in [\tilde \lambda_{\alpha,\beta}, \, 1]:\; \frac{1-\mu}{\theta_\alpha} + \frac{\mu}{\theta_\beta} \le \frac{1}{\sigma}\right\}.
\end{align}
The numbers $p_*$, $\theta_*$, $p_{**}$, $\theta_{**}$ are defined by (\ref{p_st1}), (\ref{p_st2}). From the inequalities
\begin{align}
\label{nu_a_phi_tl}
\begin{array}{l}
\nu_\alpha^{1-\tilde \lambda_{\alpha,\beta}} \nu_\beta^{\tilde \lambda_{\alpha,\beta}} \Phi(2, \, \tilde \theta_{\alpha,\beta}) \le \nu_\alpha^{1-\lambda} \nu_\beta^{\lambda} \Phi(p_*, \, \theta_*), \\ \nu_\alpha^{1-\tilde \lambda_{\alpha,\beta}} \nu_\beta^{\tilde \lambda_{\alpha,\beta}} \Phi(2, \, \tilde \theta_{\alpha,\beta}) \le \nu_\alpha^{1-\tilde\lambda} \nu_\beta^{\tilde\lambda} \Phi(p_{**}, \, \theta_{**})
\end{array}
\end{align}
we get that (see \eqref{phi3})
$$
\nu_\alpha^{1-\tilde \lambda_{\alpha,\beta}} \nu_\beta^{\tilde \lambda_{\alpha,\beta}} k^{\frac{1}{\sigma} -\frac{1-\tilde \lambda_{\alpha,\beta}}{\theta_\alpha} - \frac{\tilde \lambda_{\alpha,\beta}}{\theta_\beta}} \le \nu_\alpha^{1-\lambda} \nu_\beta^{\lambda} k^{\frac{1}{\sigma} -\frac{1-\lambda}{\theta_\alpha} - \frac{ \lambda}{\theta_\beta}},
$$
$$
\nu_\alpha^{1-\tilde \lambda_{\alpha,\beta}} \nu_\beta^{\tilde \lambda_{\alpha,\beta}} k^{\frac{1}{\sigma} -\frac{1-\tilde \lambda_{\alpha,\beta}}{\theta_\alpha} - \frac{\tilde \lambda_{\alpha,\beta}}{\theta_\beta}} \le \nu_\alpha^{1-\tilde\lambda} \nu_\beta^{\tilde\lambda} k^{\frac{1}{\sigma} -\frac{1-\tilde\lambda}{\theta_\alpha} - \frac{\tilde \lambda}{\theta_\beta}}.
$$
Hence $\frac{\nu_\alpha}{\nu_\beta} = k^{1/\theta_\alpha-1/\theta_\beta}$ and $\Psi = \nu_\alpha^{1-\tilde\lambda} \nu_\beta^{\tilde\lambda} \Phi(p_{**}, \, \theta_{**})$. If $\tilde\lambda=1$, then $\Psi=\Psi_0$; if $\tilde\lambda<1$, then (\ref{psi_eq_psi2}) holds with $\hat p_{\alpha,\beta}<q$.

Let now $n> km^{2/q}$. We show that
\begin{align}
\label{dn_3_s_til}
d_n(M, \, l^{m,k}_{q,\sigma}) \underset{q,\sigma}{\gtrsim} \nu_\alpha^{1-\tilde \lambda_{\alpha,\beta}} \nu_\beta^{\tilde \lambda_{\alpha,\beta}} k^{1/\sigma-1/\tilde \theta_{\alpha,\beta}} n^{-1/2}m^{1/q}k^{1/2}.
\end{align}

We define the number $r_{\alpha,\beta}$ by
\begin{align}
\label{nuab_55}
\frac{\nu_\alpha}{\nu_\beta} = r_{\alpha,\beta}^{1/p_\alpha-1/p_\beta} k^{1/\theta_\alpha-1/\theta_\beta}.
\end{align}
The numbers $\lambda$, $\tilde \lambda$ are defined by (\ref{7lam}), (\ref{7til_lam}), and the numbers $p_*$, $\theta_*$, $p_{**}$, $\theta_{**}$, by (\ref{p_st1}) and (\ref{p_st2}).
We get (\ref{nu_a_phi_tl}); applying (\ref{phi3}) and taking into account the inequalities $2<p_*\le q$, $p_{**}>2$, we get
$$
\nu_\alpha^{1-\tilde \lambda_{\alpha,\beta}} \nu_\beta^{\tilde \lambda_{\alpha,\beta}} k^{1/\sigma-1/\tilde \theta_{\alpha,\beta}} n^{-1/2}m^{1/q}k^{1/2} \le \nu_\alpha^{1-\lambda} \nu_\beta^{\lambda} k^{1/\sigma-1/\theta_*} (n^{-1/2}m^{1/q}k^{1/2})^{\frac{1/p_*-1/q}{1/2-1/q}},
$$
$$
\nu_\alpha^{1-\tilde \lambda_{\alpha,\beta}} \nu_\beta^{\tilde \lambda_{\alpha,\beta}} k^{1/\sigma-1/\tilde \theta_{\alpha,\beta}} n^{-1/2}m^{1/q}k^{1/2} \le \nu_\alpha^{1-\tilde\lambda} \nu_\beta^{\tilde\lambda} k^{1/\sigma-1/\theta_{**}} n^{-1/2}m^{1/q}k^{1/2}.
$$
Hence
$$
(n^{1/2}m^{-1/q}k^{-1/2})^{\frac{1/p_\alpha-1/p_\beta}{1/2-1/q}}k^{1/\theta_\alpha-1/\theta_\beta}\le \frac{\nu_\alpha}{\nu_\beta} \le k^{1/\theta_\alpha-1/\theta_\beta}.
$$
This together with (\ref{nuab_55}) yields that $1\le r_{\alpha,\beta}\le (n^{1/2}m^{-1/q}k^{-1/2})^{\frac{1}{1/2-1/q}}\le m$ and
\begin{align}
\label{nm2q_r_ab_12q}
n\ge m^{\frac 2q}k^{\frac{2}{\sigma}} r_{\alpha,\beta}^{1-\frac 2q} k^{1-\frac{2}{\sigma}}.
\end{align}

We set $r=\lfloor r_{\alpha,\beta}\rfloor$, $W=\nu_\alpha^{1-\tilde \lambda_{\alpha,\beta}} \nu_\beta^{\tilde \lambda_{\alpha,\beta}} r^{-1/2}k^{-1/\tilde \theta_{\alpha,\beta}} V^{m,k}_{r,k}$. If $W\subset 2M$, then
$$
d_n(M, \, l^{m,k}_{q,\sigma}) \gtrsim d_n(W, \, l^{m,k}_{q,\sigma}) \stackrel{(\ref{dn_vmk1}), (\ref{nm2q_r_ab_12q})}{\underset{q,\sigma}{\gtrsim}} \nu_\alpha^{1-\tilde \lambda_{\alpha,\beta}} \nu_\beta^{\tilde \lambda_{\alpha,\beta}} r^{-1/2}k^{-1/\tilde \theta_{\alpha,\beta}} n^{-1/2} m^{1/q} k^{1/\sigma} r^{1/2} k^{1/2}=
$$
$$
=\nu_\alpha^{1-\tilde \lambda_{\alpha,\beta}} \nu_\beta^{\tilde \lambda_{\alpha,\beta}} k^{1/\sigma-1/\tilde \theta_{\alpha,\beta}} n^{-1/2} m^{1/q} k^{1/2};
$$
hence (\ref{dn_3_s_til}) holds.

In order to prove the inclusion $W\subset 2M$, it suffices to check the inequality
\begin{align}
\label{1vkl1} \nu_\alpha^{1-\tilde \lambda_{\alpha,\beta}} \nu_\beta^{\tilde \lambda_{\alpha,\beta}} r_{\alpha,\beta}^{1/p_\gamma-1/2} k^{1/\theta_\gamma-1/\tilde \theta_{\alpha,\beta}} \le \nu_{\gamma}, \quad \gamma \in A.
\end{align}
If $\gamma=\alpha$ or $\gamma=\beta$, it follows from (\ref{nuab_55}).

Let $\gamma \notin \{\alpha, \, \beta\}$. We define the numbers $r_{\alpha,\gamma}$ and $r_{\gamma,\beta}$ by
\begin{align}
\label{nu_ag_r_agk}
\frac{\nu_\alpha}{\nu_\gamma} =r_{\alpha,\gamma}^{1/p_\alpha -1/p_\gamma} k^{1/\theta_\alpha-1/\theta_\gamma}, \quad \frac{\nu_\gamma}{\nu_\beta} =r_{\gamma,\beta}^{1/p_\gamma -1/p_\beta} k^{1/\theta_\gamma-1/\theta_\beta}.
\end{align}
Then (\ref{1vkl1}) is equivalent to each of the following inequalities:
$$
r_{\alpha,\beta}^{1/p_\gamma-1/p_\alpha} \le r_{\alpha,\gamma}^{1/p_\gamma-1/p_\alpha}, \quad r_{\alpha,\beta}^{1/p_\gamma-1/p_\beta} \le r_{\gamma,\beta}^{1/p_\gamma-1/p_\beta}.
$$
If $p_\gamma<2$, we check the first inequality, i.e., $r_{\alpha,\beta}\le r_{\alpha,\gamma}$. If $p_\gamma>2$, we check the second inequality, i.e., $r_{\alpha,\beta} \ge r_{\gamma,\beta}$.

Let $p_\gamma<2$. If $\tilde \theta_{\alpha,\gamma}>2$, then from the inequality $\nu_\alpha^{1-\tilde \lambda_{\alpha,\beta}} \nu_\beta^{\tilde \lambda_{\alpha,\beta}}\Phi(2, \, \tilde \theta_{\alpha,\beta}) \le \nu_\alpha^{1-\tilde \lambda_{\alpha,\gamma}} \nu_\gamma^{\tilde \lambda_{\alpha,\gamma}}\Phi(2, \, \tilde \theta_{\alpha,\gamma})$ we get (see (\ref{phi3}), (\ref{phi5}))
$$
\nu_\alpha^{1-\tilde \lambda_{\alpha,\beta}} \nu_\beta^{\tilde \lambda_{\alpha,\beta}}k^{1/\sigma-1/\tilde \theta_{\alpha,\beta}}n^{-1/2}m^{1/q}k^{1/2} \le \nu_\alpha^{1-\tilde \lambda_{\alpha,\gamma}} \nu_\gamma^{\tilde \lambda_{\alpha,\gamma}}k^{1/\sigma-1/\tilde \theta_{\alpha,\gamma}}n^{-1/2}m^{1/q}k^{1/2};
$$
this together with (\ref{nuab_55}), (\ref{nu_ag_r_agk}) implies that
$r_{\alpha,\beta}^{1/2-1/p_\alpha} \le r_{\alpha,\gamma}^{1/2-1/p_\alpha}$; hence $r_{\alpha,\beta}\le r_{\alpha,\gamma}$. If $\tilde \theta_{\alpha,\gamma}<2$, then there exists $\mu \in (0, \, 1)$ such that $\frac 12 =\frac{1-\mu}{\tilde \theta_{\alpha,\beta}} +\frac{\mu}{\theta_{\alpha,\gamma}}$. From $$\nu_\alpha^{1-\tilde \lambda_{\alpha,\beta}} \nu_\beta^{\tilde \lambda_{\alpha,\beta}}\Phi(2, \, \tilde \theta_{\alpha,\beta})=\Psi\le \Psi_7 \le (\nu_\alpha^{1-\tilde \lambda_{\alpha,\beta}} \nu_\beta^{\tilde \lambda_{\alpha,\beta}})^{1-\mu}(\nu_\alpha^{1-\tilde \lambda_{\alpha,\gamma}} \nu_\gamma^{\tilde \lambda_{\alpha,\gamma}})^\mu \Phi(2, \, 2)$$
we get $r_{\alpha,\beta}^{1/2-1/p_\alpha} \le r_{\alpha,\gamma}^{1/2-1/p_\alpha}$; i.e., $r_{\alpha,\beta} \le r_{\alpha,\gamma}$.

If $p_\gamma>2$, we similarly get $r_{\alpha,\beta}^{1/2-1/p_\beta} \le r_{\gamma,\beta}^{1/2-1/p_\beta}$; hence $r_{\alpha,\beta} \ge r_{\gamma,\beta}$.

{\bf Case 4a)} (case 4b) is similar). We have $\sigma>2$.

First we consider the case $q>2$. We show that either $\Psi = \Psi_0 = \nu_\alpha \Phi(p_\alpha, \, \theta_\alpha)$, or (\ref{psi_eq_psi2}) holds with $\hat p_{\alpha, \beta} > q$, or (\ref{psi4_psi}) holds with $\tilde p_{\alpha,\beta}>q$. Then the proof of the lower estimate can be reduced to the cases, which are already considered.

For $\hat \theta_{\alpha,\beta}>2$, we set
$$
\lambda = \min \left\{ \mu \in [0, \, \hat \lambda_{\alpha,\beta}]:\; \frac{1-\mu}{\theta_\alpha}+ \frac{\mu}{\theta_\beta} \in \left[\frac{1}{\sigma}, \, \frac{1}{2}\right]\right\},
$$
$$
\begin{array}{c}
\tilde\lambda = \max \left\{ \mu \in [\hat \lambda_{\alpha,\beta}, \, 1]:\; \frac{1-\mu}{\theta_\alpha}+ \frac{\mu}{\theta_\beta} \le \frac 12, \right. \\
\left. \frac{(1-\mu)/p_\alpha+\mu/p_\beta-1/q} {1/2-1/q}\le \frac{(1-\mu)/\theta_\alpha+\mu/\theta_\beta-1/\sigma}{1/2-1/\sigma}\right\};
\end{array}
$$
for $\hat \theta_{\alpha,\beta}<2$, we set
$$
\lambda = \min \left\{ \mu \in [0, \, \hat \lambda_{\alpha,\beta}]:\; \frac{1-\mu}{\theta_\alpha}+ \frac{\mu}{\theta_\beta} \ge \frac 12\right\},
$$
$$
\tilde\lambda = \max \left\{ \mu \in [\hat \lambda_{\alpha,\beta}, \, 1]:\; \frac{1-\mu}{p_\alpha}+ \frac{\mu}{p_\beta} \le \frac{1}{2}, \; \frac{1-\mu}{\theta_\alpha}+ \frac{\mu}{\theta_\beta} \ge \frac 12 \right\};
$$
the numbers $p_*$, $\theta_*$, $p_{**}$, $\theta_{**}$ are defined by (\ref{p_st1}), (\ref{p_st2}). From the inequalities $$\nu_\alpha^{1-\hat \lambda_{\alpha,\beta}} \nu_\beta^{\hat \lambda_{\alpha,\beta}} \Phi(q, \, \hat \theta_{\alpha,\beta}) \le \nu_\alpha^{1-\lambda} \nu_\beta^\lambda \Phi(p_*, \, \theta_*),$$ $$\nu_\alpha^{1-\hat \lambda_{\alpha,\beta}} \nu_\beta^{\hat \lambda_{\alpha,\beta}} \Phi(q, \, \hat \theta_{\alpha,\beta}) \le \nu_\alpha^{1-\tilde\lambda} \nu_\beta^{\tilde\lambda} \Phi(p_{**}, \, \theta_{**})$$ we get:
\begin{enumerate}
\item If $\hat\theta_{\alpha,\beta}>2$, then (see (\ref{phi2}), (\ref{phi4}))
\begin{align}
\label{5678}
\begin{array}{c}
\nu_\alpha^{1-\hat \lambda_{\alpha,\beta}} \nu_\beta^{\hat \lambda_{\alpha,\beta}} (n^{-1/2} m^{1/2}k^{1/\sigma})^{\frac{(1-\hat \lambda_{\alpha,\beta})/\theta_\alpha+\hat \lambda_{\alpha,\beta}/\theta_\beta-1/\sigma}{1/2-1/\sigma}} \le \\\le\nu_\alpha^{1-\lambda} \nu_\beta^{\lambda} m^{\frac 1q -\frac{1-\lambda}{p_\alpha} -\frac{\lambda}{p_\beta}} (n^{-1/2} m^{1/2}k^{1/\sigma})^{\frac{(1-\lambda)/\theta_\alpha+ \lambda/\theta_\beta-1/\sigma}{1/2-1/\sigma}},
\end{array}
\end{align}
$$\nu_\alpha^{1-\hat \lambda_{\alpha,\beta}} \nu_\beta^{\hat \lambda_{\alpha,\beta}} (n^{-1/2} m^{1/2}k^{1/\sigma})^{\frac{(1-\hat \lambda_{\alpha,\beta})/\theta_\alpha+\hat \lambda_{\alpha,\beta}/\theta_\beta-1/\sigma}{1/2-1/\sigma}} \le$$$$\le \nu_\alpha^{1-\tilde\lambda} \nu_\beta^{\tilde\lambda} m^{\frac 1q -\frac{1-\tilde\lambda}{p_\alpha} -\frac{\tilde\lambda}{p_\beta}} (n^{-1/2} m^{1/2}k^{1/\sigma})^{\frac{(1-\tilde\lambda)/\theta_\alpha+ \tilde\lambda/\theta_\beta-1/\sigma}{1/2-1/\sigma}}.
$$
This implies that
$$
\frac{\nu_\alpha}{\nu_\beta} = m^{1/p_\alpha-1/p_\beta} (n^{1/2} m^{-1/2}k^{-1/\sigma})^{\frac{1/\theta_\alpha-1/\theta_\beta}{1/2-1/\sigma}}.
$$
Hence $\Psi = \nu_\alpha^{1-\lambda}\nu_\beta^\lambda \Phi(p_*, \, \theta_*)$. If $2<\theta_\alpha<\sigma$, then $p_*=p_\alpha$, $\theta_*=\theta_\alpha$. If $\theta_\alpha>\sigma$, then $p_*=\hat p_{\alpha, \, \beta} \in (q, \, p_\alpha)$, $\theta_* = \sigma$. If $\theta_\alpha<2$, then $p_*=\tilde p_{\alpha, \, \beta} \in (q, \, p_\alpha)$, $\theta_* = 2$.

\item If $\hat \theta_{\alpha,\beta}<2$, then
\begin{align}
\label{6789}
\nu_\alpha^{1-\hat \lambda_{\alpha,\beta}} \nu_\beta^{\hat \lambda_{\alpha,\beta}} n^{-1/2} m^{1/2}k^{1/\sigma} \le \nu_\alpha^{1- \lambda} \nu_\beta^{\lambda} m^{\frac 1q -\frac{1-\lambda}{p_\alpha} -\frac{\lambda}{p_\beta}}n^{-1/2} m^{1/2}k^{1/\sigma},
\end{align}
$$
\nu_\alpha^{1-\hat \lambda_{\alpha,\beta}} \nu_\beta^{\hat \lambda_{\alpha,\beta}} n^{-1/2} m^{1/2}k^{1/\sigma} \le \nu_\alpha^{1- \tilde\lambda} \nu_\beta^{\tilde\lambda} m^{\frac 1q -\frac{1-\tilde\lambda}{p_\alpha} -\frac{\tilde\lambda}{p_\beta}}n^{-1/2} m^{1/2}k^{1/\sigma}.
$$
Hence $\frac{\nu_\alpha}{\nu_\beta}=m^{1/p_\alpha-1/p_\beta}$, and again we have $\Psi = \nu_\alpha^{1-\lambda}\nu_\beta^\lambda \Phi(p_*, \, \theta_*)$. If $\theta_\alpha<2$, then $p_*=p_\alpha$, $\theta_* = \theta_\alpha$. If $\theta_\alpha>2$, then $p_*=\tilde p_{\alpha, \, \beta} \in (q, \, p_\alpha)$, $\theta_* = 2$.
\end{enumerate}

Let now $q=2$. We show that (see \eqref{phi2})
\begin{align}
\label{dn_m_lqs} d_n(M, \, l^{m,k}_{q,\sigma}) \underset{q,\sigma}{\gtrsim} \nu_\alpha^{1-\hat \lambda_{\alpha,\beta}} \nu_\beta^{\hat \lambda _{\alpha,\beta}} (n^{-1/2}m^{1/2}k^{1/\sigma}) ^{\omega_{\hat \theta_{\alpha,\beta},\sigma}}.
\end{align}

We set $\tilde l = (n^{1/2}m^{-1/2}k^{-1/\sigma})^{\frac{1}{1/2-1/\sigma}}$ for $\hat \theta_{\alpha,\beta}>2$, and $\tilde l=1$ for $\hat \theta_{\alpha,\beta}<2$. The number $r_{\alpha,\beta}$ is defined by the equation
\begin{align}
\label{nu_ab_4a}
\frac{\nu_\alpha}{\nu_\beta} = r_{\alpha,\beta}^{1/p_\alpha-1/p_\beta} \tilde l^{1/\theta_\alpha-1/\theta_\beta}.
\end{align}
For $\hat \theta_{\alpha,\beta}>2$, we set 
$$
\lambda=\min \left\{ \mu\in [0, \, \hat \lambda_{\alpha,\beta}]:\; \frac{1-\mu}{\theta_\alpha}+\frac{\mu}{\theta_\beta} \in \left[\frac{1}{\sigma}, \, \frac 12\right]\right\},
$$
$$
\tilde\lambda=\max \left\{ \mu\in [\hat \lambda_{\alpha,\beta}, \, 1]:\; \frac{1-\mu}{\theta_\alpha}+\frac{\mu}{\theta_\beta} \in \left[\frac{1}{\sigma}, \, \frac 12\right]\right\},
$$
and for $\hat \theta_{\alpha,\beta}<2$,
$$
\lambda=\min \left\{ \mu\in [0, \, \hat \lambda_{\alpha,\beta}]:\; \frac{1-\mu}{\theta_\alpha}+\frac{\mu}{\theta_\beta} \ge \frac 12\right\},
$$
$$
\tilde\lambda=\max \left\{ \mu\in [\hat \lambda_{\alpha,\beta}, \, 1]:\; \frac{1-\mu}{\theta_\alpha}+\frac{\mu}{\theta_\beta} \ge \frac 12\right\}.
$$

The numbers $p_{**}$, $\theta_{**}$ are defined by (\ref{p_st1}), (\ref{p_st2}). From the inequalities $$\nu_\alpha^{1-\hat \lambda_{\alpha,\beta}} \nu_\beta^{\hat \lambda_{\alpha,\beta}} \Phi(q, \, \hat \theta_{\alpha,\beta}) \le \nu_\alpha^{1-\lambda} \nu_\beta^\lambda \Phi(p_*, \, \theta_*),$$ $$\nu_\alpha^{1-\hat \lambda_{\alpha,\beta}} \nu_\beta^{\hat \lambda_{\alpha,\beta}} \Phi(q, \, \hat \theta_{\alpha,\beta}) \le \nu_\alpha^{1-\tilde\lambda} \nu_\beta^{\tilde\lambda} \Phi(p_{**}, \, \theta_{**})$$ we get:
\begin{enumerate}
\item If $\hat \theta_{\alpha,\beta}>2$, then (see (\ref{phi2}), (\ref{phi5}) with $q=2$) the inequality (\ref{5678}) holds and 
$$\nu_\alpha^{1-\hat \lambda_{\alpha,\beta}} \nu_\beta^{\hat \lambda_{\alpha,\beta}} (n^{-1/2} m^{1/2}k^{1/\sigma})^{\frac{(1-\hat \lambda_{\alpha,\beta})/\theta_\alpha+\hat \lambda_{\alpha,\beta}/\theta_\beta-1/\sigma}{1/2-1/\sigma}} \le$$$$\le \nu_\alpha^{1-\tilde\lambda} \nu_\beta^{\tilde\lambda} (n^{-1/2} m^{1/2}k^{1/\sigma})^{\frac{(1-\tilde\lambda)/\theta_\alpha+ \tilde\lambda/\theta_\beta-1/\sigma}{1/2-1/\sigma}}.
$$
Hence
$m^{1/p_\alpha-1/p_\beta}\tilde l^{1/\theta_\alpha-1/\theta_\beta}\le \frac{\nu_\alpha}{\nu_\beta} \le \tilde l^{1/\theta_\alpha-1/\theta_\beta}$. This together with (\ref{nu_ab_4a}) yields that $1\le r_{\alpha,\beta}\le m$. In addition, $1\le \tilde l \le k$, $n = mk^{\frac{2}{\sigma}} \tilde l^{1-\frac{2}{\sigma}}$.

We set $r=\lceil r_{\alpha,\beta}\rceil$, $l=\lceil \tilde l \rceil$, $W=\nu_\alpha^{1-\hat \lambda_{\alpha,\beta}} \nu_\beta^{\hat \lambda_{\alpha,\beta}}r^{-1/q}l^{-1/\hat{\theta}_{\alpha,\beta}} V^{m,k}_{r,l}$. Then, if $W\subset 4M$, we get
$$
d_n(M, \, l^{m,k}_{q,\sigma}) \gtrsim d_n(W, \, l^{m,k}_{q,\sigma}) \stackrel{(\ref{dn_vmk1})}{\underset{q,\sigma}{\gtrsim}} \nu_\alpha^{1-\hat \lambda_{\alpha,\beta}} \nu_\beta^{\hat \lambda_{\alpha,\beta}}r^{-1/q}l^{-1/\hat{\theta}_{\alpha,\beta}} r^{1/q} l^{1/\sigma} \asymp
$$
$$
\asymp \nu_\alpha^{1-\hat \lambda_{\alpha,\beta}} \nu_\beta^{\hat \lambda_{\alpha,\beta}} (n^{-1/2} m^{1/2}k^{1/\sigma})^{\frac{1/\hat \theta_{\alpha,\beta}-1/\sigma}{1/2-1/\sigma}};
$$
hence (\ref{dn_m_lqs}) holds.
\item If $\hat \theta_{\alpha,\beta}<2$, then (\ref{6789}) holds and (see (\ref{phi6}) with $q=2$)
$$
\nu_\alpha^{1-\hat \lambda_{\alpha,\beta}} \nu_\beta^{\hat \lambda_{\alpha,\beta}} n^{-1/2} m^{1/2}k^{1/\sigma} \le \nu_\alpha^{1- \tilde\lambda} \nu_\beta^{\tilde\lambda}n^{-1/2} m^{1/2}k^{1/\sigma};
$$
again we get $m^{1/p_\alpha-1/p_\beta}\tilde l^{1/\theta_\alpha-1/\theta_\beta}\le \frac{\nu_\alpha}{\nu_\beta} \le \tilde l^{1/\theta_\alpha-1/\theta_\beta}$ and $1\le r_{\alpha,\beta}\le m$. Let $r =\lfloor r_{\alpha,\beta}\rfloor$, $W=\nu_\alpha^{1-\hat \lambda_{\alpha,\beta}} \nu_\beta^{\hat \lambda_{\alpha,\beta}} r^{-1/2}V^{m,k}_{r,1}$. If $W \subset 4M$, then from the inequality $n\ge mk^{2/\sigma}$ we get
$$
d_n(M, \, l^{m,k}_{q,\sigma}) \gtrsim d_n(W, \, l^{m,k}_{q,\sigma}) \stackrel{(\ref{dn_vmk1})}{\underset{q,\sigma}{\gtrsim}} \nu_\alpha^{1-\hat \lambda_{\alpha,\beta}} \nu_\beta^{\hat \lambda_{\alpha,\beta}}r^{-1/2}n^{-1/2}m^{1/2}k^{1/\sigma}r^{1/2}=$$$$=\nu_\alpha^{1-\hat \lambda_{\alpha,\beta}} \nu_\beta^{\hat \lambda_{\alpha,\beta}}n^{-1/2}m^{1/2}k^{1/\sigma}.
$$
\end{enumerate}
The inclusion $W\subset 4M$ can be proved as in the cases 2a) and 3a); the numbers $r_{\alpha,\gamma}$ and $r_{\gamma,\beta}$ are defined by the equations $$\frac{\nu_\alpha}{\nu_\gamma} = r_{\alpha,\gamma}^{1/p_\alpha-1/p_\gamma} \tilde l^{1/\theta_\alpha-1/\theta_\gamma}, \quad \frac{\nu_\gamma}{\nu_\beta}= r_{\gamma,\beta}^{1/p_\gamma-1/p_\beta} \tilde l^{1/\theta_\gamma-1/\theta_\beta}.$$
Here we apply the inequalities
$$
\nu_\alpha^{1-\hat \lambda_{\alpha,\beta}} \nu_\beta^{\hat \lambda_{\alpha,\beta}} \tilde l^{1/\sigma-1/\hat\theta_{\alpha,\beta}} \le \nu_\alpha^{1-\hat \lambda_{\alpha,\gamma}} \nu_\gamma^{\hat \lambda_{\alpha,\gamma}} \tilde l^{1/\sigma-\hat\theta_{\alpha,\gamma}} \quad \text{for }p_\gamma<2,
$$
$$
\nu_\alpha^{1-\hat \lambda_{\alpha,\beta}} \nu_\beta^{\hat \lambda_{\alpha,\beta}} \tilde l^{1/\sigma-1/\hat\theta_{\alpha,\beta}} \le \nu_\gamma^{1-\hat \lambda_{\gamma,\beta}} \nu_\beta^{\hat \lambda_{\gamma,\beta}} \tilde l^{1/\sigma-\hat\theta_{\gamma,\beta}} \quad \text{for }p_\gamma>2,
$$
which follow from (\ref{psi_eq_psi1}).

Hence, the consideration of the cases $\Psi=\Psi_1$ and $\Psi=\Psi_2$ is completed.

Now we consider the cases 5a), b) and 6a), b). First we notice that if $n \le m^{2/q}k^{2/\sigma}$, then $\Psi=\Psi_0$, $\Psi=\Psi_1$ or $\Psi=\Psi_2$; i.e., the estimating can be reduced to the cases which are already considered. To this end, we argue as in case 3a) and get $\frac{\nu_\alpha}{\nu_\beta}=1$; here we use the equality $\Phi(p, \, \theta)=1$ for $1\le p\le q$, $1\le \theta \le \sigma$.

In what follows, we suppose that $n > m^{2/q}k^{2/\sigma}$.

{\bf Case 5a)} (case 5b) is similar). 

We prove the estimate
\begin{align}
\label{dn_s_l_2} d_n(M, \, l^{m,k}_{q,\sigma}) \underset{q,\sigma}{\gtrsim} \nu_\alpha^{1- \tilde \lambda_{\alpha,\beta}} \nu_\beta^{\tilde \lambda _{\alpha,\beta}} n^{-1/2}m^{1/q}k^{1/\sigma}.
\end{align}

Let the number $r_{\alpha,\beta}$ be defined by the equation
\begin{align}
\label{2_r_ab_def} \frac{\nu_\alpha}{\nu_\beta} = r_{\alpha,\beta}^{1/p_\alpha-1/p_\beta}.
\end{align}
We show that
\begin{align}
\label{r_ineq_gn} 1\le r_{\alpha,\beta} \le m, \quad n\ge m^{\frac 2q} k^{\frac{2}{\sigma}}r_{\alpha,\beta} ^{1-\frac 2q}. 
\end{align}

Let
$$\lambda := \min\left\{ \mu \in [0, \, \tilde \lambda _{\alpha,\beta}]:\; \frac{1-\mu}{p_\alpha} + \frac{\mu}{p_\beta} \ge \frac 1q, \; \frac{1-\mu}{\theta_\alpha} + \frac{\mu}{\theta_\beta} \ge \frac 12\right\},$$ 
$$
\tilde \lambda:= \max \left\{\mu \in [\tilde \lambda _{\alpha,\beta}, \, 1]: \;  \frac{1-\mu}{\theta_\alpha} + \frac{\mu}{\theta_\beta} \ge \frac 12\right\}. 
$$

If $m^{2/q}k^{2/\sigma}< n \le mk^{2/\sigma}$, then (see \eqref{phi4}, \eqref{phi6})
$$
\nu_\alpha^{1-\tilde \lambda_{\alpha,\beta}} \nu_\beta ^{\tilde \lambda_{\alpha,\beta}} n^{-1/2} m^{1/q} k^{1/\sigma} \le \nu_\alpha^{1-\lambda} \nu_\beta^{\lambda} (n^{-1/2} m^{1/q} k^{1/\sigma}) ^{\frac{(1-\lambda)/p_\alpha +\lambda/p_\beta-1/q}{1/2-1/q}},
$$
$$
\nu_\alpha^{1-\tilde \lambda_{\alpha,\beta}} \nu_\beta ^{\tilde \lambda_{\alpha,\beta}} n^{-1/2} m^{1/q} k^{1/\sigma} \le \nu_\alpha^{1-\tilde\lambda} \nu_\beta^{\tilde\lambda} n^{-1/2} m^{1/q} k^{1/\sigma}.
$$
Hence
$$
(n^{1/2} m^{-1/q} k^{-1/\sigma})^{\frac{1/p_\alpha-1/p_\beta}{1/2-1/q}}\le\frac{\nu_\alpha}{\nu_\beta} \le 1.
$$
This together with (\ref{2_r_ab_def}) and the condition $p_\alpha>p_\beta$ yields that $$1\le r_{\alpha,\beta} \le \left(n^{1/2}m^{-1/q}k^{-1/\sigma}\right)^{\frac{1}{1/2-1/q}};$$ this implies (\ref{r_ineq_gn}).

If $n> mk^{2/\sigma}$, then (see (\ref{phi4}), (\ref{phi6}))
$$
\nu_\alpha^{1-\tilde \lambda_{\alpha,\beta}} \nu_\beta ^{\tilde \lambda_{\alpha,\beta}} n^{-1/2} m^{1/q} k^{1/\sigma} \le \nu_\alpha^{1-\lambda} \nu_\beta^{\lambda} m^{\frac{1}{q}-\frac{1-\lambda}{p_\alpha} -\frac{\lambda}{p_\beta}}n^{-\frac 12} m^{\frac 12} k^{\frac{1}{\sigma}},
$$
$$
\nu_\alpha^{1-\tilde \lambda_{\alpha,\beta}} \nu_\beta ^{\tilde \lambda_{\alpha,\beta}} n^{-1/2} m^{1/q} k^{1/\sigma} \le \nu_\alpha^{1-\tilde\lambda} \nu_\beta^{\tilde\lambda} n^{-1/2} m^{1/q} k^{1/\sigma}.
$$
Hence, $m^{1/p_\alpha-1/p_\beta}\le \frac{\nu_\alpha}{\nu_\beta} \le 1$; therefore, $1\le r_{\alpha,\beta} \le m$. Taking into account the condition $n\ge mk^{2/\sigma}$, we get the second inequality of (\ref{r_ineq_gn}).

Let $r=\lfloor r_{\alpha,\beta}\rfloor$, $W = \nu_\alpha^{1-\tilde \lambda_{\alpha,\beta}} \nu_\beta ^{\tilde \lambda_{\alpha,\beta}} r^{-1/2} V^{m,k}_{r,1}$. We show that $W\subset 2M$. Then
$$
d_n(M, \, l_{q,\sigma}^{m,k}) \gtrsim d_n(W, \, l_{q,\sigma}^{m,k}) \stackrel{(\ref{dn_vmk1}), (\ref{r_ineq_gn})}{\underset{q,\sigma}{\gtrsim}} \nu_\alpha^{1-\tilde \lambda_{\alpha,\beta}} \nu_\beta ^{\tilde \lambda_{\alpha,\beta}} r^{-1/2} n^{-1/2}m^{1/q} k^{1/\sigma} r^{1/2}=
$$
$$
=\nu_\alpha^{1-\tilde \lambda_{\alpha,\beta}} \nu_\beta ^{\tilde \lambda_{\alpha,\beta}} n^{-1/2}m^{1/q} k^{1/\sigma};
$$
i.e., (\ref{dn_s_l_2}) holds. It suffices to check that
\begin{align}
\label{incl_2_ineq} \nu_\alpha^{1-\tilde \lambda_{\alpha,\beta}} \nu_\beta ^{\tilde \lambda_{\alpha,\beta}} r_{\alpha,\beta}^{1/p_\gamma-1/2} \le \nu_\gamma, \quad \gamma \in A.
\end{align}
If $\gamma=\alpha$ or $\gamma=\beta$, it follows from (\ref{2_r_ab_def}).

Let $\gamma \notin \{\alpha, \, \beta\}$. We define the numbers $r_{\alpha,\gamma}$ and $r_{\gamma,\beta}$ by the equations
$$
\frac{\nu_\alpha}{\nu_\gamma} = r_{\alpha,\gamma} ^{1/p_\alpha-1/p_\gamma}, \quad \frac{\nu_\gamma}{\nu_\beta} = r_{\gamma,\beta}^{1/p_\gamma-1/p_\beta}.
$$

Then (\ref{incl_2_ineq}) is equivalent to each of the following inequalities:
\begin{align}
\label{r_ab_bg} r_{\alpha,\beta}^{1/p_\gamma-1/p_\alpha} \le r_{\alpha,\gamma}^{1/p_\gamma -1/p_\alpha}, \quad r_{\alpha,\beta}^{1/p_\gamma-1/p_\beta} \le r_{\gamma,\beta}^{1/p_\gamma -1/p_\beta}.
\end{align}

Let $p_\gamma<2$. We check the first inequality of (\ref{r_ab_bg}). Since $p_\alpha>2>p_\gamma$, it is equivalent to the condition $r_{\alpha,\beta}\le r_{\alpha,\gamma}$. If $\tilde \theta_{\alpha,\gamma}<2$, then from the relations $\nu_\alpha^{1-\tilde \lambda_{\alpha,\beta}} \nu_\beta^{\tilde \lambda_{\alpha,\beta}}\Phi(2, \, \tilde \theta_{\alpha,\beta}) =\Psi\le \Psi_3\le \nu_\alpha^{1-\tilde \lambda_{\alpha,\gamma}} \nu_\gamma^{\tilde \lambda_{\alpha,\gamma}}\Phi(2, \, \tilde \theta_{\alpha,\gamma})$ we get
$$
\nu_\alpha^{1-\tilde \lambda_{\alpha,\beta}} \nu_\beta^{\tilde \lambda_{\alpha,\beta}}n^{-1/2}m^{1/q}k^{1/\sigma} \le \nu_\alpha^{1-\tilde \lambda_{\alpha,\gamma}} \nu_\gamma^{\tilde \lambda_{\alpha,\gamma}}n^{-1/2}m^{1/q}k^{1/\sigma}.
$$
It is equivalent to the inequality 
\begin{align}
\label{na_mk_s}
r_{\alpha,\beta}^{\tilde \lambda_{\alpha,\beta}(1/p_\beta-1/p_\alpha)} \le r_{\alpha,\gamma}^{\tilde \lambda_{\alpha,\gamma}(1/p_\gamma-1/p_\alpha)};
\end{align}
 i.e., $r_{\alpha,\beta}^{1/2-1/p_\alpha}\le r_{\alpha,\gamma}^{1/2-1/p_\alpha}$. Since $p_\alpha>2$, we get $r_{\alpha,\beta}\le r_{\alpha,\gamma}$.

Let $\tilde\theta_{\alpha,\gamma}<2$. We define $\mu\in (0, \, 1)$ by the equation $\frac 12 = \frac{1-\mu}{\tilde \theta_{\alpha,\beta}} + \frac{\mu}{\tilde \theta_{\alpha,\gamma}}$. From the relations $\nu_\alpha^{1-\tilde \lambda_{\alpha,\beta}} \nu_\beta^{\tilde \lambda_{\alpha,\beta}} \Phi(2, \, \tilde \theta_{\alpha,\beta})=\Psi\le \Psi_7 \le (\nu_\alpha^{1-\tilde \lambda_{\alpha,\beta}} \nu_\beta^{\tilde \lambda_{\alpha,\beta}})^{1-\mu}(\nu_\alpha^{1-\tilde \lambda_{\alpha,\gamma}} \nu_\gamma^{\tilde \lambda_{\alpha,\gamma}})^\mu \Phi(2, \, 2)$ we again get (\ref{na_mk_s}).

Let $p_\gamma>2$. We check the second inequality of (\ref{r_ab_bg}). Since $p_\gamma>2>p_\beta$, it is equivalent to $r_{\alpha,\beta}\ge r_{\gamma,\beta}$. Arguing as for $p_\gamma<2$, we get the inequality $r_{\alpha,\beta}^{1/2-1/p_\beta}\le r_{\gamma,\beta}^{1/2-1/p_\beta}$. It remain to apply the inequality $p_\beta<2$.

{\bf Case 6a)} (case 6b) is similar). We have $q>2$, $\sigma>2$.

We show that if $m^{2/q}k^{2/\sigma} < n \le m^{2/q}k$, then $\Psi = \Psi_0$, $\Psi=\Psi_2$ or $\Psi = \Psi_4=\nu_\alpha^{1-\tilde \mu_{\alpha,\beta}} \nu_\beta^{\tilde \mu_{\alpha,\beta}} \Phi(\tilde p_{\alpha,\beta}, 2)$ with $\tilde p_{\alpha,\beta}<2$, and the estimating of the width can be reduced to one of the cases, which is already considered.

We set $$\lambda = \min \left\{ \mu \in [0, \, \tilde \lambda_{\alpha,\beta}]: \, \frac{1-\mu}{\theta_\alpha} + \frac{\mu}{\theta_\beta} \ge \frac{1}{\sigma}, \right. $$$$\left. \frac{(1-\mu)/p_\alpha +\mu/p_\beta-1/q}{1/2-1/q}\ge \frac{(1-\mu)/\theta_\alpha +\mu/\theta_\beta-1/\sigma}{1/2-1/\sigma}\right\},$$
$$
\tilde \lambda = \max \left\{ \mu \in [\tilde \lambda_{\alpha,\beta}, \, 1]:\; \frac{1-\mu}{\theta_\alpha} + \frac{\mu}{\theta_\beta} \in [1/\sigma, \, 1/2]\right\};
$$
the numbers $p_*$, $\theta_*$, $p_{**}$, $\theta_{**}$ are defined by (\ref{p_st1}), (\ref{p_st2}).
We get the inequalities (\ref{nu_a_phi_tl}), i.e.,
$$
\nu_\alpha^{1-\tilde \lambda_{\alpha,\beta}} \nu_\beta^{\tilde \lambda_{\alpha,\beta}} (n^{-1/2}m^{1/q}k^{1/\sigma})^{\frac{(1-\tilde \lambda_{\alpha,\beta})/\theta_\alpha+ \tilde \lambda_{\alpha,\beta}/\theta_\beta -1/\sigma}{1/2-1/\sigma}}\le
$$
$$
\le \nu_\alpha^{1-\lambda} \nu_\beta^{\lambda} (n^{-1/2}m^{1/q}k^{1/\sigma})^{\frac{(1- \lambda)/\theta_\alpha+ \lambda/\theta_\beta -1/\sigma}{1/2-1/\sigma}},
$$
$$
\nu_\alpha^{1-\tilde \lambda_{\alpha,\beta}} \nu_\beta^{\tilde \lambda_{\alpha,\beta}} (n^{-1/2}m^{1/q}k^{1/\sigma})^{\frac{(1-\tilde \lambda_{\alpha,\beta})/\theta_\alpha+ \tilde \lambda_{\alpha,\beta}/\theta_\beta -1/\sigma}{1/2-1/\sigma}}\le
$$
$$
\le \nu_\alpha^{1-\tilde\lambda} \nu_\beta^{\tilde\lambda} (n^{-1/2}m^{1/q}k^{1/\sigma})^{\frac{(1- \tilde\lambda)/\theta_\alpha+ \tilde\lambda/\theta_\beta -1/\sigma}{1/2-1/\sigma}}
$$
(see (\ref{phi5})). Hence, $\frac{\nu_\alpha}{\nu_\beta} = (n^{1/2}m^{-1/q} k^{-1/\sigma})^{\frac{1/\theta_\alpha -1/\theta_\beta}{1/2-1/\sigma}}$, and we get $\Psi = \nu_\alpha^{1-\tilde \lambda} \nu_\beta^{\tilde \lambda} \Phi(p_{**}, \, \theta_{**})$. If $\theta_{**}\in (2, \, \sigma)$, then $\Psi = \Psi_0$; if $\theta_{**}=\sigma$, then $\Psi=\Psi_2$. If $\theta_{**}=2$, then $\Psi=\nu_\alpha^{1-\tilde \mu_{\alpha,\beta}} \nu_\beta^{\tilde \mu_{\alpha,\beta}} \Phi(\tilde p_{\alpha,\beta}, 2)$, where $\tilde p_{\alpha,\beta}<2$.

Let now $n> m^{2/q}k$. We prove that (see \eqref{phi5})
\begin{align}
\label{dn_nm2qk} d_n(M, \, l_{q,\sigma}^{m,k}) \underset{q,\sigma}{\gtrsim} \nu_\alpha^{1-\tilde\lambda_{\alpha,\beta}} \nu_\beta^{\tilde\lambda_{\alpha,\beta}} k^{1/\sigma-1/\tilde \theta_{\alpha,\beta}} n^{-1/2}m^{1/q}k^{1/2}.
\end{align}
Let the number $r_{\alpha,\beta}$ be defined by the equation
\begin{align}
\label{r_a_b_2s_d} \frac{\nu_\alpha}{\nu_\beta} = r_{\alpha,\beta}^{1/p_\alpha-1/p_\beta} k^{1/\theta_\alpha-1/\theta_\beta}.
\end{align}
We show that
\begin{align}
\label{r_est_2s} 1\le r_{\alpha,\beta} \le m, \quad n \ge m^{\frac 2q} k^{\frac{2}{\sigma}} r_{\alpha,\beta}^{1-\frac 2q} k^{1-\frac{2}{\sigma}}.
\end{align}
The numbers $\lambda$, $\tilde \lambda$ are defined as for $m^{2/q}k^{2/\sigma}<n\le m^{2/q}k$. We get
$$
\nu_\alpha^{1-\tilde\lambda_{\alpha,\beta}} \nu_\beta^{\tilde\lambda_{\alpha,\beta}} k^{1/\sigma-1/\tilde \theta_{\alpha,\beta}} n^{-1/2}m^{1/q}k^{1/2} \le $$$$\le\nu_\alpha^{1-\lambda} \nu_\beta^{\lambda} k^{1/\sigma-(1-\lambda)/\theta_\alpha -\lambda/\theta_\beta} (n^{-1/2}m^{1/q}k^{1/2})^{\frac{(1-\lambda)/p_\alpha+\lambda/p_\beta - 1/q}{1/2-1/q}},
$$
$$
\nu_\alpha^{1-\tilde\lambda_{\alpha,\beta}} \nu_\beta^{\tilde\lambda_{\alpha,\beta}} k^{1/\sigma-1/\tilde \theta_{\alpha,\beta}} n^{-1/2}m^{1/q}k^{1/2} \le \nu_\alpha^{1-\tilde\lambda} \nu_\beta^{\tilde\lambda} k^{1/\sigma-(1-\tilde \lambda)/\theta_{\alpha} -\tilde \lambda/\theta_\beta} n^{-1/2}m^{1/q}k^{1/2}.
$$
Hence,
$$
(n^{1/2}m^{-1/q}k^{-1/2}) ^{\frac{1/p_\alpha-1/p_\beta}{1/2-1/q}}k^{1/\theta_\alpha-1/\theta_\beta} \le \frac{\nu_\alpha}{\nu_\beta} \le k^{1/\theta_\alpha-1/\theta_\beta}.
$$
This together with (\ref{r_a_b_2s_d}) yields that $1\le r_{\alpha,\beta} \le (n^{1/2}m^{-1/q}k^{-1/2}) ^{\frac{1}{1/2-1/q}}$; hence, (\ref{r_est_2s}) holds.

Let $r = \lfloor r_{\alpha,\beta} \rfloor$, $W = \nu_\alpha^{1-\tilde\lambda_{\alpha,\beta}} \nu_\beta^{\tilde\lambda_{\alpha,\beta}}r^{-1/2} k^{-1/\tilde \theta_{\alpha,\beta}} V^{m,k}_{r,k}$. We show that $W \subset 2M$. Then
$$
d_n(M, \, l_{q,\sigma}^{m,k}) \gtrsim d_n(W, \, l_{q,\sigma}^{m,k}) \stackrel{(\ref{dn_vmk1}), (\ref{r_est_2s})}{\underset{q,\sigma}{\gtrsim}} \nu_\alpha^{1-\tilde\lambda_{\alpha,\beta}} \nu_\beta^{\tilde\lambda_{\alpha,\beta}}r^{-1/2} k^{-1/\tilde \theta_{\alpha,\beta}} n^{-1/2}m^{1/q}k^{1/\sigma} r^{1/2}k^{1/2}=
$$
$$
=\nu_\alpha^{1-\tilde\lambda_{\alpha,\beta}} \nu_\beta^{\tilde\lambda_{\alpha,\beta}}k^{1/\sigma-1/\tilde \theta_{\alpha,\beta}} n^{-1/2}m^{1/q} k^{1/2};
$$
hence, (\ref{dn_nm2qk}) holds.

In order to prove the inclusion, it suffices to check that
\begin{align}
\label{nu_a_g_2s} \nu_\alpha^{1-\tilde\lambda_{\alpha,\beta}} \nu_\beta^{\tilde\lambda_{\alpha,\beta}}r_{\alpha,\beta}^{1/p_\gamma-1/2} k^{1/\theta_\gamma-1/\tilde \theta_{\alpha,\beta}}\le \nu_\gamma, \quad \gamma \in A.
\end{align}

If $\gamma=\alpha$ or $\gamma=\beta$, it follows from (\ref{r_a_b_2s_d}). 

Let $\gamma \notin \{\alpha,\beta\}$. We define the numbers $r_{\alpha,\gamma}$ and $r_{\gamma,\beta}$ by the equations
$$
\frac{\nu_\alpha}{\nu_\gamma} =r_{\alpha,\gamma}^{1/p_\alpha-1/p_\gamma}k^{1/\theta_\alpha-1/\theta_\gamma}, \quad \frac{\nu_\gamma}{\nu_\beta} = r_{\gamma,\beta}^{1/p_\gamma-1/p_\beta} k^{1/\theta_\gamma -1/\theta_\beta}.
$$
Then (\ref{nu_a_g_2s}) is equivalent to each of the following inequalities:
\begin{align}
\label{2s_r_abg} r_{\alpha,\beta}^{1/p_\gamma-1/p_\alpha} \le r_{\alpha,\gamma}^{1/p_\gamma-1/p_\alpha}, \quad r_{\alpha,\beta}^{1/p_\gamma-1/p_\beta} \le r_{\gamma,\beta}^{1/p_\gamma-1/p_\beta}.
\end{align}

Let $p_\gamma<2$. We prove that the first inequality of (\ref{2s_r_abg}) holds.

Let $\tilde \theta_{\alpha,\gamma}>2$. From the inequality $\nu_\alpha ^{1-\tilde \lambda_{\alpha,\beta}} \nu_\beta^{\tilde \lambda_{\alpha,\beta}} \Phi(2, \, \tilde \theta_{\alpha,\beta}) \le \nu_\alpha ^{1-\tilde \lambda_{\alpha,\gamma}} \nu_\beta^{\tilde \lambda_{\alpha,\gamma}} \Phi(2, \, \tilde \theta_{\alpha,\gamma})$ it follows that (see (\ref{phi3}), (\ref{phi5}))
\begin{align}
\label{nu_a_12m1q}
\nu_\alpha ^{1-\tilde \lambda_{\alpha,\beta}} \nu_\beta^{\tilde \lambda_{\alpha,\beta}} k^{1/\sigma-1/\tilde \theta_{\alpha,\beta}}n^{-1/2}m^{1/q}k^{1/2} \le \nu_\alpha ^{1-\tilde \lambda_{\alpha,\gamma}} \nu_\beta^{\tilde \lambda_{\alpha,\gamma}} k^{1/\sigma-1/\tilde \theta_{\alpha,\gamma}}n^{-1/2}m^{1/q}k^{1/2}.
\end{align}
Hence, $r_{\alpha,\beta}^{1/2-1/p_\alpha} \le r_{\alpha,\gamma}^{1/2-1/p_\alpha}$. Since $p_\alpha>2>p_\gamma$, we get the first inequality of (\ref{2s_r_abg}).

Let $\tilde \theta_{\alpha,\gamma}<2$. Then there is $\mu \in (0, \, 1)$ such that $\frac 12 = \frac{1-\mu}{\tilde \theta_{\alpha,\beta}} + \frac{\mu}{\tilde \theta_{\alpha,\gamma}}$. From the inequality
$\nu_\alpha ^{1-\tilde \lambda_{\alpha,\beta}} \nu_\beta^{\tilde \lambda_{\alpha,\beta}} \Phi(2, \, \tilde \theta_{\alpha,\beta}) \le (\nu_\alpha ^{1-\tilde \lambda_{\alpha,\beta}} \nu_\beta^{\tilde \lambda_{\alpha,\beta}})^{1-\mu} (\nu_\alpha ^{1-\tilde \lambda_{\alpha,\gamma}} \nu_\gamma^{\tilde \lambda_{\alpha,\gamma}})^\mu \Phi(2, \, 2)$ we get (\ref{nu_a_12m1q}), which again implies the first inequality of (\ref{2s_r_abg}).

If $p_\gamma>2$, then we similarly get the second inequality of (\ref{2s_r_abg}).

\subsection{The case $\Psi = \Psi_5$}

From (\ref{n5}), (\ref{psi5}) it follows that $q>2$, $\sigma>2$.

For $n\le m^{2/q}k^{2/\sigma}$, similarly as in the previous subsection, we prove that $\nu_\alpha = \nu_\beta$ and $\Psi=\Psi_0$, $\Psi = \Psi_1$ or $\Psi=\Psi_2$.

Let $n> m^{2/q}k^{2/\sigma}$. If $m=1$ or $k=1$, then, respectively, $\Phi(p, \, \theta) = (n^{-1/2}k^{1/\sigma})^{\frac{1/\theta-1/\sigma}{1/2-1/\sigma}}$ or $\Phi(p, \, \theta) = (n^{-1/2}m^{1/q})^{\frac{1/p-1/q}{1/2-1/q}}$ for $2\le p\le q$, $2\le \theta \le \sigma$. This implies that $\Psi= \Psi_j$ for some $j\in \{0, \, \dots, \, 4\}$; i.e., we get the case which was already considered (the arguments are as in the previous subsection).

Let now $m\ge 2$, $k\ge 2$. We have 
\begin{align}
\label{psi_5_psi_eq}
\begin{array}{c}
\Psi = \nu_{\alpha}^{1-\lambda_{\alpha,\beta}} \nu_\beta^{\lambda_{\alpha,\beta}}(n^{-1/2}m^{1/q}k^{1/\sigma})^{\frac{1/p_{\alpha,\beta}-1/q}{1/2-1/q}} = \nu_{\alpha}^{1-\lambda_{\alpha,\beta}} \nu_\beta^{\lambda_{\alpha,\beta}}(n^{-1/2}m^{1/q}k^{1/\sigma})^{\frac{1/\theta_{\alpha,\beta}-1/\sigma}{1/2-1/\sigma}}=\\
=\nu_{\alpha}^{1-\lambda_{\alpha,\beta}} \nu_\beta^{\lambda_{\alpha,\beta}}k^{1/\sigma-1/\theta_{\alpha,\beta}}(n^{-1/2}m^{1/q}k^{1/2})^{\frac{1/p_{\alpha,\beta}-1/q}{1/2-1/q}} = \\=\nu_{\alpha}^{1-\lambda_{\alpha,\beta}} \nu_\beta^{\lambda_{\alpha,\beta}}m^{1/q-1/p_{\alpha,\beta}}(n^{-1/2}m^{1/2}k^{1/\sigma})^{\frac{1/\theta_{\alpha,\beta}-1/\sigma}{1/2-1/\sigma}}.
\end{array}
\end{align}
Without lost of generality,
\begin{align}
\label{paq_tas} \frac{1/p_\alpha-1/q}{1/2-1/q} > \frac{1/\theta_\alpha-1/\sigma}{1/2-1/\sigma}, \quad \frac{1/p_\beta-1/q}{1/2-1/q} < \frac{1/\theta_\beta-1/\sigma}{1/2-1/\sigma}.
\end{align}
This implies that
\begin{align}
\label{paq_tas1} \frac{1/p_\alpha-1/p_\beta}{1/2-1/q} > \frac{1/\theta_\alpha-1/\theta_\beta}{1/2-1/\sigma}.
\end{align}

We define the numbers $r_{\alpha,\beta}$, $l_{\alpha,\beta}$ by the equation 
\begin{align}
\label{nua_nub_rl} \frac{\nu_\alpha}{\nu_\beta} = r_{\alpha,\beta}^{1/p_\alpha-1/p_\beta} l_{\alpha,\beta} ^{1/\theta_\alpha-1/\theta_\beta};
\end{align}
here 
\begin{align}
\label{rl1} r_{\alpha,\beta} = (n^{1/2} m^{-1/q} k^{-1/\sigma}) ^{\frac{1-\tau}{1/2-1/q}}, \quad l_{\alpha,\beta} = (n^{1/2}m^{-1/q} k^{-1/\sigma}) ^{\frac{\tau}{1/2-1/\sigma}}
\end{align}
for $m^{2/q}k^{2/\sigma} < n\le \min\{m^{2/q}k, \, mk^{2/\sigma}\}$;
\begin{align}
\label{rl2} r_{\alpha,\beta} = m^\tau, \quad l_{\alpha,\beta} = (n^{1/2}m^{-1/q} k^{-1/\sigma}) ^{\frac{1-\tau}{1/2-1/\sigma}}(n^{1/2}m^{-1/2} k^{-1/\sigma}) ^{\frac{\tau}{1/2-1/\sigma}}
\end{align}
for $mk^{2/\sigma}< n\le m^{2/q}k$;
\begin{align}
\label{rl3} r_{\alpha,\beta} = (n^{1/2}m^{-1/q} k^{-1/\sigma}) ^{\frac{1-\tau}{1/2-1/q}}(n^{1/2}m^{-1/q} k^{-1/2}) ^{\frac{\tau}{1/2-1/q}}, \quad l_{\alpha,\beta} = k^{\tau},
\end{align}
for $m^{2/q}k< n\le mk^{2/\sigma}$;
\begin{align}
\label{rl4} r_{\alpha,\beta} = m^\tau(n^{1/2}m^{-1/q} k^{-1/2}) ^{\frac{1-\tau}{1/2-1/q}}, \quad l_{\alpha,\beta} = k^{1-\tau}(n^{1/2}m^{-1/2}k^{-1/\sigma}) ^{\frac{\tau}{1/2-1/\sigma}}
\end{align}
for $\max\{mk^{2/\sigma}, \, m^{2/q}k\}< n \le mk/2$.
From (\ref{paq_tas1}) and the conditions $m^{2/q}k^{2/\sigma}< n< mk$, $m\ge 2$, $k\ge 2$ it follows that such number $\tau$ is well-defined. We show that $\tau\in [0, \, 1]$. Then
\begin{align}
\label{r1m_l1k_ab} 1\le r_{\alpha,\beta}\le m, \quad 1\le l_{\alpha,\beta} \le k, \quad n\le m^{\frac 2q}k^{\frac{2}{\sigma}} r_{\alpha,\beta}^{1-\frac 2q} l_{\alpha,\beta} ^{1-\frac{2}{\sigma}}.
\end{align}

We set
$$
\lambda = \min \left\{ \mu \in [0, \, \lambda_{\alpha,\beta}]:\; \frac{1-\mu}{p_\alpha}+\frac{\mu}{p_\beta}\in [1/q, \, 1/2], \; \frac{1-\mu}{\theta_\alpha}+\frac{\mu}{\theta_\beta}\in [1/\sigma, \, 1/2]\right\},
$$
$$
\tilde\lambda = \max \left\{ \mu \in [\lambda_{\alpha,\beta}, \, 1]:\; \frac{1-\mu}{p_\alpha}+\frac{\mu}{p_\beta}\in [1/q, \, 1/2], \; \frac{1-\mu}{\theta_\alpha}+\frac{\mu}{\theta_\beta}\in [1/\sigma, \, 1/2]\right\}.
$$
Then $\lambda<\lambda_{\alpha,\beta}<\tilde \lambda$. The numbers $p_*$, $\theta_*$, $p_{**}$, $\theta_{**}$ are defined by (\ref{p_st1}), (\ref{p_st2}).
We have $$\nu_\alpha^{1-\lambda_{\alpha,\beta}} \nu_\beta^{\lambda_{\alpha,\beta}} \Phi(p_{\alpha,\beta}, \, \theta_{\alpha,\beta}) \le \nu_{\alpha}^{1-\lambda} \nu_\beta^\lambda \Phi(p_*, \, \theta_*), $$$$\nu_\alpha^{1-\lambda_{\alpha,\beta}} \nu_\beta^{\lambda_{\alpha,\beta}} \Phi(p_{\alpha,\beta}, \, \theta_{\alpha,\beta}) \le \nu_{\alpha}^{1-\tilde\lambda} \nu_\beta^{\tilde\lambda} \Phi(p_{**}, \, \theta_{**}).$$
Now we write these inequalities applying (\ref{phi4}), (\ref{phi5}), (\ref{psi_5_psi_eq}) and taking into account that $\omega_{p_*,q}\ge \omega_{\theta_*,\sigma}$, $\omega_{p_{**},q}\le \omega_{\theta_{**},\sigma}$ by (\ref{paq_tas}).

If $m^{2/q}k^{2/\sigma} < n\le \min \{mk^{2/\sigma}, \, m^{2/q}k\}$, then
$$
\nu_\alpha^{1-\lambda_{\alpha,\beta}} \nu_\beta^{\lambda_{\alpha,\beta}} (n^{-1/2}m^{1/q}k^{1/\sigma})^{\frac{(1-\lambda_{\alpha,\beta})/\theta_\alpha + \lambda_{\alpha,\beta}/\theta_\beta -1/\sigma}{1/2-1/\sigma}} \le 
$$
$$
\le\nu_\alpha^{1-\lambda} \nu_\beta^{\lambda} (n^{-1/2}m^{1/q}k^{1/\sigma})^{\frac{(1-\lambda)/\theta_\alpha + \lambda/\theta_\beta -1/\sigma}{1/2-1/\sigma}},
$$
$$
\nu_\alpha^{1-\lambda_{\alpha,\beta}} \nu_\beta^{\lambda_{\alpha,\beta}} (n^{-1/2}m^{1/q}k^{1/\sigma})^{\frac{(1-\lambda_{\alpha,\beta})/p_\alpha + \lambda_{\alpha,\beta}/p_\beta -1/q}{1/2-1/q}} \le 
$$
$$
\le\nu_\alpha^{1-\tilde\lambda} \nu_\beta^{\tilde\lambda} (n^{-1/2}m^{1/q}k^{1/\sigma})^{\frac{(1-\tilde\lambda)/p_\alpha + \tilde\lambda/p_\beta -1/q}{1/2-1/q}}.
$$
Hence,
$$
(n^{1/2}m^{-1/q}k^{-1/\sigma})^{\frac{1/\theta_\alpha-1/\theta_\beta}{1/2-1/\sigma}}\le \frac{\nu_\alpha}{\nu_\beta} \le (n^{1/2}m^{-1/q}k^{-1/\sigma})^{\frac{1/p_\alpha-1/p_\beta}{1/2-1/q}}.
$$
This together with (\ref{nua_nub_rl}), (\ref{rl1}) yields that $0\le \tau \le 1$.

If $mk^{2/\sigma}< n\le m^{2/q}k$, then
$$
\nu_\alpha^{1-\lambda_{\alpha,\beta}} \nu_\beta^{\lambda_{\alpha,\beta}} (n^{-1/2}m^{1/q}k^{1/\sigma})^{\frac{(1-\lambda_{\alpha,\beta})/\theta_\alpha + \lambda_{\alpha,\beta}/\theta_\beta -1/\sigma}{1/2-1/\sigma}} \le 
$$
$$
\le \nu_\alpha^{1-\lambda} \nu_\beta^{\lambda} (n^{-1/2}m^{1/q}k^{1/\sigma})^{\frac{(1-\lambda)/\theta_\alpha + \lambda/\theta_\beta -1/\sigma}{1/2-1/\sigma}},
$$
$$
\nu_\alpha^{1-\lambda_{\alpha,\beta}} \nu_\beta^{\lambda_{\alpha,\beta}} m^{1/q-(1-\lambda_{\alpha,\beta})/p_\alpha -\lambda_{\alpha,\beta}/p_\beta}(n^{-1/2}m^{1/2}k^{1/\sigma})^{\frac{(1-\lambda_{\alpha,\beta})/\theta_\alpha + \lambda_{\alpha,\beta}/\theta_\beta -1/\sigma}{1/2-1/\sigma}} \le 
$$
$$
\le\nu_\alpha^{1-\tilde\lambda} \nu_\beta^{\tilde\lambda} m^{1/q-(1-\tilde\lambda)/p_\alpha -\tilde\lambda/p_\beta}(n^{-1/2}m^{1/2}k^{1/\sigma})^{\frac{(1-\tilde\lambda)/\theta_\alpha + \tilde\lambda/\theta_\beta -1/\sigma}{1/2-1/\sigma}}.
$$
Hence,
$$
(n^{1/2}m^{-1/q}k^{-1/\sigma})^{\frac{1/\theta_\alpha-1/\theta_\beta}{1/2-1/\sigma}}\le \frac{\nu_\alpha}{\nu_\beta} \le m^{1/p_\alpha-1/p_\beta}(n^{1/2}m^{-1/2}k^{-1/\sigma})^{\frac{1/\theta_\alpha-1/\theta_\beta}{1/2-1/\sigma}}.
$$
This together with (\ref{nua_nub_rl}), (\ref{rl2}) implies that $0\le \tau \le 1$.

The case $m^{2/q}k< n\le mk^{2/\sigma}$ is considered similarly, applying (\ref{rl3}).

For $n> \max\{mk^{2/\sigma}, \, m^{2/q}k\}$ we get
$$
\nu_\alpha^{1-\lambda_{\alpha,\beta}} \nu_\beta^{\lambda_{\alpha,\beta}} k^{1/\theta-(1-\lambda_{\alpha,\beta})/\theta_\alpha-\lambda_{\alpha,\beta}/\theta_\beta}(n^{-1/2}m^{1/q}k^{1/2})^{\frac{(1-\lambda_{\alpha,\beta})/p_\alpha + \lambda_{\alpha,\beta}/p_\beta -1/q}{1/2-1/q}} \le 
$$
$$
\le \nu_\alpha^{1-\lambda} \nu_\beta^{\lambda} k^{1/\theta -(1-\lambda)/\theta_\alpha -\lambda/\theta_\beta}(n^{-1/2}m^{1/q}k^{1/2})^{\frac{(1-\lambda)/p_\alpha + \lambda/p_\beta -1/q}{1/2-1/q}},
$$
$$
\nu_\alpha^{1-\lambda_{\alpha,\beta}} \nu_\beta^{\lambda_{\alpha,\beta}} m^{1/q-(1-\lambda_{\alpha,\beta})/p_\alpha -\lambda_{\alpha,\beta}/p_\beta}(n^{-1/2}m^{1/2}k^{1/\sigma})^{\frac{(1-\lambda_{\alpha,\beta})/\theta_\alpha + \lambda_{\alpha,\beta}/\theta_\beta -1/\sigma}{1/2-1/\sigma}} \le 
$$
$$
\le\nu_\alpha^{1-\tilde\lambda} \nu_\beta^{\tilde\lambda} m^{1/q-(1-\tilde\lambda)/p_\alpha -\tilde\lambda/p_\beta}(n^{-1/2}m^{1/2}k^{1/\sigma})^{\frac{(1-\tilde\lambda)/\theta_\alpha + \tilde\lambda/\theta_\beta -1/\sigma}{1/2-1/\sigma}}.
$$
Hence,
$$
k^{1/\theta_\alpha-1/\theta_\beta}(n^{1/2}m^{-1/q}k^{-1/2})^{\frac{1/p_\alpha-1/p_\beta}{1/2-1/q}}\le \frac{\nu_\alpha}{\nu_\beta} \le m^{1/p_\alpha-1/p_\beta}(n^{1/2}m^{-1/2}k^{-1/\sigma})^{\frac{1/\theta_\alpha-1/\theta_\beta}{1/2-1/\sigma}}.
$$
This together with (\ref{nua_nub_rl}), (\ref{rl4}) yields that $0\le \tau \le 1$.

Let $r=\lceil r_{\alpha,\beta}\rceil$, $l = \lceil l_{\alpha,\beta}\rceil$. By (\ref{r1m_l1k_ab}), we have $1\le r\le m$, $1\le l \le k$, $n \le m^{\frac 2q} k^{\frac{2}{\sigma}} r^{1-\frac 2q} l^{1-\frac{2}{\sigma}}$. Let $W= \nu_\alpha^{1-\lambda_{\alpha,\beta}}\nu_\beta^{\lambda_{\alpha,\beta}} r^{-1/p_{\alpha,\beta}}l^{-1/\theta_{\alpha,\beta}} V^{m,k}_{r,l}$. We show that $W \subset 4M$. Then
$$
d_n(M, \, l^{m,k}_{q,\sigma}) \gtrsim d_n(W, \, l^{m,k}_{q,\sigma}) \stackrel{(\ref{dn_vmk1})}{\underset{q,\sigma}{\gtrsim}} \nu_\alpha^{1-\lambda_{\alpha,\beta}}\nu_\beta^{\lambda_{\alpha,\beta}} r^{-1/p_{\alpha,\beta}}l^{-1/\theta_{\alpha,\beta}}r^{1/q} l^{1/\sigma};
$$
this together with (\ref{psi_5_psi_eq}), (\ref{rl1})--(\ref{rl4}) yields the desired estimate for the widths.

In order to prove the inclusion, it suffices to check that
\begin{align}
\label{nu_a_1lab_r_pg}
\nu_\alpha^{1-\lambda_{\alpha,\beta}}\nu_\beta^{\lambda_{\alpha,\beta}} r_{\alpha,\beta}^{1/p_\gamma-1/p_{\alpha,\beta}}l_{\alpha,\beta}^{1/\theta_\gamma-1/\theta_{\alpha,\beta}} \le \nu_\gamma, \, \quad \gamma \in A.
\end{align}

If $\gamma=\alpha$ or $\gamma = \beta$, it follows from (\ref{nua_nub_rl}).

Let $\gamma \notin \{\alpha,\beta\}$. We suppose that 
\begin{align}
\label{pgq2q_tg}
\frac{1/p_\gamma-1/q}{1/2-1/q} < \frac{1/\theta_\gamma-1/\sigma}{1/2-1/\sigma}
\end{align}
(the case $\frac{1/p_\gamma-1/q}{1/2-1/q} > \frac{1/\theta_\gamma-1/\sigma}{1/2-1/\sigma}$ is similar).
Then from (\ref{paq_tas}) it follows that
\begin{align}
\label{1pgpa2q}
\frac{1/p_\gamma-1/p_\alpha}{1/2-1/q} < \frac{1/\theta_\gamma-1/\theta_\alpha}{1/2-1/\sigma}.
\end{align}
We define the numbers $r_{\alpha,\gamma}$ and $l_{\alpha,\gamma}$ by the equation
\begin{align}
\label{nua_nub_rl1} \frac{\nu_\alpha}{\nu_\gamma} = r_{\alpha,\gamma}^{1/p_\alpha-1/p_\gamma} l_{\alpha,\gamma} ^{1/\theta_\alpha-1/\theta_\gamma};
\end{align}
here
\begin{align}
\label{rl11} r_{\alpha,\gamma} = (n^{1/2} m^{-1/q} k^{-1/\sigma}) ^{\frac{1-\tau'}{1/2-1/q}}, \quad l_{\alpha,\gamma} = (n^{1/2}m^{-1/q} k^{-1/\sigma}) ^{\frac{\tau'}{1/2-1/\sigma}}
\end{align}
for $m^{2/q}k^{2/\sigma} < n\le \min\{m^{2/q}k, \, mk^{2/\sigma}\}$;
\begin{align}
\label{rl21} r_{\alpha,\gamma} = m^{\tau'}, \quad l_{\alpha,\gamma} = (n^{1/2}m^{-1/q} k^{-1/\sigma}) ^{\frac{1-\tau'}{1/2-1/\sigma}}(n^{1/2}m^{-1/2} k^{-1/\sigma}) ^{\frac{\tau'}{1/2-1/\sigma}}
\end{align}
for $mk^{2/\sigma}< n\le m^{2/q}k$;
\begin{align}
\label{rl31} r_{\alpha,\gamma} = (n^{1/2}m^{-1/q} k^{-1/\sigma}) ^{\frac{1-\tau'}{1/2-1/q}}(n^{1/2}m^{-1/q} k^{-1/2}) ^{\frac{\tau'}{1/2-1/q}}, \quad l_{\alpha,\gamma} = k^{\tau'}
\end{align}
for $m^{2/q}k< n\le mk^{2/\sigma}$;
\begin{align}
\label{rl41} r_{\alpha,\gamma} = m^{\tau'}(n^{1/2}m^{-1/q} k^{-1/2}) ^{\frac{1-\tau'}{1/2-1/q}}, \quad l_{\alpha,\gamma} = k^{1-\tau'}(n^{1/2}m^{-1/2}k^{-1/\sigma}) ^{\frac{\tau'}{1/2-1/\sigma}}
\end{align}
for $\max\{mk^{2/\sigma}, \, m^{2/q}k\}< n \le \frac{mk}{2}$.
Such number $\tau'$ is well-defined, by (\ref{1pgpa2q}) and the conditions $m\ge 2$, $k\ge 2$, $m^{2/q}k^{2/\sigma}< n < mk$.

From (\ref{nua_nub_rl}), (\ref{nua_nub_rl1}) it follows that (\ref{nu_a_1lab_r_pg}) is equivalent to 
\begin{align}
\label{r_ab_1pgpa} r_{\alpha,\beta}^{1/p_\gamma-1/p_\alpha} l_{\alpha,\beta} ^{1/\theta_\gamma-1/\theta_\alpha} \le r_{\alpha,\gamma}^{1/p_\gamma-1/p_\alpha} l_{\alpha,\gamma}^{1/\theta_\gamma-1/\theta_\alpha}.
\end{align}

By (\ref{paq_tas}) and (\ref{pgq2q_tg}), there is a number $\lambda_{\alpha,\gamma}\in (0, \, 1)$ such that $\frac{1/p_{\alpha,\gamma}-1/q}{1/2-1/q} = \frac{1/\theta_{\alpha,\gamma}-1/\sigma}{1/2-1/\sigma}$, where $\frac{1}{p_{\alpha,\gamma}} = \frac{1-\lambda_{\alpha,\gamma}}{p_\alpha} + \frac{\lambda_{\alpha,\gamma}}{p_\gamma}$, $\frac{1}{\theta_{\alpha,\gamma}} = \frac{1-\lambda_{\alpha,\gamma}}{\theta_\alpha} + \frac{\lambda_{\alpha,\gamma}}{\theta_\gamma}$.

If $2\le p_{\alpha,\gamma}\le q$, then
\begin{align}
\label{phi_compar}
\nu_{\alpha}^{1-\lambda_{\alpha,\beta}} \nu_\beta^{\lambda_{\alpha,\beta}} \Phi(p_{\alpha,\beta}, \, \theta_{\alpha,\beta})= \Psi = \Psi_5 \le \nu_{\alpha}^{1-\lambda_{\alpha,\gamma}} \nu_\gamma^{\lambda_{\alpha,\gamma}} \Phi(p_{\alpha,\gamma}, \, \theta_{\alpha,\gamma}).
\end{align}
Let $p_{\alpha,\gamma}<2$. Then there exists a number $\mu\in (0, \, 1)$ such that $\frac{1}{2} = \frac{1-\mu}{p_{\alpha,\beta}} + \frac{\mu}{p_{\alpha,\gamma}}$, $\frac{1}{2} = \frac{1-\mu}{\theta_{\alpha,\beta}} + \frac{\mu}{\theta_{\alpha,\gamma}}$. In this case,
\begin{align}
\label{phi_compar1}
\nu_{\alpha}^{1-\lambda_{\alpha,\beta}} \nu_\beta^{\lambda_{\alpha,\beta}} \Phi(p_{\alpha,\beta}, \, \theta_{\alpha,\beta})=\Psi\le \Psi_7 \le (\nu_{\alpha}^{1-\lambda_{\alpha,\beta}} \nu_\beta^{\lambda_{\alpha,\beta}})^{1-\mu}(\nu_{\alpha}^{1-\lambda_{\alpha,\gamma}} \nu_\gamma^{\lambda_{\alpha,\gamma}})^{\mu} \Phi(2, \, 2).
\end{align}
Let $p_{\alpha,\gamma}>q$. Then there is a number $\mu\in (0, \, 1)$ such that $\frac{1}{q} = \frac{1-\mu}{p_{\alpha,\beta}} + \frac{\mu}{p_{\alpha,\gamma}}$, $\frac{1}{\sigma} = \frac{1-\mu}{\theta_{\alpha,\beta}} + \frac{\mu}{\theta_{\alpha,\gamma}}$. In this case,
\begin{align}
\label{phi_compar2}
\nu_{\alpha}^{1-\lambda_{\alpha,\beta}} \nu_\beta^{\lambda_{\alpha,\beta}} \Phi(p_{\alpha,\beta}, \, \theta_{\alpha,\beta})=\Psi\le \Psi_6 \le (\nu_{\alpha}^{1-\lambda_{\alpha,\beta}} \nu_\beta^{\lambda_{\alpha,\beta}})^{1-\mu}(\nu_{\alpha}^{1-\lambda_{\alpha,\gamma}} \nu_\gamma^{\lambda_{\alpha,\gamma}})^{\mu} \Phi(q, \, \sigma).
\end{align}

Let $m^{2/q}k^{2/\sigma}< n\le \min \{mk^{2/\sigma}, \, m^{2/q}k\}$. Then (\ref{r_ab_1pgpa}) is equivalent to
$$
(n^{1/2}m^{-1/q}k^{-1/\sigma})^{\frac{(1-\tau)(1/p_\gamma-1/p_\alpha)}{1/2-1/q}}(n^{1/2}m^{-1/q}k^{-1/\sigma})^{\frac{\tau(1/\theta_\gamma-1/\theta_\alpha)}{1/2-1/\sigma}} \le$$$$\le (n^{1/2}m^{-1/q}k^{-1/\sigma})^{\frac{(1-\tau')(1/p_\gamma-1/p_\alpha)}{1/2-1/q}}(n^{1/2}m^{-1/q}k^{-1/\sigma})^{\frac{\tau'(1/\theta_\gamma-1/\theta_\alpha)}{1/2-1/\sigma}},
$$
or
$$
(n^{1/2}m^{-1/q}k^{-1/\sigma})^{\tau\left(\frac{1/p_\alpha-1/p_\gamma}{1/2-1/q} - \frac{1/\theta_\alpha-1/\theta_\gamma}{1/2-1/\sigma}\right)} \le (n^{1/2}m^{-1/q}k^{-1/\sigma})^{\tau'\left(\frac{1/p_\alpha-1/p_\gamma}{1/2-1/q} - \frac{1/\theta_\alpha-1/\theta_\gamma}{1/2-1/\sigma}\right)}.
$$
Since $n> m^{2/q}k^{2/\sigma}$ and (\ref{1pgpa2q}) holds, it is equivalent to
\begin{align}
\label{ttpr} \tau \le \tau'.
\end{align}

In order to prove (\ref{ttpr}), we use (\ref{phi_compar}), (\ref{phi_compar1}) or (\ref{phi_compar2}). The corresponding inequality is as follows:
$$
\nu_\alpha^{1-\lambda_{\alpha,\beta}} \nu_\beta^{\lambda_{\alpha,\beta}} (n^{1/2}m^{-1/q} k^{-1/\sigma})^{\frac{1/q-1/p_{\alpha,\beta}}{1/2-1/q}} \le \nu_\alpha^{1-\lambda_{\alpha,\gamma}} \nu_\gamma^{\lambda_{\alpha,\gamma}} (n^{1/2}m^{-1/q} k^{-1/\sigma})^{\frac{1/q-1/p_{\alpha,\gamma}}{1/2-1/q}}.
$$
By (\ref{nua_nub_rl}), (\ref{rl1}), (\ref{nua_nub_rl1}),  (\ref{rl11}), it is equivalent to
$$
(n^{1/2}m^{-1/q}k^{-1/\sigma})^{-\lambda_{\alpha,\beta}(1-\tau)\frac{1/p_\alpha-1/p_\beta}{1/2-1/q}} (n^{1/2}m^{-1/q}k^{-1/\sigma})^{-\lambda_{\alpha,\beta}\tau\frac{1/\theta_\alpha-1/\theta_\beta}{1/2-1/\sigma}} \times$$$$\times(n^{1/2}m^{-1/q}k^{-1/\sigma})^{\frac{1/q-(1-\lambda_{\alpha,\beta})/p_\alpha -\lambda_{\alpha,\beta}/p_\beta}{1/2-1/q}} \le
$$
$$
\le (n^{1/2}m^{-1/q}k^{-1/\sigma})^{-\lambda_{\alpha,\gamma}(1-\tau')\frac{1/p_\alpha-1/p_\gamma}{1/2-1/q}} (n^{1/2}m^{-1/q}k^{-1/\sigma})^{-\lambda_{\alpha,\gamma}\tau'\frac{1/\theta_\alpha-1/\theta_\gamma}{1/2-1/\sigma}} \times$$$$\times(n^{1/2}m^{-1/q}k^{-1/\sigma})^{\frac{1/q-(1-\lambda_{\alpha,\gamma})/p_\alpha -\lambda_{\alpha,\gamma}/p_\gamma}{1/2-1/q}},
$$
or
$$
(n^{1/2}m^{-1/q}k^{-1/\sigma})^{\tau \lambda_{\alpha,\beta} \left(\frac{1/p_\alpha-1/p_\beta}{1/2-1/q} -\frac{1/\theta_\alpha-1/\theta_\beta}{1/2-1/\sigma}\right)} \le (n^{1/2}m^{-1/q}k^{-1/\sigma})^{\tau' \lambda_{\alpha,\gamma} \left(\frac{1/p_\alpha-1/p_\gamma}{1/2-1/q} -\frac{1/\theta_\alpha-1/\theta_\gamma}{1/2-1/\sigma}\right)}.
$$
Since
\begin{align}
\label{lab_q_lag}
\begin{array}{c}
\lambda_{\alpha,\beta} \left(\frac{1/p_\alpha-1/p_\beta}{1/2-1/q} -\frac{1/\theta_\alpha-1/\theta_\beta}{1/2-1/\sigma}\right) = \lambda_{\alpha,\gamma} \left(\frac{1/p_\alpha-1/p_\gamma}{1/2-1/q} -\frac{1/\theta_\alpha-1/\theta_\gamma}{1/2-1/\sigma}\right) =\\=\frac{1/p_\alpha-1/q}{1/2-1/q} -\frac{1/\theta_\alpha-1/\sigma}{1/2-1/\sigma}\stackrel{(\ref{paq_tas})}{>}0
\end{array}
\end{align}
and $n> m^{2/q}k^{2/\sigma}$, we get (\ref{ttpr}).

Let $mk^{2/\sigma}< n\le m^{2/q}k$. Then (\ref{r_ab_1pgpa}) is equivalent to (see (\ref{rl2}), (\ref{rl21}))
$$
m^{\tau(1/p_\gamma-1/p_\alpha)}(n^{1/2}m^{-1/q}k^{-1/\sigma})^{\frac{(1-\tau)(1/\theta_\gamma-1/\theta_\alpha)}{1/2-1/\sigma}}(n^{1/2}m^{-1/2}k^{-1/\sigma})^{\frac{\tau(1/\theta_\gamma-1/\theta_\alpha)}{1/2-1/\sigma}} \le$$$$\le m^{\tau'(1/p_\gamma-1/p_\alpha)}(n^{1/2}m^{-1/q}k^{-1/\sigma})^{\frac{(1-\tau')(1/\theta_\gamma-1/\theta_\alpha)}{1/2-1/\sigma}}(n^{1/2}m^{-1/2}k^{-1/\sigma})^{\frac{\tau'(1/\theta_\gamma-1/\theta_\alpha)}{1/2-1/\sigma}} ,
$$
or
$$
m^{\tau\left(1/p_\gamma-1/p_\alpha-\frac{(1/\theta_\gamma-1/\theta_\alpha)(1/2-1/q)}{1/2-1/\sigma}\right)} \le m^{\tau'\left(1/p_\gamma-1/p_\alpha-\frac{(1/\theta_\gamma-1/\theta_\alpha)(1/2-1/q)}{1/2-1/\sigma}\right)}.
$$
By (\ref{1pgpa2q}) and the condition $m\ge 2$, it is equivalent to
\begin{align}
\label{tt2} \tau \ge \tau'.
\end{align}
In order to prove (\ref{tt2}), we apply (\ref{phi_compar}), (\ref{phi_compar1}) or (\ref{phi_compar2}). The corresponding inequality is as follows:
$$
\nu_\alpha^{1-\lambda_{\alpha,\beta}} \nu_\beta^{\lambda_{\alpha,\beta}} (n^{1/2}m^{-1/q} k^{-1/\sigma})^{\frac{1/\sigma-1/\theta_{\alpha,\beta}}{1/2-1/\sigma}} \le \nu_\alpha^{1-\lambda_{\alpha,\gamma}} \nu_\gamma^{\lambda_{\alpha,\gamma}} (n^{1/2}m^{-1/q} k^{-1/\sigma})^{\frac{1/\sigma-1/\theta_{\alpha,\gamma}}{1/2-1/\theta}}.
$$
By (\ref{nua_nub_rl}), (\ref{rl2}), (\ref{nua_nub_rl1}),  (\ref{rl21}), it is equivalent to
$$
m^{-\tau(1/p_\alpha-1/p_\beta)\lambda_{\alpha,\beta}}(n^{1/2}m^{-1/q}k^{-1/\sigma})^{-(1-\tau)\frac{(1/\theta_\alpha-1/\theta_\beta)\lambda_{\alpha,\beta}}{1/2-1/\sigma}}\times$$$$\times(n^{1/2}m^{-1/2}k^{-1/\sigma})^{-\tau \frac{(1/\theta_\alpha-1/\theta_\beta)\lambda_{\alpha,\beta}}{1/2-1/\sigma}}(n^{1/2}m^{-1/q} k^{-1/\sigma})^{\frac{1/\sigma-(1-\lambda_{\alpha,\beta})/\theta_{\alpha}-\lambda_{\alpha,\beta} /\theta_\beta}{1/2-1/\sigma}}\le
$$
$$
\le m^{-\tau'(1/p_\alpha-1/p_\gamma)\lambda_{\alpha,\gamma}}(n^{1/2}m^{-1/q}k^{-1/\sigma})^{-(1-\tau')\frac{(1/\theta_\alpha-1/\theta_\gamma)\lambda_{\alpha,\gamma}}{1/2-1/\sigma}}\times$$$$\times(n^{1/2}m^{-1/2}k^{-1/\sigma})^{-\tau' \frac{(1/\theta_\alpha-1/\theta_\gamma)\lambda_{\alpha,\gamma}}{1/2-1/\sigma}}(n^{1/2}m^{-1/q} k^{-1/\sigma})^{\frac{1/\sigma-(1-\lambda_{\alpha,\gamma})/\theta_{\alpha}-\lambda_{\alpha,\gamma} /\theta_\gamma}{1/2-1/\sigma}},
$$
or
$$
m^{-\tau\left(1/p_\alpha-1/p_\beta -\frac{(1/\theta_\alpha-1/\theta_\beta)(1/2-1/q)}{1/2-1/\sigma}\right)\lambda _{\alpha,\beta}} \le m^{-\tau'\left(1/p_\alpha-1/p_\gamma -\frac{(1/\theta_\alpha-1/\theta_\gamma)(1/2-1/q)}{1/2-1/\sigma}\right)\lambda _{\alpha,\gamma}}.
$$
Taking into account (\ref{lab_q_lag}) and the condition $m\ge 2$, we get (\ref{tt2}).

The case $m^{2/q}k < n\le mk^{2/\sigma}$ is similar; here we use (\ref{rl3}) and (\ref{rl31}).

Let now $\max\{mk^{2/\sigma}, \, m^{2/q}k\} < n \le \frac{mk}{2}$. Then
(\ref{r_ab_1pgpa}) is equivalent to (see (\ref{rl4}), (\ref{rl41}))
$$
m^{\tau(1/p_\gamma-1/p_\alpha)}(n^{1/2}m^{-1/q}k^{-1/2})^{(1-\tau) \frac{1/p_\gamma-1/p_\alpha}{1/2-1/q}}\times $$$$\times k^{(1-\tau)(1/\theta_\gamma -1/\theta_\alpha)}(n^{1/2}m^{-1/2}k^{-1/\sigma})^{\frac{\tau(1/\theta_\gamma-1/\theta_\alpha)}{1/2-1/\sigma}} \le$$$$\le m^{\tau'(1/p_\gamma-1/p_\alpha)}(n^{1/2}m^{-1/q}k^{-1/2})^{(1-\tau') \frac{1/p_\gamma-1/p_\alpha}{1/2-1/q}} \times $$$$ \times k^{(1-\tau')(1/\theta_\gamma -1/\theta_\alpha)}(n^{1/2}m^{-1/2}k^{-1/\sigma})^{\frac{\tau'(1/\theta_\gamma-1/\theta_\alpha)}{1/2-1/\sigma}},
$$
or
$$
(n^{-1/2}m^{1/2}k^{1/2}) ^{\tau\left(\frac{1/p_\gamma-1/p_\alpha}{1/2-1/q} - \frac{1/\theta_\gamma-1/\theta_\alpha}{1/2-1/\sigma}\right)}\le (n^{-1/2}m^{1/2}k^{1/2}) ^{\tau'\left(\frac{1/p_\gamma-1/p_\alpha}{1/2-1/q} - \frac{1/\theta_\gamma-1/\theta_\alpha}{1/2-1/\sigma}\right)}.
$$
By (\ref{1pgpa2q}) and the condition $n<mk$, it is equivalent to
\begin{align}
\label{tt3} \tau \ge \tau'.
\end{align}
In order to prove (\ref{tt3}), we use (\ref{phi_compar}), (\ref{phi_compar1}) or (\ref{phi_compar2}). The corresponding inequality is as follows:
$$
\nu_\alpha^{1-\lambda_{\alpha,\beta}} \nu_\beta^{\lambda_{\alpha,\beta}} m^{1/q-1/p_{\alpha,\beta}}(n^{1/2}m^{-1/2} k^{-1/\sigma})^{\frac{1/\sigma-1/\theta_{\alpha,\beta}}{1/2-1/\sigma}} \le$$
$$\le \nu_\alpha^{1-\lambda_{\alpha,\gamma}} \nu_\gamma^{\lambda_{\alpha,\gamma}} m^{1/q-1/p_{\alpha,\gamma}}(n^{1/2}m^{-1/2} k^{-1/\sigma})^{\frac{1/\sigma-1/\theta_{\alpha,\gamma}}{1/2-1/\sigma}}.
$$
By (\ref{nua_nub_rl}), (\ref{rl4}), (\ref{nua_nub_rl1}),  (\ref{rl41}), it is equivalent to
$$
m^{-\tau\lambda_{\alpha,\beta}(1/p_\alpha-1/p_\beta)} (n^{1/2}m^{-1/q}k^{-1/2})^{-(1-\tau)\lambda_{\alpha,\beta}\frac{1/p_\alpha-1/p_\beta}{1/2-1/q}} \times$$$$\times k^{-(1-\tau) \lambda_{\alpha,\beta}(1/\theta_\alpha -1/\theta_\beta)} (n^{1/2} m^{-1/2} k^{-1/\sigma}) ^{-\tau \lambda_{\alpha,\beta} \frac{1/\theta_\alpha-1/\theta_\beta}{1/2-1/\sigma}} \times$$$$\times m^{-(1-\lambda_{\alpha,\beta})/p_\alpha -\lambda_{\alpha,\beta}/p_\beta} (n^{1/2} m^{-1/2}k^{-1/\sigma})^{-\frac{(1-\lambda _{\alpha,\beta})/\theta_\alpha +\lambda_{\alpha,\beta}/\theta_\beta}{1/2-1/\sigma}} \le
$$
$$
\le m^{-\tau'\lambda_{\alpha,\gamma}(1/p_\alpha-1/p_\gamma)} (n^{1/2}m^{-1/q}k^{-1/2})^{-(1-\tau')\lambda_{\alpha,\gamma}\frac{1/p_\alpha-1/p_\gamma}{1/2-1/q}} \times$$$$\times k^{-(1-\tau') \lambda_{\alpha,\gamma}(1/\theta_\alpha -1/\theta_\gamma)} (n^{1/2} m^{-1/2} k^{-1/\sigma}) ^{-\tau' \lambda_{\alpha,\gamma} \frac{1/\theta_\alpha-1/\theta_\gamma}{1/2-1/\sigma}} \times$$$$\times m^{-(1-\lambda_{\alpha,\gamma})/p_\alpha -\lambda_{\alpha,\gamma}/p_\gamma} (n^{1/2} m^{-1/2}k^{-1/\sigma})^{-\frac{(1-\lambda _{\alpha,\gamma})/\theta_\alpha +\lambda_{\alpha,\gamma}/\theta_\gamma}{1/2-1/\sigma}},
$$
or
$$
(n^{-1/2}m^{1/2}k^{1/2})^{(1-\tau)\lambda_{\alpha,\beta}\left(\frac{1/p_\alpha-1/p_\beta}{1/2-1/q} - \frac{1/\theta_\alpha-1/\theta_\beta}{1/2-1/\sigma}\right)} \le$$$$\le (n^{-1/2}m^{1/2}k^{1/2})^{(1-\tau')\lambda_{\alpha,\gamma}\left(\frac{1/p_\alpha-1/p_\gamma}{1/2-1/q} - \frac{1/\theta_\alpha-1/\theta_\gamma}{1/2-1/\sigma}\right)}.
$$
Taking into account (\ref{lab_q_lag}), we get (\ref{tt3}).

\subsection{The case $\Psi = \Psi_6$}

We have
\begin{align}
\label{psi_nabg}
\Psi = \nu_\alpha^{\tau_\alpha}\nu_\beta^{\tau_\beta}\nu_\gamma^{\tau_\gamma},
\end{align}
where $\alpha$, $\beta$, $\gamma\in A$ are different indices, $\tau_\alpha$, $\tau_\beta$, $\tau_\gamma >0$, $\tau_\alpha + \tau_\beta + \tau_\gamma = 1$,
$$
\frac 1q = \frac{\tau_\alpha}{p_\alpha} + \frac{\tau_\beta}{p_\beta} + \frac{\tau_\gamma}{p_\gamma}, \quad \frac{1}{\sigma} = \frac{\tau_\alpha}{\theta_\alpha} + \frac{\tau_\beta}{\theta_\beta} + \frac{\tau_\gamma}{\theta_\gamma}.
$$
Recall that the points $(1/p_\alpha, \, 1/\theta_\alpha)$, $(1/p_\beta, \, 1/\theta_\beta)$, $(1/p_\gamma, \, 1/\theta_\gamma)$ are the vertices of the triangle $\Delta_{\alpha,\beta,\gamma}\in {\cal R}$ (see (\ref{n6}), (\ref{psi6})).

We show that
\begin{align}
\label{dn_qs_est} d_n(M, \, l_{q,\sigma}^{m,k}) \underset{q,\sigma}{\gtrsim} \nu_\alpha^{\tau_\alpha}\nu_\beta^{\tau_\beta}\nu_\gamma^{\tau_\gamma}.
\end{align}

Let the numbers $r_{\alpha,\beta,\gamma}$, $l_{\alpha,\beta,\gamma}$ be defined by the equations
\begin{align}
\label{eq_syst} \frac{\nu_\alpha}{\nu_\beta} = r_{\alpha,\beta,\gamma}^{1/p_\alpha-1/p_\beta} l_{\alpha,\beta,\gamma}^{1/\theta_\alpha-1/\theta_\beta}, \quad \frac{\nu_\beta}{\nu_\gamma} = r_{\alpha,\beta,\gamma}^{1/p_\beta-1/p_\gamma} l_{\alpha,\beta,\gamma}^{1/\theta_\beta-1/\theta_\gamma}.
\end{align}
This system has the unique solution. In order to show this, it suffices to take a logarithm of the equations and notice that the vectors $(1/p_\alpha-1/p_\beta, \, 1/\theta_\alpha-1/\theta_\beta)$ and $(1/p_\beta-1/p_\gamma, \, 1/\theta_\beta-1/\theta_\gamma)$ are linearly independent, since $\Delta_{\alpha,\beta,\gamma}\in {\cal R}$.

We claim that
\begin{align}
\label{r_abg_1mk} 1\le l_{\alpha,\beta,\gamma} \le k, \quad 1\le r_{\alpha,\beta,\gamma} \le m,
\end{align}
\begin{align}
\label{n_mkr_abg_l_abg} n\le m^{\frac 2q} k^{\frac{2}{\sigma}} r_{\alpha,\beta,\gamma}^{1-\frac 2q} l_{\alpha,\beta,\gamma} ^{1-\frac{2}{\sigma}}.
\end{align}

Let $\tau_\alpha^i$, $\tau_\beta^i$, $\tau_\gamma^i\ge 0$, $\tau_\alpha^i + \tau_\beta^i + \tau_\gamma^i =1$, $i=1, \, 2$, and let
\begin{align}
\label{nu_i_p_i_t_i}
\nu_{(i)} = \nu_\alpha^{\tau_\alpha^i} \nu_\beta ^{\tau_\beta^i} \nu_\gamma^{\tau_\gamma^i}, \quad \frac{1}{p_{(i)}} = \frac{\tau_\alpha^i}{p_\alpha} + \frac{\tau_\beta^i}{p_\beta} + \frac{\tau_\gamma^i}{p_\gamma}, \quad \frac{1}{\theta_{(i)}} = \frac{\tau_\alpha^i}{\theta_\alpha} + \frac{\tau_\beta^i}{\theta_\beta} + \frac{\tau_\gamma^i}{\theta_\gamma}.
\end{align}
Then from (\ref{eq_syst}) it follows that
\begin{align}
\label{nu_1_nu_2} \frac{\nu_{(1)}}{\nu_{(2)}} = r_{\alpha,\beta,\gamma}^{1/p_{(1)}-1/p_{(2)}} l_{\alpha,\beta,\gamma}^{1/\theta_{(1)}-1/\theta_{(2)}}.
\end{align}

We check the first relation of (\ref{r_abg_1mk}) (the second can be proved similarly). Since $(1/q, \, 1/\sigma)$ is the interior point of $\Delta_{\alpha,\beta,\gamma}$, we have $$\Delta_{\alpha,\beta,\gamma}\cap [(1/q, \, 0), \, (1/q, \, 1)] = [(1/q, \, 1/\theta'), \, (1/q, \, 1/\theta'')],$$ where $\theta'>\sigma$, $\theta'' <\sigma$.

We set $\theta_{(1)} = \theta'$, $p_{(1)} = q$. Then there exist numbers $\tau^1_\alpha$, $\tau^1_\beta$, $\tau^1_\gamma\ge 0$ such that $\tau^1_\alpha+\tau^1_\beta+\tau^1_\gamma=1$ and the equation (\ref{nu_i_p_i_t_i}) holds for $p_{(1)}$, $\theta_{(1)}$; the number $\nu_{(1)}$ is also defined by (\ref{nu_i_p_i_t_i}).

We set $p_{(2)}=q$. The number $\theta_{(2)}$ is defined as follows. If $\theta''>2$ or $n \le mk^{2/\sigma}$, then we set $\theta_{(2)} = \theta''$; if $\theta''<2$ and $n > mk^{2/\sigma}$, then we set $\theta_{(2)}=2$ (in this case, $\sigma>2$). There exist the numbers $\tau^2_\alpha$, $\tau^2_\beta$, $\tau^2_\gamma\ge 0$ such that $\tau^2_\alpha+\tau^2_\beta+\tau^2_\gamma=1$ and the equality (\ref{nu_i_p_i_t_i}) holds for $p_{(2)}$, $\theta_{(2)}$; the number $\nu_{(2)}$ is also defined by (\ref{nu_i_p_i_t_i}).

In the both cases $\theta_{(2)}<\sigma < \theta_{(1)}$.
Hence there exists a number $\mu\in (0, \, 1)$ such that $\frac{1}{\sigma} = \frac{1-\mu}{\theta_{(1)}} + \frac{\mu}{\theta_{(2)}}$. Then $\nu_{\alpha}^{\tau_\alpha} \nu_\beta^{\tau_\beta} \nu_\gamma^{\tau_\gamma} = \nu_{(1)}^{1-\mu} \nu_{(2)}^{\mu}$. From (\ref{psi_nabg}) it follows that $\nu_{(1)}^{1-\mu} \nu_{(2)}^{\mu}=\Psi\le \Psi_1\le \nu_{(1)}\Phi(q, \, \theta_{(1)})$, i.e.,
$$
\nu_{(1)}^{1-\mu} \nu_{(2)}^{\mu} \le \nu_{(1)} k^{1/\sigma - 1/\theta_{(1)}}
$$
(see \eqref{phi1}); hence
$\frac{\nu_{(1)}}{\nu_{(2)}} \ge k^{1/\theta_{(1)}-1/\theta_{(2)}}$. This together with (\ref{nu_1_nu_2}) and the relations $p_{(1)} = p_{(2)} = q$, $\theta_{(1)}>\theta_{(2)}$ implies that $l_{\alpha,\beta,\gamma} \le k$.

We prove that $l_{\alpha,\beta,\gamma} \ge 1$. Let first $n \le mk^{2/\sigma}$. Then $\theta_{(2)} = \theta''$. From (\ref{psi_nabg}) it follows that $\nu_{(1)}^{1-\mu} \nu_{(2)}^{\mu}=\Psi\le \Psi_1 \le \nu_{(2)}\Phi(q, \, \theta_{(2)})$; i.e.,
$$
\nu_{(1)}^{1-\mu} \nu_{(2)}^{\mu} \le \nu_{(2)}
$$
(see \eqref{phi2}). Hence $\frac{\nu_{(1)}}{\nu_{(2)}} \le 1$. This together with (\ref{nu_1_nu_2}) and the relations $p_{(1)} = p_{(2)} = q$, $\theta_{(1)}>\theta_{(2)}$ yields that $l_{\alpha,\beta,\gamma} \ge 1$.

Let $n > mk^{2/\sigma}$. Then $\sigma>2$. If $\theta''>2$, then $\theta_{(2)} = \theta''$. From (\ref{psi_nabg}) and (\ref{phi2}) it follows that
\begin{align}
\label{nu12_th12}
\nu_{(1)}^{1-\mu} \nu_{(2)}^{\mu} \le \nu_{(2)}(n^{-1/2} m^{1/2}k^{1/\sigma})^{\frac{1/\theta_{(2)}-1/\sigma}{1/2-1/\sigma}};
\end{align}
hence $$\frac{\nu_{(1)}}{\nu_{(2)}} \le (n^{1/2}m^{-1/2} k^{-1/\sigma}) ^{\frac{1/\theta_{(1)}-1/\theta_{(2)}}{1/2-1/\sigma}}.$$ This together with (\ref{nu_1_nu_2}) and the relations $p_{(1)} = p_{(2)} = q$, $\theta_{(1)}>\theta_{(2)}$ implies that 
\begin{align}
\label{l_abg_ge}
l_{\alpha,\beta,\gamma} \ge (n^{1/2}m^{-1/2} k^{-1/\sigma})^{\frac{1}{1/2-1/\sigma}}\ge 1.
\end{align}

Let now $\theta''<2$. We show that (\ref{nu12_th12}) holds with $\theta_{(2)}=2$; this implies (\ref{l_abg_ge}). We have:
$$
\Delta_{\alpha,\beta,\gamma} \cap [(0, \, 1/2), \, (1/2, \, 1/2)] = [(1/p', \, 1/2), \, (1/p'', \, 1/2)],
$$
where $p'>q$, $p''<q$ for $q>2$, $p''=2$ for $q=2$. There exist numbers $\tilde \tau'_\alpha$, $\tilde \tau'_\beta$, $\tilde \tau'_\gamma\ge 0$, $\tilde \tau''_\alpha$, $\tilde \tau''_\beta$, $\tilde \tau''_\gamma\ge 0$ such that $\tilde \tau'_\alpha + \tilde \tau'_\beta + \tilde \tau'_\gamma =1$, $\tilde \tau''_\alpha + \tilde \tau''_\beta + \tilde \tau''_\gamma =1$, 
$$
\frac{1}{p'} = \frac{\tilde \tau'_\alpha}{p_\alpha} + \frac{\tilde \tau'_\beta}{p_\beta} + \frac{\tilde \tau'_\gamma}{p_\gamma}, \quad \frac 12 = \frac{\tilde \tau'_\alpha}{\theta_\alpha} + \frac{\tilde \tau'_\beta}{\theta_\beta} + \frac{\tilde \tau'_\gamma}{\theta_\gamma},
$$
$$
\frac{1}{p''} = \frac{\tilde \tau''_\alpha}{p_\alpha} + \frac{\tilde \tau''_\beta}{p_\beta} + \frac{\tilde \tau''_\gamma}{p_\gamma}, \quad \frac 12 = \frac{\tilde \tau''_\alpha}{\theta_\alpha} + \frac{\tilde \tau''_\beta}{\theta_\beta} + \frac{\tilde \tau''_\gamma}{\theta_\gamma}.
$$
We set $\nu' = \nu_\alpha^{\tilde \tau'_\alpha}\nu_\beta^{\tilde \tau'_\beta}\nu_\gamma^{\tilde \tau'_\gamma}$, $\nu'' = \nu_\alpha^{\tilde \tau''_\alpha}\nu_\beta^{\tilde \tau''_\beta}\nu_\gamma^{\tilde \tau''_\gamma}$. 

There exist a number $\lambda \in (0, \, 1]$ such that $\frac 1q =\frac{1-\lambda}{p'} + \frac{\lambda}{p''}$. Then $\nu_{(2)}= (\nu')^{1-\lambda}(\nu'')^\lambda$. From (\ref{psi_nabg}) it follows that $\nu_{(1)}^{1-\mu} \nu_{(2)}^\mu=\Psi\le \Psi_4 \le \nu' \Phi(p', \, 2)$. For $p''>2$ and $p''=2$ we have, respectively, $\nu_{(1)}^{1-\mu} \nu_{(2)}^\mu=\Psi\le \Psi_4 \le \nu'' \Phi(p'', \, 2)$, $\nu_{(1)}^{1-\mu} \nu_{(2)}^\mu=\Psi\le \Psi_7 \le \nu'' \Phi(p'', \, 2)$. Therefore,
$$
\nu_{(1)}^{1-\mu} \nu_{(2)}^\mu \le \nu' m^{1/q-1/p'}n^{-1/2}m^{1/2}k^{1/\sigma}, \quad \nu_{(1)}^{1-\mu} \nu_{(2)}^\mu \le \nu'' m^{1/q-1/p''}n^{-1/2}m^{1/2}k^{1/\sigma}
$$
(see (\ref{phi2}), (\ref{phi4})). This implies that
$$
\nu_{(1)}^{1-\mu} \nu_{(2)}^\mu \le (\nu')^{1-\lambda} (\nu'')^{\lambda} m^{1/q-(1-\lambda)/p' - \lambda/p''}n^{-1/2}m^{1/2}k^{1/\sigma} =
$$
$$
= \nu_{(2)}n^{-1/2}m^{1/2}k^{1/\sigma}.
$$
Since $\theta_{(2)}=2$, we get (\ref{nu12_th12}). This completes the proof of the inequality $l_{\alpha,\beta,\gamma}\ge 1$.

{\bf Remark.} We have obtained that
\begin{align}
\label{l_abg_n2m} l_{\alpha,\beta,\gamma} \ge (n^{1/2}m^{-1/2}k^{-1/\sigma})^{\frac{1}{1/2-1/\sigma}} \quad \text{for }n> mk^{2/\sigma};
\end{align}
it can be similarly proved that
\begin{align}
\label{m_abg_n2m} r_{\alpha,\beta,\gamma} \ge (n^{1/2}m^{-1/q}k^{-1/2})^{\frac{1}{1/2-1/q}} \quad \text{for }n> m^{2/q}k.
\end{align}

\vskip 0.3cm

Now we check (\ref{n_mkr_abg_l_abg}). If $n\le m^{2/q}k^{2/\sigma}$, this inequality is obvious. Now we consider the case $n> m^{2/q}k^{2/\sigma}$. Let first $q>2$, $\sigma>2$. We have
$$
\Delta_{\alpha,\beta,\gamma} \cap \left\{ \left(\frac 1p, \, \frac{1}{\theta}\right):\; \frac{1/p-1/q}{1/2-1/q} = \frac{1/\theta-1/\sigma}{1/2-1/\sigma}, \; 2\le p \le q\right\} = [(1/q, \, 1/\sigma), \, (1/p_*, \, 1/\theta_*)].
$$
Then $p_*<q$, $\theta_*<\sigma$. There exist numbers $\tau_\alpha^*$, $\tau_\beta^*$, $\tau_\gamma^*\ge 0$ such that $\tau_\alpha^* + \tau_\beta^* + \tau_\gamma^* = 1$, $\frac{1}{p_*} = \frac{\tau_\alpha^*}{p_\alpha} + \frac{\tau_\beta^*}{p_\beta} + \frac{\tau_\gamma^*}{p_\gamma}$, $\frac{1}{\theta_*} = \frac{\tau_\alpha^*}{\theta_\alpha} + \frac{\tau_\beta^*}{\theta_\beta} + \frac{\tau_\gamma^*}{\theta_\gamma}$.

From (\ref{psi_nabg}) it follows that $\nu_\alpha^{\tau_\alpha} \nu_\beta^{\tau_\beta} \nu_\gamma^{\tau_\gamma}=\Psi \le \min\{\Psi_5, \, \Psi_7\}\le \nu_\alpha^{\tau^*_\alpha} \nu_\beta^{\tau^*_\beta} \nu_\gamma^{\tau^*_\gamma} \Phi(p_*, \, \theta_*)$; i.e.,
$$
\nu_\alpha^{\tau_\alpha} \nu_\beta^{\tau_\beta} \nu_\gamma^{\tau_\gamma} \le \nu_\alpha^{\tau^*_\alpha} \nu_\beta^{\tau^*_\beta} \nu_\gamma^{\tau^*_\gamma} (n^{-1/2}m^{1/q} k^{1/\sigma}) ^{\frac{1/p_*-1/q}{1/2-1/q}};
$$
notice that $\Phi(p_*, \, \theta_*)=(n^{-1/2}m^{1/q} k^{1/\sigma}) ^{\frac{1/p_*-1/q}{1/2-1/q}}$ for $m^{2/q}k^{2/\sigma} \le n \le \frac{mk}{2}$, since $\frac{1/p_*-1/q}{1/2-1/q} = \frac{1/\theta_*-1/\sigma}{1/2-1/\sigma}$.

Applying (\ref{nu_1_nu_2}), we get
$$
r_{\alpha,\beta,\gamma}^{1/q-1/p_*} l_{\alpha,\beta,\gamma} ^{1/\sigma - 1/\theta_*} \le (n^{-1/2}m^{1/q} k^{1/\sigma}) ^{\frac{1/p_*-1/q}{1/2-1/q}}.
$$
Since $\frac{1/p_*-1/q}{1/2-1/q} = \frac{1/\theta_*-1/\sigma}{1/2-1/\sigma}$ and $p_*<q$, we have
$$
r_{\alpha,\beta,\gamma}^{1/q-1/2} l_{\alpha,\beta,\gamma} ^{1/\sigma - 1/2} \le n^{-1/2}m^{1/q} k^{1/\sigma},
$$
which implies (\ref{n_mkr_abg_l_abg}).

Let now $\sigma=2$. We prove that $$n\le m^{\frac2q}k r_{\alpha,\beta,\gamma}^{1-\frac 2q}.$$ If $n\le m^{2/q}k$, it is obvious. If $n> m^{2/q}k$, then $q>2$; the desired inequality follows from (\ref{m_abg_n2m}).

The case $q=2$ is similar; here we apply (\ref{l_abg_n2m}).

Let $r = \lceil r_{\alpha,\beta,\gamma} \rceil$, $l = \lceil l_{\alpha,\beta,\gamma} \rceil$, $W = \nu_\alpha^{\tau_\alpha} \nu_\beta^{\tau_\beta} \nu_\gamma^{\tau_\gamma} r^{-1/q}l^{-1/\sigma} V^{m,k} _{r,l}$. We show that $W\subset 4M$. Then
$$
d_n(M, \, l^{m,k}_{q,\sigma}) \gtrsim d_n(W, \, l^{m,k}_{q,\sigma}) \stackrel{(\ref{dn_vmk1}), (\ref{n_mkr_abg_l_abg})}{\underset{q,\sigma}{\gtrsim}} \nu_\alpha^{\tau_\alpha} \nu_\beta^{\tau_\beta} \nu_\gamma^{\tau_\gamma} r^{-1/q}l^{-1/\sigma} r^{1/q}l^{1/\sigma} = \nu_\alpha^{\tau_\alpha} \nu_\beta^{\tau_\beta} \nu_\gamma^{\tau_\gamma};
$$
i.e., (\ref{dn_qs_est}) holds. In order to prove the inclusion, it suffices to check that
\begin{align}
\label{abgd} \nu_\alpha^{\tau_\alpha} \nu_\beta^{\tau_\beta} \nu_\gamma^{\tau_\gamma} r_{\alpha,\beta,\gamma}^{1/p_\delta-1/q} l_{\alpha,\beta,\gamma}^{1/\theta_\delta-1/\sigma} \le \nu_\delta, \quad \delta \in A.
\end{align}
If $\delta \in \{\alpha,\, \beta, \, \gamma\}$, it follows from (\ref{nu_1_nu_2}).

Let $\delta \notin \{\alpha,\, \beta, \, \gamma\}$. Notice that $(1/q, \, 1/\sigma)$ is the interior point of the convex hull of $(1/p_\delta, \, 1/\theta_\delta)$ and one of the sides of the triangle $\Delta_{\alpha,\beta,\gamma}$ (since the points are in a general position). Without lost of generality, $(1/q, \, 1/\sigma)$ is the interior point of the convex hull of $(1/p_\alpha, \, 1/\theta_\alpha)$, $(1/p_\beta, \, 1/\theta_\beta)$ and $(1/p_\delta, \, 1/\theta_\delta)$. Then there exist the numbers $\tau_\alpha'$, $\tau_\beta'$, $\tau_\delta'> 0$ such that $\tau_\alpha' + \tau_\beta' + \tau_\delta' = 1$ and $\frac 1q = \frac{\tau_\alpha'}{p_\alpha} + \frac{\tau_\beta'}{p_\beta} + \frac{\tau_\delta'}{p_\delta}$, $\frac{1}{\sigma} = \frac{\tau_\alpha'}{\theta_\alpha} + \frac{\tau_\beta'}{\theta_\beta} + \frac{\tau_\delta'}{\theta_\delta}$. From (\ref{psi_nabg}) it follows that
\begin{align}
\label{nabg_nabd} \nu_\alpha^{\tau_\alpha} \nu_\beta^{\tau_\beta} \nu_\gamma^{\tau_\gamma} \le \nu_\alpha^{\tau'_\alpha} \nu_\beta^{\tau'_\beta} \nu_\delta^{\tau'_\delta}.
\end{align}
We define the numbers $r_{\alpha,\beta,\delta}$ and $l_{\alpha,\beta,\delta}$ by the equations
\begin{align}
\label{eq_syst1} \frac{\nu_\alpha}{\nu_\beta} = r_{\alpha,\beta,\delta}^{1/p_\alpha-1/p_\beta} l_{\alpha,\beta,\delta}^{1/\theta_\alpha-1/\theta_\beta}, \quad \frac{\nu_\beta}{\nu_\delta} = r_{\alpha,\beta,\delta}^{1/p_\beta-1/p_\delta} l_{\alpha,\beta,\delta}^{1/\theta_\beta-1/\theta_\delta}.
\end{align}
The vectors $(1/p_\alpha-1/p_\beta, \, 1/\theta_\alpha - 1/\theta_\beta)$ are $(1/p_\beta-1/p_\delta, \, 1/\theta_\beta-1/\theta_\delta)$ linearly independent, since $\Delta_{\alpha,\beta,\delta}\in {\cal R}$.
Hence the system has the unique solution.

Dividing the both sides of (\ref{nabg_nabd}) into $\nu_\alpha$, we get
$$
\left(\frac{\nu_\beta}{\nu_\alpha}\right)^{\tau_\beta} \left(\frac{\nu_\gamma}{\nu_\alpha}\right)^{\tau_\gamma} \le \left(\frac{\nu_\beta}{\nu_\alpha}\right)^{\tau'_\beta} \left(\frac{\nu_\delta}{\nu_\alpha}\right)^{\tau'_\delta}.
$$
This together with (\ref{eq_syst}) and (\ref{eq_syst1}) yield that
$$
r_{\alpha,\beta,\gamma}^{1/q-1/p_\alpha} l_{\alpha,\beta,\gamma}^{1/\sigma-1/\theta_\alpha} \le r_{\alpha,\beta,\delta}^{1/q-1/p_\alpha} l_{\alpha,\beta,\delta}^{1/\sigma-1/\theta_\alpha}.
$$
Similarly we get the inequality
$$
r_{\alpha,\beta,\gamma}^{1/q-1/p_\beta} l_{\alpha,\beta,\gamma}^{1/\sigma-1/\theta_\beta} \le r_{\alpha,\beta,\delta}^{1/q-1/p_\beta} l_{\alpha,\beta,\delta}^{1/\sigma-1/\theta_\beta}.
$$
Hence, for all $\tau \in [0, \, 1]$, we have
\begin{align}
\label{rabg_q_g} r_{\alpha,\beta,\gamma}^{1/q-(1-\tau)/p_\alpha-\tau/p_\beta} l_{\alpha,\beta,\gamma}^{1/\sigma-(1-\tau)/\theta_\alpha-\tau/\theta_\beta} \le r_{\alpha,\beta,\delta}^{1/q-(1-\tau)/p_\alpha-\tau/p_\beta} l_{\alpha,\beta,\delta}^{1/\sigma-(1-\tau)/\theta_\alpha-\tau/\theta_\beta}.
\end{align}

We choose $\tau \in (0, \, 1)$ such that the vectors $(1/q-(1-\tau)/p_\alpha-\tau/p_\beta, \, 1/\sigma-(1-\tau)/\theta_\alpha-\tau/\theta_\beta)$ and $(1/p_\delta-(1-\tau)/p_\alpha-\tau/p_\beta, \, 1/\theta_\delta-(1-\tau)/\theta_\alpha-\tau/\theta_\beta)$ are codirectional. Then from (\ref{rabg_q_g}) it follows that
$$
r_{\alpha,\beta,\gamma}^{1/p_\delta-(1-\tau)/p_\alpha-\tau/p_\beta} l_{\alpha,\beta,\gamma}^{1/\theta_\delta-(1-\tau)/\theta_\alpha-\tau/\theta_\beta} \le r_{\alpha,\beta,\delta}^{1/p_\delta-(1-\tau)/p_\alpha-\tau/p_\beta} l_{\alpha,\beta,\delta}^{1/\theta_\delta-(1-\tau)/\theta_\alpha-\tau/\theta_\beta}.
$$
Applying (\ref{eq_syst}) and (\ref{eq_syst1}), we get
$$
\left(\frac{\nu_\beta}{\nu_\alpha}\right)^{\tau_\beta (1-\tau)} \left(\frac{\nu_\gamma}{\nu_\alpha}\right)^{\tau_\gamma (1-\tau)} r_{\alpha,\beta,\gamma} ^{(1-\tau)[\tau_\beta(1/p_\alpha-1/p_\beta) + \tau_\gamma(1/p_\alpha-1/p_\gamma)]} l_{\alpha,\beta,\gamma} ^{(1-\tau)[\tau_\beta(1/\theta_\alpha-1/\theta_\beta) + \tau_\gamma(1/\theta_\alpha-1/\theta_\gamma)]}\times
$$
$$
\times \left(\frac{\nu_\alpha}{\nu_\beta}\right)^{\tau_\alpha \tau} \left(\frac{\nu_\gamma}{\nu_\beta}\right)^{\tau_\gamma \tau} r_{\alpha,\beta,\gamma} ^{\tau[\tau_\alpha(1/p_\beta-1/p_\alpha) + \tau_\gamma(1/p_\beta-1/p_\gamma)]} l_{\alpha,\beta,\gamma} ^{\tau[\tau_\alpha(1/\theta_\beta-1/\theta_\alpha) + \tau_\gamma(1/\theta_\beta-1/\theta_\gamma)]} \times
$$
$$
\times r_{\alpha,\beta,\gamma}^{1/p_\delta-(1-\tau)/p_\alpha-\tau/p_\beta} l_{\alpha,\beta,\gamma}^{1/\theta_\delta-(1-\tau)/\theta_\alpha-\tau/\theta_\beta} \le
$$
$$
\le \left(\frac{\nu_\beta}{\nu_\alpha}\right)^{\tau'_\beta (1-\tau)} \left(\frac{\nu_\delta}{\nu_\alpha}\right)^{\tau'_\delta (1-\tau)} r_{\alpha,\beta,\delta} ^{(1-\tau)[\tau'_\beta(1/p_\alpha-1/p_\beta) + \tau'_\delta(1/p_\alpha-1/p_\delta)]} l_{\alpha,\beta,\delta} ^{(1-\tau)[\tau'_\beta(1/\theta_\alpha-1/\theta_\beta) + \tau'_\delta(1/\theta_\alpha-1/\theta_\delta)]}\times
$$
$$
\times \left(\frac{\nu_\alpha}{\nu_\beta}\right)^{\tau'_\alpha \tau} \left(\frac{\nu_\delta}{\nu_\beta}\right)^{\tau'_\delta \tau} r_{\alpha,\beta,\delta} ^{\tau[\tau'_\alpha(1/p_\beta-1/p_\alpha) + \tau'_\delta(1/p_\beta-1/p_\delta)]} l_{\alpha,\beta,\delta} ^{\tau[\tau'_\alpha(1/\theta_\beta-1/\theta_\alpha) + \tau'_\delta(1/\theta_\beta-1/\theta_\delta)]} \times
$$
$$
\times r_{\alpha,\beta,\delta}^{1/p_\delta-(1-\tau)/p_\alpha-\tau/p_\beta} l_{\alpha,\beta,\delta}^{1/\theta_\delta-(1-\tau)/\theta_\alpha-\tau/\theta_\beta}.
$$
Hence,
$$
\frac{\nu_\alpha^{\tau_\alpha} \nu_\beta^{\tau_\beta} \nu_\gamma^{\tau_\gamma}}{\nu_\alpha^{1-\tau} \nu_\beta^\tau} r_{\alpha,\beta,\gamma}^{1/p_\delta-1/q} l_{\alpha,\beta,\gamma}^{1/\theta_\delta-1/\sigma} \le \frac{\nu_\alpha^{\tau'_\alpha} \nu_\beta^{\tau'_\beta} \nu_\delta^{\tau'_\delta}}{\nu_\alpha^{1-\tau} \nu_\beta^\tau} r_{\alpha,\beta,\delta}^{1/p_\delta-1/q} l_{\alpha,\beta,\delta}^{1/\theta_\delta-1/\sigma}.
$$
We apply (\ref{eq_syst1}) once again and get $\nu_\alpha^{\tau'_\alpha} \nu_\beta^{\tau'_\beta} \nu_\delta^{\tau'_\delta}r_{\alpha,\beta,\delta}^{1/p_\delta-1/q} l_{\alpha,\beta,\delta}^{1/\theta_\delta-1/\sigma}=\nu_\delta$; this implies (\ref{abgd}).

\subsection{The case $\Psi=\Psi_7$}

We may assume that $q>2$ or $\sigma>2$; otherwise, $\Psi_7=\Psi_6$.

There exist the numbers $\tau_\alpha$, $\tau_\beta$, $\tau_\gamma >0$ such that $\tau_\alpha+\tau _\beta + \tau_\gamma =1$, $\frac 12 = \frac{\tau_\alpha}{p_\alpha} + \frac{\tau_\beta}{p_\beta} +\frac{\tau_\gamma}{p_\gamma}$, $\frac 12 = \frac{\tau_\alpha}{\theta_\alpha} + \frac{\tau_\beta}{\theta_\beta} +\frac{\tau_\gamma}{\theta_\gamma}$. 

Let first $n> m^{2/q}k^{2/\sigma}$. Then 
\begin{align}
\label{psi_7}
\Psi = \nu_\alpha^{\tau_\alpha}\nu_\beta^{\tau_\beta}\nu_\gamma^{\tau_\gamma} n^{-\frac 12}m^{\frac 1q} k^{\frac{1}{\sigma}}.
\end{align}
We prove that
\begin{align}
\label{dn_qs_est_2} d_n(M, \, l_{q,\sigma}^{m,k}) \underset{q,\sigma}{\gtrsim} \nu_\alpha^{\tau_\alpha}\nu_\beta^{\tau_\beta}\nu_\gamma^{\tau_\gamma}n^{-\frac 12}m^{\frac 1q} k^{\frac{1}{\sigma}}.
\end{align}

Let the numbers $r_{\alpha,\beta,\gamma}$, $l_{\alpha,\beta,\gamma}$ be defined as the solution of the system (\ref{eq_syst}); this solution is well-defined, and (\ref{nu_1_nu_2}) holds, where $\nu_{(i)}$, $p_{(i)}$, $\theta_{(i)}$ $(i=1, \, 2)$ are defined by (\ref{nu_i_p_i_t_i}).

We claim that
\begin{align}
\label{r_abg_1mk_2} 1\le r_{\alpha,\beta,\gamma} \le m, \quad 1\le l_{\alpha,\beta,\gamma} \le k,
\end{align}
\begin{align}
\label{n_mkr_abg_l_abg_2} n\ge m^{\frac 2q} k^{\frac{2}{\sigma}} r_{\alpha,\beta,\gamma}^{1-\frac 2q} l_{\alpha,\beta,\gamma} ^{1-\frac{2}{\sigma}}.
\end{align}

We check the first inequality of (\ref{r_abg_1mk_2}); the second is similar.

Let 
$$
\Delta_{\alpha,\beta,\gamma}\cap [(0, \, 1/2), \, (1, \, 1/2)] = [(1/p', \, 1/2), \, (1/p'', \, 1/2)].
$$
Then $p'>2$, $p''<2$. We set $p_{(2)} = p''$. If $p'<q$ or $n\ge mk^{2/\sigma}$, then we set $p_{(1)} = p'$. If $p'>q$ or $n< mk^{2/\sigma}$, we set $p_{(1)} = q$. Notice that if $n<mk^{2/\sigma}$, then $q>2$; otherwise, we get the contradiction with the condition $n\ge m^{2/q}k^{2/\sigma}$. Hence $p_{(1)}>2>p_{(2)}$. We also set $\theta_{(1)}=\theta_{(2)}=2$.

There exist nonnegative numbers $\tau'_\alpha$, $\tau'_\beta$, $\tau'_\gamma$, $\tau''_\alpha$, $\tau''_\beta$, $\tau''_\gamma$ such that $\tau'_\alpha+\tau'_\beta+\tau'_\gamma=1$, $\tau''_\alpha+\tau''_\beta+\tau''_\gamma=1$, $\frac{1}{p_{(1)}} = \frac{\tau'_\alpha}{p_\alpha} + \frac{\tau'_\beta}{p_\beta} + \frac{\tau'_\gamma}{p_\gamma}$, $\frac{1}{\theta_{(1)}} = \frac{\tau'_\alpha}{\theta_\alpha} + \frac{\tau'_\beta}{\theta_\beta} + \frac{\tau'_\gamma}{\theta_\gamma}$, $\frac{1}{p_{(2)}} = \frac{\tau''_\alpha}{p_\alpha} + \frac{\tau''_\beta}{p_\beta} + \frac{\tau''_\gamma}{p_\gamma}$, $\frac{1}{\theta_{(2)}} = \frac{\tau''_\alpha}{\theta_\alpha} + \frac{\tau''_\beta}{\theta_\beta} + \frac{\tau''_\gamma}{\theta_\gamma}$. In addition, there exists a number $\mu\in (0, \, 1)$ such that $\frac 12 = \frac{1-\mu}{p_{(1)}} + \frac{\mu}{p_{(2)}}$.

We set $\nu_{(1)} = \nu_\alpha^{\tau'_\alpha}\nu_\beta^{\tau'_\beta} \nu_{\gamma}^{\tau'_\gamma}$, $\nu_{(2)} = \nu_\alpha^{\tau''_\alpha}\nu_\beta^{\tau''_\beta} \nu_{\gamma}^{\tau''_\gamma}$. Then $\nu_\alpha^{\tau_\alpha} \nu_\beta^{\tau_{\beta}} \nu_{\gamma}^{\tau_\gamma} = \nu_{(1)}^{1-\mu} \nu_{(2)}^\mu$.

From (\ref{psi_7}) and (\ref{phi6}) it follows that 
$$
\nu_{(1)}^{1-\mu} \nu_{(2)}^\mu n^{-1/2}m^{1/q} k^{1/\sigma} \le \nu_{(2)}n^{-1/2}m^{1/q} k^{1/\sigma}.
$$
Hence $\frac{\nu_{(1)}}{\nu_{(2)}}\le 1$. This together with
\begin{align}
\label{nu1nu1rp1}
\frac{\nu_{(1)}}{\nu_{(2)}} \stackrel{(\ref{nu_1_nu_2})}{=} r_{\alpha,\beta,\gamma} ^{1/p_{(1)}-1/p_{(2)}},
\end{align}
$p_{(1)}>p_{(2)}$ implies that $r_{\alpha,\beta,\gamma} \ge 1$.

Let $n\ge mk^{2/\sigma}$. Then from 
(\ref{psi_7}), (\ref{phi2}), (\ref{phi4}) it follows that 
$$
\nu_{(1)}^{1-\mu} \nu_{(2)}^\mu n^{-1/2}m^{1/q} k^{1/\sigma} \le \nu_{(1)} m^{1/q-1/p_{(1)}} n^{-1/2}m^{1/2} k^{1/\sigma}.
$$
Hence, $\frac{\nu_{(1)}}{\nu_{(2)}} \ge m^{1/p_{(1)} - 1/p_{(2)}}$; this together with (\ref{nu1nu1rp1}) and the inequality $p_{(1)}>p_{(2)}$ implies that $r_{\alpha,\beta,\gamma}\le m$.

Let $n< mk^{2/\sigma}$. If $p'< q$, then from 
(\ref{psi_7}) and (\ref{phi4}) it follows that 
$$
\nu_{(1)}^{1-\mu} \nu_{(2)}^\mu n^{-1/2}m^{1/q} k^{1/\sigma} \le \nu_{(1)} (n^{-1/2}m^{1/q} k^{1/\sigma})^{\frac{1/p_{(1)}-1/q}{1/2-1/q}}.
$$
Hence, $\frac{\nu_{(1)}}{\nu_{(2)}} \ge (n^{1/2}m^{-1/q}k^{-1/\sigma})^{\frac{1/p_{(1)} - 1/p_{(2)}}{1/2-1/q}}$; this yields that $$r_{\alpha,\beta,\gamma}\stackrel{(\ref{nu1nu1rp1})}{\le} (n^{1/2}m^{-1/q}k^{-1/\sigma})^{\frac{1}{1/2-1/q}} \le m$$ (since $n\le mk^{2/\sigma}$).

If $p'>q$, then similarly as in the case $\Psi=\Psi_6$ we can prove that 
$$
\nu_{(1)}^{1-\mu} \nu_{(2)}^\mu n^{-1/2}m^{1/q} k^{1/\sigma} \le \nu_{(1)}
$$
(here we use the equality $\Phi(q,\, \theta)\stackrel{(\ref{phi2})}{=}1$ for $1\le \theta\le \sigma$); this implies that $r_{\alpha,\beta,\gamma}\le (n^{1/2}m^{-1/q}k^{-1/\sigma})^{\frac{1}{1/2-1/q}} \le m$.

Notice that for $n< mk^{2/\sigma}$ we get the inequality
\begin{align}
\label{bbbbb} r_{\alpha,\beta,\gamma}\le (n^{1/2}m^{-1/q}k^{-1/\sigma})^{\frac{1}{1/2-1/q}}.
\end{align}
Similarly for $n< m^{2/q}k$ we have
\begin{align}
\label{bbbbb1} l_{\alpha,\beta,\gamma}\le (n^{1/2}m^{-1/q}k^{-1/\sigma})^{\frac{1}{1/2-1/\sigma}}.
\end{align}

Now we prove (\ref{n_mkr_abg_l_abg_2}). 

Let first $q>2$, $\sigma>2$. We have
$$
\Delta_{\alpha,\beta,\gamma} \cap \left\{ \left(\frac 1p, \, \frac{1}{\theta}\right):\; \frac{1/p-1/q}{1/2-1/q} = \frac{1/\theta-1/\sigma}{1/2-1/\sigma}, \; 2\le p \le q\right\} = [(1/p_*, \, 1/\theta_*), \, (1/2, \, 1/2)].
$$
Then $p_*>2$, $\theta_*>2$. There exist numbers $\tau_\alpha^*$, $\tau_\beta^*$, $\tau_\gamma^*\ge 0$ such that $\tau_\alpha^* + \tau_\beta^* + \tau_\gamma^* = 1$, $\frac{1}{p_*} = \frac{\tau_\alpha^*}{p_\alpha} + \frac{\tau_\beta^*}{p_\beta} + \frac{\tau_\gamma^*}{p_\gamma}$, $\frac{1}{\theta_*} = \frac{\tau_\alpha^*}{\theta_\alpha} + \frac{\tau_\beta^*}{\theta_\beta} + \frac{\tau_\gamma^*}{\theta_\gamma}$.

From (\ref{psi_7}) it follows that
$$
\nu_\alpha^{\tau_\alpha} \nu_\beta^{\tau_\beta} \nu_\gamma^{\tau_\gamma}n^{-1/2}m^{1/q} k^{1/\sigma} \le \min \{\Psi_5, \, \Psi_6\}\le \nu_\alpha^{\tau^*_\alpha} \nu_\beta^{\tau^*_\beta} \nu_\gamma^{\tau^*_\gamma} (n^{-1/2}m^{1/q} k^{1/\sigma}) ^{\frac{1/p_*-1/q}{1/2-1/q}}.
$$
Applying (\ref{nu_1_nu_2}), we get
$$
r_{\alpha,\beta,\gamma}^{1/2-1/p_*} l_{\alpha,\beta,\gamma} ^{1/2 - 1/\theta_*} \le (n^{1/2}m^{-1/q} k^{-1/\sigma}) ^{\frac{1/2-1/p_*}{1/2-1/q}}.
$$
Since $\frac{1/2-1/p_*}{1/2-1/q} = \frac{1/2-1/\theta_*}{1/2-1/\sigma}$ and $p_*>2$, we have
$$
r_{\alpha,\beta,\gamma}^{1/2-1/q} l_{\alpha,\beta,\gamma} ^{1/2-1/\sigma} \le n^{1/2}m^{-1/q} k^{-1/\sigma};
$$
this implies (\ref{n_mkr_abg_l_abg_2}).

Let $\sigma=2$. We claim that $n\ge m^{\frac 2q}k r_{\alpha,\beta,\gamma}^{1-\frac 2q}$. It follows from (\ref{bbbbb}) (for $\sigma=2$ the condition $n< mk^{2/\sigma}$ is obvious).

The case $q=2$ is similar; here we use (\ref{bbbbb1}).

We set $r = \lfloor r_{\alpha,\beta,\gamma}\rfloor$, $l = \lfloor l_{\alpha,\beta,\gamma} \rfloor$, $W = \nu_\alpha^{\tau_\alpha} \nu_\beta^{\tau_\beta} \nu_\gamma^{\tau_\gamma} r^{-1/2}l^{-1/2} V^{m,k} _{r,l}$. Arguing as in the case $\Psi=\Psi_6$ and replacing $(1/q, \, 1/\sigma)$ with $(1/2, \, 1/2)$, we get that $W \subset 4M$. Hence,
$$
d_n(M, \, l_{q,\sigma}^{m,k}) \gtrsim d_n(W, \, l_{q,\sigma}^{m,k}) \stackrel{(\ref{dn_vmk1}), (\ref{n_mkr_abg_l_abg_2})}{\underset{q,\sigma}{\gtrsim}} \nu_\alpha^{\tau_\alpha} \nu_\beta^{\tau_\beta} \nu_\gamma^{\tau_\gamma} r^{-1/2}l^{-1/2} n^{-1/2}m^{1/q}k^{1/\sigma} r^{1/2}l^{1/2} =$$$$= \nu_\alpha^{\tau_\alpha} \nu_\beta^{\tau_\beta} \nu_\gamma^{\tau_\gamma} n^{-1/2}m^{1/q}k^{1/\sigma};
$$
i.e., (\ref{dn_qs_est_2}) holds.

Let now $n\le m^{2/q}k^{2/\sigma}$. Then $\Phi(p, \, \theta) = 1$ for $1\le p\le q$, $1\le \theta\le \sigma$. Arguing as in the proof of  (\ref{r_abg_1mk_2}), we get that $\Psi=\Psi_3$ or $\Psi = \Psi_4$, and the estimating of the widths can be reduced to the case, which is already considered.

\section{Proof of the lower estimate: $A$ is finite, the set $\{(1/p_\alpha, \, 1/\theta_\alpha)\}_{\alpha\in A}$ is arbitrary}

Let $A$ be a fimite set. Now we consider arbitrary points $(1/p_\alpha, \, 1/\theta_\alpha)$, $\alpha\in A$. We prove that (\ref{dn_main}) holds.

\begin{Lem}
\label{lem_a1}
Let $M=\cap _{\alpha \in A} \nu_\alpha B^{m,k}_{p_\alpha,\theta_\alpha}$, $\tilde M=\cap _{\alpha \in A} \nu_\alpha B^{m,k}_{\tilde p_\alpha,\tilde\theta_\alpha}$, $|1/p_\alpha-1/\tilde p_\alpha|<\frac{\log 2}{\log (mk)}$, $|1/\theta_\alpha-1/\tilde \theta_\alpha|<\frac{\log 2}{\log (mk)}$. Then 
\begin{align}
\label{12p_tilp_incl}
\frac 12 M\subset \tilde M \subset 2M.
\end{align}
\end{Lem}
\begin{proof}
It follows from the inclusions 
$$
m^{-|1/p_\alpha-1/\tilde p_\alpha|}k^{-|1/\theta_\alpha-1/\tilde \theta_\alpha|}B^{m,k}_{p_\alpha,\theta_\alpha}\subset B^{m,k}_{\tilde p_\alpha,\tilde\theta_\alpha} \subset m^{|1/p_\alpha-1/\tilde p_\alpha|}k^{|1/\theta_\alpha-1/\tilde \theta_\alpha|}B^{m,k}_{p_\alpha,\theta_\alpha}.
$$
\end{proof}

\begin{Lem}
\label{lem_a2}
Let $\alpha \in A$, $p_\alpha^N, \, \theta_\alpha^N\in [1, \, +\infty]$, and let $(1/p_\alpha^N, \, 1/\theta_\alpha^N) \underset{N\to \infty}{\to} (1/p_\alpha, \, 1/\theta_\alpha)$. For each $\beta \ne \alpha$ we set $(p^N_\beta, \, \theta^N_\beta) = (p_\beta, \, \theta_\beta)$. We denote $P_N = \{(p^N_\beta, \,  \theta^N_\beta)\}_{\beta \in A}$. Let the values $\Psi(P)$ and $\Psi(P_N)$ be defined by \eqref{psi_def}. Then
\begin{align}
\label{psi_lim_1}
\frac{\Psi(P_N)}{\Psi(P)} \underset{N \to \infty}{\to} 1.
\end{align}
\end{Lem}
\begin{proof}
The following assertions yield (\ref{psi_lim_1}):
\begin{enumerate}
\item The function $(t, \, s) \mapsto \Phi(1/t, \, 1/s)$ is continuous on $[0, \, 1]^2$.

\item Let $(\alpha, \, \beta) \in {\cal N}_1(P)$. Then $(\alpha, \, \beta) \in {\cal N}_1(P_N)$ for large $N\in \N$. Let the numbers $\hat\lambda_{\alpha,\beta}^N$, $\hat\theta^N_{\alpha,\beta}$ be defined by the equations $\frac 1q=\frac{1-\hat\lambda_{\alpha,\beta}^N}{p_\alpha^N} +\frac{\hat\lambda_{\alpha,\beta}^N}{p_\beta}$, $\frac{1}{\hat\theta^N_{\alpha,\beta}}=\frac{1-\hat\lambda_{\alpha,\beta}^N}{\theta_\alpha^N} +\frac{\hat\lambda_{\alpha,\beta}^N}{\theta_\beta}$. Then $\hat \lambda_{\alpha,\beta}^N \underset{N\to \infty}{\to} \hat \lambda_{\alpha,\beta}$. Hence, $\nu_\alpha^{1-\hat \lambda_{\alpha,\beta}^N} \nu_\beta^{\hat \lambda_{\alpha,\beta}^N} \Phi(q, \, \hat\theta^N_{\alpha,\beta}) \underset{N\to \infty}{\to} \nu_\alpha^{1-\hat \lambda_{\alpha,\beta}} \nu_\beta^{\hat \lambda_{\alpha,\beta}} \Phi(q, \, \hat\theta_{\alpha,\beta})$. The similar assertions hold for $(\alpha, \, \beta) \in {\cal N}_j(P)$, $2\le j\le 5$.

\item Let $(\alpha, \, \beta, \, \gamma) \in {\cal N}_6(P)$. Then $(\alpha, \, \beta, \, \gamma) \in {\cal N}_6(P_N)$ for large $N\in \N$. Let the numbers $\tau_\alpha^N$, $\tau_\beta^N$, $\tau_\gamma^N$ be defined by the equations $\tau_\alpha^N+\tau_\beta^N+\tau_\gamma^N=1$, $\frac 1q=\frac{\tau_\alpha^N}{p_\alpha^N}+\frac{\tau_\beta^N}{p_\beta}+\frac{\tau_\gamma^N}{p_\gamma}$, $\frac{1}{\sigma}=\frac{\tau_\alpha^N}{\theta_\alpha^N}+\frac{\tau_\beta^N}{\theta_\beta}+\frac{\tau_\gamma^N}{\theta_\gamma}$. Then $(\tau_\alpha^N, \, \tau_\beta^N, \, \tau_\gamma^N) \underset{N \to \infty}{\to} (\tau_\alpha, \, \tau_\beta, \, \tau_\gamma)$. Hence, $\nu_\alpha^{\tau_\alpha^N}\nu_\beta^{\tau_\beta^N}\nu_\gamma^{\tau_\gamma^N}\Phi(q, \, \sigma) \underset{N\to \infty}{\to} \nu_\alpha^{\tau_\alpha}\nu_\beta^{\tau_\beta}\nu_\gamma^{\tau_\gamma}\Phi(q, \, \sigma)$. The similar assertion holds for $(\alpha, \, \beta, \, \gamma) \in {\cal N}_7(P)$.

\item Let $(\alpha, \beta)\notin {\cal N}_1(P)$, and let $(\alpha, \beta)\in {\cal N}_1(P_{N_j})$ for some subsequence $\{N_j\}_{j\in \N}$. We define the numbers $\hat \lambda^{N_j}_{\alpha,\beta}$, $\hat \theta^{N_j}_{\alpha,\beta}$ as in assertion 2. Then $p_\beta\ne q$, $p_\alpha = q$ and $\hat \lambda^{N_j}_{\alpha,\beta} \underset{j\to \infty}{\to} 0$; hence, $\nu_\alpha^{1-\hat \lambda^{N_j}_{\alpha,\beta}} \nu_\beta^{\hat \lambda^{N_j}_{\alpha,\beta}} \Phi(q, \, \hat \theta_{\alpha,\beta}^{N_j}) \underset{j \to \infty}{\to} \nu_\alpha \Phi(p_\alpha, \, \theta_\alpha)$. The similar assertions hold if $(\alpha, \beta)\notin {\cal N}_i(P)$, and $(\alpha, \beta)\in {\cal N}_i(P_{N_j})$ for some subsequence $\{N_j\}_{j\in \N}$, where $2\le i\le 4$.

\item Let $(\alpha, \beta)\notin {\cal N}_5(P)$, and let $(\alpha, \beta)\in {\cal N}_5(P_{N_j})$ for some subsequence $\{N_j\}_{j\in \N}$. We define the numbers $\lambda_{\alpha,\beta}^{N_j}$, $p_{\alpha,\beta}^{N_j}$ and $\theta_{\alpha,\beta}^{N_j}$ by the equations $\frac{1}{p_{\alpha,\beta}^{N_j}} = \frac{1-\lambda_{\alpha,\beta}^{N_j}}{p^{N_j}_\alpha} + \frac{\lambda_{\alpha,\beta}^{N_j}}{p_\beta}$, $\frac{1}{\theta_{\alpha,\beta}^{N_j}} = \frac{1-\lambda_{\alpha,\beta}^{N_j}}{\theta^{N_j}_\alpha} + \frac{\lambda_{\alpha,\beta}^{N_j}}{\theta_\beta}$, $\frac{1/p_{\alpha,\beta}^{N_j}-1/q}{1/2-1/q}=\frac{1/\theta_{\alpha,\beta}^{N_j}-1/\sigma}{1/2-1/\sigma}$. The following cases are possible:
\begin{itemize}
\item $\frac{1/p_\beta-1/q}{1/2-1/q}\ne \frac{1/\theta_\beta-1/\sigma}{1/2-1/\sigma}$, $\frac{1/p_\alpha-1/q}{1/2-1/q}= \frac{1/\theta_\alpha-1/\sigma}{1/2-1/\sigma}$ and $\lambda_{\alpha,\beta}^{N_j} \underset{j\to \infty}{\to} 0$; hence, $$\nu_\alpha^{1- \lambda^{N_j}_{\alpha,\beta}} \nu_\beta^{\lambda^{N_j}_{\alpha,\beta}} \Phi(p^{N_j}_{\alpha,\beta}, \, \theta_{\alpha,\beta}^{N_j}) \underset{j \to \infty}{\to} \nu_\alpha \Phi(p_\alpha, \, \theta_\alpha).$$
\item $\frac{1/p_\beta-1/q}{1/2-1/q}\ne \frac{1/\theta_\beta-1/\sigma}{1/2-1/\sigma}$, $\frac{1/p_\alpha-1/q}{1/2-1/q}\ne \frac{1/\theta_\alpha-1/\sigma}{1/2-1/\sigma}$, and $[(1/p_\alpha,\, 1/\theta_\alpha), \,  (1/p_\beta,\, 1/\theta_\beta)]$ intersects with $[(1/q, \, 1/\sigma), \, (1/2, \, 1/2)]$ at the point $(1/q, \, 1/\sigma)$ or $(1/2, \, 1/2)$. In this case, $\nu_\alpha^{1- \lambda^{N_j}_{\alpha,\beta}} \nu_\beta^{\lambda^{N_j}_{\alpha,\beta}} \Phi(p^{N_j}_{\alpha,\beta}, \, \theta_{\alpha,\beta}^{N_j})$ converges to $\nu_\alpha^{1-\hat \lambda_{\alpha,\beta}} \nu_\beta^{\hat \lambda_{\alpha,\beta}} \Phi(q, \, \hat\theta_{\alpha,\beta})$, $\nu_\alpha^{1-\hat \mu_{\alpha,\beta}} \nu_\beta^{\hat \mu_{\alpha,\beta}} \Phi(\hat p_{\alpha,\beta}, \, \sigma)$, $\nu_\alpha^{1-\tilde \lambda_{\alpha,\beta}} \nu_\beta^{\tilde \lambda_{\alpha,\beta}} \Phi(2, \, \tilde\theta_{\alpha,\beta})$ or $\nu_\alpha^{1-\tilde \mu_{\alpha,\beta}} \nu_\beta^{\tilde \mu_{\alpha,\beta}} \Phi(\tilde p_{\alpha,\beta}, \, 2)$.
\end{itemize}

\item Let $(\alpha, \beta, \, \gamma)\notin {\cal N}_6(P)$, and let $(\alpha, \beta, \, \gamma)\in {\cal N}_6(P_{N_j})$ for some subsequence $\{N_j\}_{j\in \N}$. Let the numbers $\tau_\alpha^{N_j}$, $\tau_\beta^{N_j}$, $\tau_\gamma^{N_j}$ be defined by the equations $\tau_\alpha^{N_j}+\tau_\beta^{N_j}+\tau_\gamma^{N_j}=1$, $\frac{1}{q} = \frac{\tau_\alpha^{N_j}}{p_\alpha^{N_j}} + \frac{\tau_\beta^{N_j}}{p_\beta} + \frac{\tau_\gamma^{N_j}}{p_\gamma}$, $\frac{1}{\sigma} = \frac{\tau_\alpha^{N_j}}{\theta_\alpha^{N_j}} + \frac{\tau_\beta^{N_j}}{\theta_\beta} + \frac{\tau_\gamma^{N_j}}{\theta_\gamma}$. Then $\Delta_{\alpha,\beta,\gamma}\in {\cal R}$, $(1/q, \, 1/\sigma)$ is the boundary point of $\Delta_{\alpha,\beta,\gamma}$; the sequence $\nu_\alpha^{\tau_\alpha^{N_j}} \nu_\beta ^{\tau_\beta^{N_j}} \nu_\gamma^{\tau_\gamma^{N_j}}$ converges to one of the following values (see \eqref{n1}, \eqref{n2}): $\nu_\alpha^{1-\hat \lambda_{\alpha,\beta}} \nu_\beta^{\hat \lambda_{\alpha,\beta}} = \nu_\alpha^{1-\hat \lambda_{\alpha,\beta}} \nu_\beta^{\hat \lambda_{\alpha,\beta}} \Phi(q, \, \sigma)$, $\nu_\alpha^{1-\hat \mu_{\alpha,\beta}} \nu_\beta^{\hat \mu_{\alpha,\beta}} = \nu_\alpha^{1-\hat \mu_{\alpha,\beta}} \nu_\beta^{\hat \mu_{\alpha,\beta}} \Phi(q, \, \sigma)$, $\nu_\alpha^{1-\hat \lambda_{\alpha,\gamma}} \nu_\gamma^{\hat \lambda_{\alpha,\gamma}} = \nu_\alpha^{1-\hat \lambda_{\alpha,\gamma}} \nu_\gamma^{\hat \lambda_{\alpha,\gamma}} \Phi(q, \, \sigma)$, $\nu_\alpha^{1-\hat \mu_{\alpha,\gamma}} \nu_\gamma^{\hat \mu_{\alpha,\gamma}} = \nu_\alpha^{1-\hat \mu_{\alpha,\gamma}} \nu_\gamma^{\hat \mu_{\alpha,\gamma}} \Phi(q, \, \sigma)$ or $\nu_\alpha = \nu_\alpha \Phi(p_\alpha, \, \theta_\alpha)$. The similar assertion holds if $(\alpha, \beta, \, \gamma)\notin {\cal N}_7(P)$, and $(\alpha, \beta, \, \gamma)\in {\cal N}_7(P_{N_j})$ for some subsequence $\{N_j\}_{j\in \N}$.
\end{enumerate}
This completes the proof.
\end{proof}

Replacing the points $(1/p_\alpha, \, 1/\theta_\alpha)$ one by one and applying Lemma \ref{lem_a2}, we ontain the set $\tilde P=\{(\tilde p_\alpha, \, \tilde \theta_\alpha)\}_{\alpha\in A}$ such that the points $\{(1/\tilde p_\alpha, \, 1/\tilde \theta_\alpha)\}_{\alpha\in A}$ are in a general position, $|1/p_\alpha-1/\tilde p_\alpha|<\frac{\log 2}{\log (mk)}$, $|1/\theta_\alpha-1/\tilde \theta_\alpha|<\frac{\log 2}{\log (mk)}$ and $\frac 12 \Psi(P)\le \Psi(\tilde P) \le 2\Psi(P)$. This together with Lemma \ref{lem_a1} and the estimates from \S 2, 3 yields that
$$
d_n(M, \, l_{q,\sigma}^{m,k}) \asymp d_n(\tilde M, \, l_{q,\sigma}^{m,k}) \underset{q,\sigma}{\asymp} \Psi(\tilde P) \asymp \Psi(P).
$$

\section{Proof of the lower estimate: $A$ is arbitrary}

It suffices to consider the case $\inf _{\alpha\in A} \nu_\alpha >0$.

Let $S=\{(1/p_\alpha, \, 1/\theta_\alpha):\; \alpha \in A\}$. There is a finite covering $\{Q_i\}_{i=1}^l$ of the set $S$ by squares with side length $\delta$; we may assume that $Q_i\cap S\ne \varnothing$ for each $i=1, \, \dots, \, l$. Given $i=1, \, \dots, \, l$, we choose a point $\alpha_i\in A$ such that $(1/p_{\alpha_i}, \, 1/\theta_{\alpha_i}) \in Q_i\cap S$ and $\nu_{\alpha_i} \le \frac 32\cdot\inf \{\nu_{\alpha}:\; \alpha \in A, \; (1/p_\alpha, \, 1/\theta_\alpha) \in Q_i\cap S\}$. Let $P' = \{(p_{\alpha_i}, \, \theta_{\alpha_i})\}_{1\le i\le l}$, $$M'= \cap _{i=1}^l \nu_{\alpha_i} B^{m,k}_{p_{\alpha_i}, \theta_{\alpha_i}}.$$
Then $M\subset M'$. We claim that $M' \subset 2M$ for sufficiently small $\delta>0$. Indeed, let $\alpha \in A$. Then $(1/p_\alpha, \, 1/\theta_\alpha) \in Q_i$ for some $i=1, \, \dots, \, l$. We have
$$
\nu_{\alpha} B^{m,k}_{p_\alpha, \, \theta_\alpha} \supset \frac{2\nu_{\alpha_i}}{3} m^{-|1/p_\alpha-1/p_{\alpha_i}|} k^{-|1/\theta_\alpha-1/\theta_{\alpha_i}|} B^{m,k}_{p_{\alpha_i},\theta_{\alpha_i}} \supset \frac 12 \nu_{\alpha_i}B^{m,k}_{p_{\alpha_i},\theta_{\alpha_i}}
$$
if $(mk)^\delta \le \frac 43$.

Now we apply the estimates for the widths from \S 2, 4, and get
$$
\Psi(P) \le \Psi(P') \underset{q,\sigma}{\lesssim} d_n(M', \, l^{m,k}_{q,\sigma}) \lesssim d_n(M, \, l^{m,k}_{q,\sigma}) \stackrel{(\ref{main_up_est})}{\underset{q,\sigma}{\lesssim}} \Psi(P).
$$
Hence, $$d_n(M, \, l^{m,k}_{q,\sigma}) \underset{q,\sigma}{\asymp} \Psi(P).$$
This completes the proof.


\begin{Biblio}
\bibitem{dir_ull} S. Dirksen, T. Ullrich, ``Gelfand numbers related to structured sparsity and Besov space embeddings with small mixed smoothness'', {\it J. Compl.}, {\bf 48} (2018), 69--102.

\bibitem{galeev2} E.M. Galeev,  ``Kolmogorov widths of classes of periodic functions of one and several variables'', {\it Math. USSR-Izv.},  {\bf 36}:2 (1991),  435--448.

\bibitem{galeev5} E.M. Galeev, ``Kolmogorov $n$-width of some finite-dimensional sets in a mixed measure'', {\it Math. Notes}, {\bf 58}:1 (1995),  774--778.


\bibitem{galeev6} E. M. Galeev, ``An estimate for the Kolmogorov widths of classes $H^r_p$ of periodic functions
of several variables of small smoothnes'', {\it Theory of functions and its applications. Proc. conf. young scientists}. 1986. P. 17--24 (in Russian).

\bibitem{galeev1} E.M.~Galeev, ``The Kolmogorov diameter of the intersection of classes of periodic
functions and of finite-dimensional sets'', {\it Math. Notes},
{\bf 29}:5 (1981), 382--388.

\bibitem{galeev4} E.M. Galeev,  ``Widths of functional classes and finite-dimensional sets'', {\it Vladikavkaz. Mat. Zh.}, {\bf 13}:2 (2011), 3--14.

\bibitem{garn_glus} A.Yu. Garnaev and E.D. Gluskin, ``On widths of the Euclidean ball'', {\it Dokl.Akad. Nauk SSSR}, {bf 277}:5 (1984), 1048--1052 [Sov. Math. Dokl. 30 (1984), 200--204]

\bibitem{gluskin1} E.D. Gluskin, ``On some finite-dimensional problems of the theory of diameters'', {\it Vestn. Leningr. Univ.}, {\bf 13}:3 (1981), 5--10 (in Russian).

\bibitem{bib_gluskin} E.D. Gluskin, ``Norms of random matrices and diameters
of finite-dimensional sets'', {\it Math. USSR-Sb.}, {\bf 48}:1
(1984), 173--182.

\bibitem{izaak1} A. D. Izaak, ``Kolmogorov widths in finite-dimensional spaces with mixed norms'', {\it Math. Notes}, {\bf 55}:1 (1994), 30--36.

\bibitem{izaak2} A. D. Izaak, ``Widths of H\"{o}lder--Nikol'skij classes and finite-dimensional subsets in spaces with mixed norm'', {\it Math. Notes}, {\bf 59}:3 (1996), 328--330.

\bibitem{kashin_oct} B.S. Kashin, ``The diameters of octahedra'', {\it Usp. Mat. Nauk} {\bf 30}:4 (1975), 251--252 (in Russian).

\bibitem{bib_kashin} B.S. Kashin, ``The widths of certain finite-dimensional
sets and classes of smooth functions'', {\it Math. USSR-Izv.},
{\bf 11}:2 (1977), 317--333.


\bibitem{k_p_s} A.N. Kolmogorov, A. A. Petrov, Yu. M. Smirnov, ``A formula of Gauss in the theory of the method of least squares'', {\it Izvestiya Akad. Nauk SSSR. Ser. Mat.} {\bf 11} (1947), 561--566 (in Russian).

\bibitem{mal_rjut} Yu. V. Malykhin, K. S. Ryutin, ``The Product of Octahedra is Badly Approximated in the $l_{2,1}$-Metric'', {\it Math. Notes}, {\bf 101}:1 (2017), 94--99.

\bibitem{pietsch1} A. Pietsch, ``$s$-numbers of operators in Banach space'', {\it Studia Math.},
{\bf 51} (1974), 201--223.

\bibitem{kniga_pinkusa} A. Pinkus, {\it $n$-widths
in approximation theory.} Berlin: Springer, 1985.

\bibitem{stech_poper} S. B. Stechkin, ``On the best approximations of given classes of functions by arbitrary polynomials'', {\it Uspekhi Mat. Nauk, 9:1(59)} (1954) 133--134 (in Russian).

\bibitem{stesin} M.I. Stesin, ``Aleksandrov diameters of finite-dimensional sets
and of classes of smooth functions'', {\it Dokl. Akad. Nauk SSSR},
{\bf 220}:6 (1975), 1278--1281 [Soviet Math. Dokl.].

\bibitem{nvtp} V.M. Tikhomirov, {\it Some questions in approximation theory}, Izdat. Moskov. Univ., Moscow, 1976.

\bibitem{itogi_nt} V.M. Tikhomirov, ``Theory of approximations''. In: {\it Current problems in
mathematics. Fundamental directions.} vol. 14. ({\it Itogi Nauki i
Tekhniki}) (Akad. Nauk SSSR, Vsesoyuz. Inst. Nauchn. i Tekhn.
Inform., Moscow, 1987), pp. 103--260 [Encycl. Math. Sci. vol. 14,
1990, pp. 93--243].


\bibitem{vas_besov} A. A. Vasil'eva, ``Kolmogorov and linear widths of the weighted Besov classes with singularity at the origin'', {\it J. Approx. Theory}, {\bf 167} (2013), 1--41.

\bibitem{vas_int_sob} A. A. Vasil'eva, ``Kolmogorov widths of an intersection of a finite family of Sobolev classes'', {\it  Izvestiya: Mathematics} (to appear).

\bibitem{vas_ball_inters} A. A. Vasil'eva, ``Kolmogorov widths of intersections of finite-dimensional balls'', {\it J. Compl.}, {\bf 72} (2022), article 101649.

\bibitem{vas_mix2} A. A. Vasil'eva, ``Estimates for the Kolmogorov widths of an intersection of two balls in a mixed norm'', arXiv:2303.13147v1.

\bibitem{vyb_06} J. Vyb\'{\i}ral, {\it Function spaces with dominating mixed smoothness}, Dissertationes Math. (Rozprawy Mat.) 436, 1--73 (2006).

\end{Biblio}
\end{document}